\documentclass[11pt]{amsart}

\usepackage[margin=3cm]{geometry}
\usepackage[T1]{fontenc}
\usepackage{lmodern}
\usepackage{microtype}
\usepackage{mathtools,amssymb,amsthm,mathrsfs}
\usepackage{comment}
\usepackage{hyperref}
\usepackage{tikz}
\numberwithin{equation}{section}
\usepackage[foot]{amsaddr}
\allowdisplaybreaks
\usepackage{float}

\theoremstyle{plain}
\newtheorem{theorem}{Theorem}[section]
\newtheorem{lemma}[theorem]{Lemma}
\newtheorem{corollary}[theorem]{Corollary}

\newtheorem{definition}{Definition}[section]
\theoremstyle{remark}
\newtheorem{remark}[theorem]{Remark}

\title[]
{Singular analysis of the principal eigenvalue for cooperative
elliptic systems with drifts}
\author[]{Shuang Liu}

\email{liushuangnqkg@bit.edu.cn
}

\thanks{{S. Liu}: School of Mathematics and Statistics, Beijing Institute of Technology, Beijing, 100081, 
China.
}


\begin{document}


\begin{abstract}
We determine the small-diffusion limit of the principal eigenvalue for
a one-dimensional 
cooperative elliptic system with
component-dependent drifts and Neumann boundary conditions. For scalar
equations, this limit is determined by local data at the drift
equilibria and boundary points. We show that such localization fails
for systems: switching between components can make some intervals
spectrally relevant. We first prove that the logarithmic transforms of all component eigenfunctions converge to
the same function, which solves a scalar Hamilton-Jacobi equation. The
associated Aubry set determines the asymptotic behavior of the principal eigenvalue. Under nondegeneracy assumptions, this set
decomposes into isolated points, regular intervals, and composite
intervals containing common drift zeros. We then identify the effective
value of each class through local Ornstein-Uhlenbeck spectra and
interval transport problems, and establish that the limit of the principal eigenvalue is
the minimum of these class values.  A unique minimizing class also
determines the common logarithmic eigenfunction profile.
\end{abstract}

\maketitle

\section{\bf Introduction}
\label{sec:introduction-main}


Let $\ell<r$, and consider the coupled elliptic eigenvalue problem
\begin{equation}\label{eq:main-system}
 \begin{cases}
 -d\varphi_{i}''-b_i(x)\varphi_{i}'
 +\displaystyle\sum_{j=1}^2M_{ij}(x)\varphi_{j}
 =\lambda\varphi_{i},&x\in(\ell,r),\quad i=1,2,\\[1mm]
 \varphi_{i}'(\ell)=\varphi_{i}'(r)=0,&i=1,2.
 \end{cases}
\end{equation}
Here $d>0$ is the diffusion rate,
$b_i\in C^1([\ell,r])$, and
$\boldsymbol M=(M_{ij})\in
C^1([\ell,r];\mathbb R^{2\times2})$. We assume
$M_{12},M_{21}<0$, so that \eqref{eq:main-system} is strongly
cooperative. By the Krein-Rutman theorem \cite{CWW2020,KR1950}, the problem
has a real and simple principal eigenvalue, denoted by  $\lambda(d)$. It is the
eigenvalue with the smallest real part and has an eigenfunction whose
two components are strictly positive.  Our aim is to study the limit of
$\lambda(d)$ as $d\searrow0$. For the sake of clarity and tractability, we restrict our attention to the one-dimensional case in this paper. The more general higher-dimensional results will be presented in subsequent work.

The principal eigenvalue of  diffusion-advection operators characterizes the delicate balance among dispersal and directed motion \cite{ALL2015}. Its asymptotic behavior in the small-diffusion regime plays a pivotal role in determining the stability, persistence, and invasion dynamics of spatially heterogeneous models \cite{CC2003,LL2022book}. Classical approaches to this singular limit,
via local asymptotic analysis and large deviations \cite{DevinatzEllisFriedman1974,EizenbergKifer1987, Friedman1973,Kifer1990,Wentzell1975}, have established some fundamental connections between the limiting eigenvalue and the underlying drift structure. 
However, a central issue remains largely unresolved: 
which parts of
the spatial domain determine the limit of the principal eigenvalue as diffusion vanishes. 

For scalar elliptic operators with  gradient drifts, Chen and Lou
\cite{ChenLou2012} obtained an explicit limit under nondegeneracy and
appropriate boundary hypotheses. The candidates are local values at
interior equilibria and the relevant boundary critical points,
determined by the zero-order coefficient and the linearized drift.
In this setting, finitely many nondegenerate equilibria determine the asymptotic behaviors of the principal eigenvalue. We also refer to recent progress on the scalar elliptic problem with divergence-free drifts \cite{BHN2005,LL2024,M2025}, and to related small-diffusion results for the general boundary conditions \cite{GuoLouZhang2025,LL2024,PengZhangZhou2019,PR2016}.

For cooperative systems without the advection term,
it was shown in  \cite{Dancer2009,LamLou2016} that the
small-diffusion limits of the principal eigenvalue is determined by the pointwise matrix eigenvalues. 
Similar observation was found in \cite{WangLamZhang2025} for the cooperative system with a common drift $b_1=b_2$ and \cite{BH2020,BHN2022,L2025} for the time-periodic parabolic systems; see also \cite{GM2026,LRW2026} for some related results. 
However, the presence of fixed, component-dependent drifts poses a qualitatively different question: If each drift
has only finitely many nondegenerate equilibria, does local information
at these equilibria and the boundary still determine the limit?
Unlike the scalar case, the coupling between components permits a spatial itinerary that may involve more than one drift. Hence, the dynamics of either component alone need not identify the relevant spectral regions. In this paper, we investigate this issue for the case of two components on an interval.


Our approach separates the identification of the relevant sets
from the calculation of their spectral values. After normalization,
the two logarithmic eigenfunction components have common
subsequential limits, each satisfying a scalar Hamilton-Jacobi
equation. Its Aubry set, determined entirely by the drifts,
identifies the candidate points and intervals. We then determine
the effective spectral problem on each static class, including
the endpoint conditions inherited from \eqref{eq:main-system}.

\subsection{The logarithmic limit and the Aubry set}


We first establish the asymptotic result for the logarithmic
eigenfunctions.

\begin{theorem}
\label{thm:main-HJ-limit}
Let
$\boldsymbol\varphi_d=(\varphi_{1,d},\varphi_{2,d})^{\mathsf T}\gg0$
be the principal eigenfunction of \eqref{eq:main-system}, normalized
by $\max_{i,x}\varphi_{i,d}(x)=1$, and set
$W_{i,d}=-d\log\varphi_{i,d}$. For every sequence
$d_k\searrow0$, there is a subsequence along which
$W_{1,d_k}$ and $W_{2,d_k}$ converge uniformly on
$[\ell,r]$ to the same Lipschitz function $W$. The limit $W$
is a viscosity solution of the Hamilton-Jacobi equation
\begin{equation}\label{main:eq:intro-HJ-Neumann}
 \begin{cases}
 \medskip
 \displaystyle\max_{1\leq i\leq 2}\left\{|W'|^2- b_i(x)W'\right\}=0,\quad x\in (\ell, r),\\
 \displaystyle
 W'(\ell)=W'(r)=0,\qquad 
 \min_{x\in[\ell,r]}W(x)=0.\end{cases}
\end{equation}
\end{theorem}

Theorem~\ref{thm:main-HJ-limit} is the two-component,
one-dimensional specialization of
Theorem~\ref{thm:common-HJ-limit} in
Section~\ref{sec:HJ-Aubry-selection}, where the result is proved
for any finite number of components in arbitrary dimension. We refer to \cite{B2013,CrandallLions1983, L1985} for the definition of the viscosity solution of \eqref{main:eq:intro-HJ-Neumann}, for which  the Hamiltonian is
$H(x,p)=\max_{i=1,2}\{p^2-b_i(x)p\}$. The strict cooperativity forces
the two logarithmic profiles to have the same limit, but
$\boldsymbol M$ itself disappears from $H$. The Hamilton-Jacobi
equation \eqref{main:eq:intro-HJ-Neumann} therefore locates the possible concentration set but does
not yet determine the limiting behaviors of the principal  eigenvalue. 

To describe the geometry from $H$, let $d_H(x,y)$ be the supremum of
$u(x)-u(y)$ over all Lipschitz viscosity subsolutions $u$ of
\eqref{main:eq:intro-HJ-Neumann}, including the relaxed boundary condition.  The Aubry set $\mathcal A_N$ consists of the points
$y$ for which $d_H(\,\cdot\,,y)$ is also a supersolution. The representation
$W(x)=\min_{y\in\mathcal A_N}\{W(y)+d_H(x,y)\}$ proved in 
\cite[Theorem~6.8]{Ishii2011} explains its relevance: once the values
of $W$ on $\mathcal A_N$ are known, the solution $W$ is determined
everywhere. Two Aubry points lie in the same static class when
$d_H(x,y)+d_H(y,x)=0$. Every critical solution is constant on each
such class.

Set $b_-=\min\{b_1,b_2\}$ and $b_+=\max\{b_1,b_2\}$. 
Theorem~\ref{thm:exact-one-dimensional-Aubry} establishes that
\begin{equation}\label{eq:main-Aubry-set}
 \begin{aligned}
 \mathcal A_N&=\{x\in[\ell,r]:b_1(x)b_2(x)\leq0\}\cup\{\ell:b_-(\ell)\leq0\}
   \cup\{r:b_+(r)\geq0\}.
 \end{aligned}
\end{equation}
Hence, the static classes are the connected components of
$\mathcal A_N$. 
Indeed, on an interval where the drifts have opposite signs,
$H(x,p)\leq0$ forces $p=0$, and the critical semidistance
vanishes in both directions. 
For the further  analysis we impose the following assumption:
\begin{itemize}
    \item[(H)] $b_i'(z)\ne0$ if  $ b_i(z)=0$, and $b_i(\ell)b_i(r)\ne0$ for $i=1,2$.
\end{itemize}

Under (H), the drifts have only finitely many zeros, and hence \eqref{eq:main-Aubry-set} yields finitely many static classes.
We denote their family by \(\mathcal C\). Each connected component
of \(\mathcal A_N\) is a point or an interval. An \emph{isolated
class} is an interior common zero where the drift derivatives
share the same sign, or a boundary singleton. An interval component
with no interior common zero has drifts of strictly opposite signs
throughout its interior and is called a \emph{regular interval
class}. Otherwise, it is called a \emph{composite interval class}.
At its interior common zeros, the drift derivatives have
opposite signs, dividing the class into transport cells. These alternatives exhaust
all connected components of \(\mathcal A_N\).
Figures~\ref{fig:quadratic-static-classes}
and~\ref{fig:moving-minimum-classes} illustrate the three types.

\subsection{Class values and the selection theorem}

We assign an effective value to each static class.
For an irreducible matrix $\boldsymbol A$ with negative
off-diagonal entries, let $\sigma(\boldsymbol A)$ denote the
eigenvalue associated with a positive eigenvector, and write
$s_+=\max\{s,0\}$. For an interior isolated class $K=\{z\}$,
set $\Lambda_K=\sigma(\boldsymbol M(z)+
\operatorname{diag}((b_1'(z))_+,(b_2'(z))_+))$.
For a boundary singleton $K=\{z\}$, set
$\Lambda_K=\sigma(\boldsymbol M(z))$.
Under (H), the boundary $\{\ell\}$ is a static class exactly when both
drifts are negative at $\ell$, and $\{r\}$ is a static class
exactly when both are positive at $r$.
Section~\ref{sec:isolated-classes} derives these values from
the corresponding local problems.

For an interval class, the effective value comes from a transport
problem with endpoint conditions determined by the drift geometry.
At an interior endpoint of a regular interval, exactly one drift vanishes.
For that component, consider the forward flow
$\dot x=b_i(x)$. The endpoint is \emph{repelling} if trajectories
starting nearby on the interval side move away from it, and
\emph{attracting} if they move towards it. An endpoint equal to $\ell$ or $r$
is a \emph{domain endpoint}. Its limiting condition retains
the effect of the Neumann boundary condition.

\begin{definition}\label{reg:def:regular-interval}
Let $I=[a,b]\in\mathcal C$ be regular, and choose
$\{p,n\}=\{1,2\}$ so that $b_p>0>b_n$ in $(a,b)$.
Its endpoint types  can be classified as follows:
\[
 \begin{aligned}
 \operatorname{type}(a)&=
 \begin{cases}
 \mathrm R,&a>\ell,\ b_p(a)=0,\\
 \mathrm A,&a>\ell,\ b_n(a)=0,\\
 \mathrm D,&a=\ell,
 \end{cases}
 &\quad
 \operatorname{type}(b)&=
 \begin{cases}
 \mathrm R,&b<r,\ b_n(b)=0,\\
 \mathrm A,&b<r,\ b_p(b)=0,\\
 \mathrm D,&b=r.
 \end{cases}
 \end{aligned}
\]
Let $X=\operatorname{type}(a)$ and $Y=\operatorname{type}(b)$, with
$X,Y\in\{\mathrm R,\mathrm A,\mathrm D\}$. The regular transport value $\Lambda_{XY}^{\mathrm{reg}}(I)$ is the unique $\lambda$ for which the transport equation
\begin{equation}\label{reg:eq:transport-system}
 \begin{cases}
 -b_pu_p'+(M_{pp}-\lambda)u_p+M_{pn}u_n=0,\\
 -b_nu_n'+M_{np}u_p+(M_{nn}-\lambda)u_n=0,
 \end{cases}
 \qquad a<x<b,
\end{equation}
has a strictly positive profile
$\boldsymbol u\in C^1_{\mathrm{loc}}((a,b);\mathbb R^2)
\cap L^1((a,b);\mathbb R^2)$ satisfying
\begin{equation}\label{reg:eq:endpoint-conditions}
 \begin{aligned}
 x=a:\quad&
 \begin{cases}
 u_n(a)=0,&X=\mathrm R,\\
 M_{np}(a)u_p(a)+(M_{nn}(a)-\lambda)u_n(a)=0,
     &X\in\{\mathrm A,\mathrm D\},
 \end{cases}\\[1mm]
 x=b:\quad&
 \begin{cases}
 u_p(b)=0,&Y=\mathrm R,\\
 (M_{pp}(b)-\lambda)u_p(b)+M_{pn}(b)u_n(b)=0,
     &Y\in\{\mathrm A,\mathrm D\}.
 \end{cases}
 \end{aligned}
\end{equation}
We can 
abbreviate $\Lambda_{XY}^{\mathrm{reg}}(I)$  as $\Lambda_I^{\mathrm{reg}}$ when the endpoint types
need not be specified. 
\end{definition}

Lemma~\ref{reg:lem:transport-value} establishes the existence and
uniqueness of \(\Lambda_I^{\mathrm{reg}}\), with a positive profile
unique up to multiplication by a positive constant. The endpoint
conditions \eqref{reg:eq:endpoint-conditions} reflect the drift
geometry: a repelling endpoint imposes zero incoming data, whereas
an attracting or domain endpoint imposes an algebraic relation. These distinctions
are visible already in Corollary~\ref{cor:quadratic-potentials}. 
The condition $b_p>0>b_n$
 in  Definition \ref{reg:def:regular-interval} implies that two
families of characteristics are directed to the right and left,
respectively. The coupling has a probabilistic interpretation:
the component index switches from \(i\) to \(j\) at rate
\(-M_{ij}\), while the spatial position follows
\(\dot x=b_i(x)\) between switches
\cite{BakhtinHurthMattingly2015,BenaimEtAl2015}.
Switching therefore allows successive portions of a trajectory to
follow different drifts and reverse direction within the interval.
This provides an interpretation of the interval contribution in
the selection formula, whose value can depend on interior
coefficients beyond the local data at drift zeros and domain
endpoints.


A common zero requires an additional local value. Although the
critical semidistance vanishes across it, a common zero cannot be
crossed by a finite concatenation of drift trajectories. A composite interval is therefore
one static class containing several transport cells. 

\begin{definition}\label{def:main-class-values}
Let $K=[a,b]\in\mathcal C$ be composite, with interior common
zeros $\mathscr Z_K=\{z_1<\cdots<z_q\}$. Set $z_0=a$,
$z_{q+1}=b$, and $G_j=(z_j,z_{j+1})$. At each common zero,
define 
$$\omega_z=\max_{0\leq\vartheta\leq1}
\sigma\left(\boldsymbol M(z)+\vartheta
\operatorname{diag}(b_1'(z),b_2'(z))\right).$$
For a cell $G=(\alpha,\beta)$, label the components so that
$b_p>0>b_n$ in $G$. Move each common-zero endpoint inward by
$\delta_\xi>0$, leaving the endpoints of $K$ unchanged.
Let $\Lambda_{G,\boldsymbol\delta}$ be the transport eigenvalue
on the resulting finite interval, with zero incoming data
$u_n(\alpha+\delta_\alpha)=0$ at an artificial left endpoint
and $u_p(\beta-\delta_\beta)=0$ at an artificial right endpoint.
At every unchanged endpoint retain
\eqref{reg:eq:endpoint-conditions} and its trace requirement. Define
\begin{equation}\label{eq:main-composite-value}
 \Lambda_G^{\mathrm{tr}}
   =\lim_{\substack{\delta_\xi\searrow0\\
         \xi\in\partial G\cap\mathscr Z_K}}
       \Lambda_{G,\boldsymbol\delta},
 \qquad
 \Lambda_K^{\mathrm{comp}}
   =\min\left\{\min_{0\leq j\leq q}\Lambda_{G_j}^{\mathrm{tr}},
               \min_{1\leq j\leq q}\omega_{z_j}\right\}.
\end{equation}
The existence of the limit in \eqref{eq:main-composite-value} is established in Lemma {\rm\ref{comp:lem:cell-exhaustion}}.
\end{definition}


Let $K$ be a composite interval. Since
$b_1(x)b_2(x)\leq0$ on $K$, the Hamiltonian satisfies $\{p\in\mathbb R:H(x,p)\leq0\}=\{0\}$ for all $x\in K$. 
Every critical subsolution is therefore constant on $K$, and hence
$d_H(x,y)=d_H(y,x)=0$ for all $x,y\in K$. Thus $K$ is one
static class in the Neumann weak KAM framework of
\cite{Ishii2011}.
This conclusion concerns the critical semidistance, not finite-time
reachability by the drift flows. Indeed, let $z$ be a common zero.
The trajectory
starting away from $z$ cannot reach or cross $z$ in finite time.
The same conclusion holds for every finite concatenation of the two
flows; see \cite[Section~6]{BakhtinHurthMattingly2015} and
\cite[Section~3.2]{BenaimEtAl2015}. Consequently, the components of
$K\setminus\mathcal Z_K$ are distinct transport cells although they
belong to the same static class. Formula
\eqref{eq:main-composite-value} accounts for both mechanisms by
comparing the cell values with the common-zero
Ornstein-Uhlenbeck values.

Our main result can be stated as follows. 
\begin{theorem}\label{thm:main-one-dimensional}
Suppose that 
{\rm (H)} holds.
Let $\mathcal Z_{\mathrm{iso}}$ be the isolated points of
$\mathcal A_N$, and let $\mathcal C_{\mathrm{reg}}$ and
$\mathcal C_{\mathrm{comp}}$ be its families of regular and
composite interval classes. These sets are finite, and
$$
 \mathcal A_N=\mathcal Z_{\mathrm{iso}}
 \sqcup\Big(\bigsqcup_{I\in\mathcal C_{\mathrm{reg}}}I\Big)
 \sqcup\Big(\bigsqcup_{K\in\mathcal C_{\mathrm{comp}}}K\Big).
$$
Then there holds
$$
 \begin{aligned}
 \lim_{d\searrow0}\lambda(d)=\min\Bigg\{&
 \min_{z\in\mathcal Z_{\mathrm{iso}}\cap(\ell,r)}
 \sigma\!\left(\boldsymbol M(z)
  +\operatorname{diag}((b_1'(z))_+,(b_2'(z))_+)\right),\\
 &\min_{z\in\mathcal Z_{\mathrm{iso}}\cap\{\ell,r\}}
       \sigma(\boldsymbol M(z)),\quad
 \min_{I\in\mathcal C_{\mathrm{reg}}}\Lambda_I^{\mathrm{reg}},\quad
 \min_{K\in\mathcal C_{\mathrm{comp}}}\Lambda_K^{\mathrm{comp}}
 \Bigg\},
 \end{aligned}
$$
where $\Lambda_I^{\mathrm{reg}}$ and
$\Lambda_K^{\mathrm{comp}}$ are defined in
Definitions {\rm\ref{reg:def:regular-interval}}
and {\rm \ref{def:main-class-values}}, respectively. If exactly one static
class $K_*$ realizes the minimum, then
$-d\log\varphi_{i,d}\to d_H(\,\cdot\,,K_*)$ uniformly on
$[\ell,r]$, $i=1,2$, where
$d_H(x,K_*)=\min_{y\in K_*}d_H(x,y)$.
\end{theorem}

Theorem~\ref{thm:main-one-dimensional} provides a complete characterization of the small-diffusion limit of the principal eigenvalue under 
{\rm (H)}. 
Even when both drifts have only finitely many nondegenerate
zeros, the eigenvalue limit is not necessarily determined by pointwise
spectral values alone. An interval can contribute a transport value
that depends on its interior coefficients.
This contribution is absent from the local selection formula
for the scalar gradient-drift problem \cite{ChenLou2012}.

\subsection{Examples and discussion}

In this section, we consider the example, where the drifts are given by the
quadratic potentials, to illustrate Theorem \ref{thm:main-one-dimensional}. 

\begin{corollary}\label{cor:quadratic-potentials}
Let $b_i=m_i'$, where
$$
 m_1(x)=k(x-x_1)^2,\qquad
 m_2(x)=h(x-x_2)^2,
 \qquad kh\ne0.
$$

If $\ell<x_1<x_2<r$, the following alternatives hold.

\par\smallskip
\noindent\emph{(i)}
If $k,h>0$, the classes are $\{\ell\}$, the
repelling-repelling interval $[x_1,x_2]$, and $\{r\}$, with
$$
 \lim_{d\searrow0}\lambda(d)
 =\min\left\{
   \sigma(\boldsymbol M(\ell)),\,\,
   \Lambda_{\mathrm{RR}}^{\mathrm{reg}}([x_1,x_2]),\,\,
   \sigma(\boldsymbol M(r))
 \right\},
$$

\par\smallskip
\noindent\emph{(ii)}
If $k,h<0$, the attracting-attracting interval $[x_1,x_2]$
is the only class, and
$$
 \lim_{d\searrow0}\lambda(d)
 =\Lambda_{\mathrm{AA}}^{\mathrm{reg}}([x_1,x_2]).
$$

\par\smallskip
\noindent\emph{(iii)}
If $kh<0$, the classes are $[\ell,x_1]$ and $[x_2,r]$, with
$$
 \lim_{d\searrow0}\lambda(d)
 =\begin{cases}
 \min\bigl\{
   \Lambda_{\mathrm{DR}}^{\mathrm{reg}}([\ell,x_1]),
   \,\,\Lambda_{\mathrm{AD}}^{\mathrm{reg}}([x_2,r])
 \bigr\},&k>0>h,\\[1mm]
 \min\bigl\{
   \Lambda_{\mathrm{DA}}^{\mathrm{reg}}([\ell,x_1]),
   \,\, \Lambda_{\mathrm{RD}}^{\mathrm{reg}}([x_2,r])
 \bigr\},&h>0>k.
 \end{cases}
$$

Furthermore, if $x_1=x_2=z\in(\ell,r)$, then
\[
 \lim_{d\searrow0}\lambda(d)=
 \begin{cases}
 \min\{\sigma(\boldsymbol M(\ell)),
       \,\, \sigma(\boldsymbol M(z)+\operatorname{diag}(2k,2h)),
       \,\,\sigma(\boldsymbol M(r))\},&k,h>0,\\[1mm]
 \sigma(\boldsymbol M(z)),&k,h<0,\\[1mm]
 \min\{\Lambda_{(\ell,z)}^{\mathrm{tr}},\,\,
       \Lambda_{(z,r)}^{\mathrm{tr}}, \,\, \max_{0\leq\vartheta\leq1}
\sigma(\boldsymbol M(z)+\vartheta\operatorname{diag}(2k,2h))\},&kh<0.
 \end{cases}
\]
\end{corollary}


\begin{figure}[H]
\centering
\includegraphics[width=0.85\linewidth]
{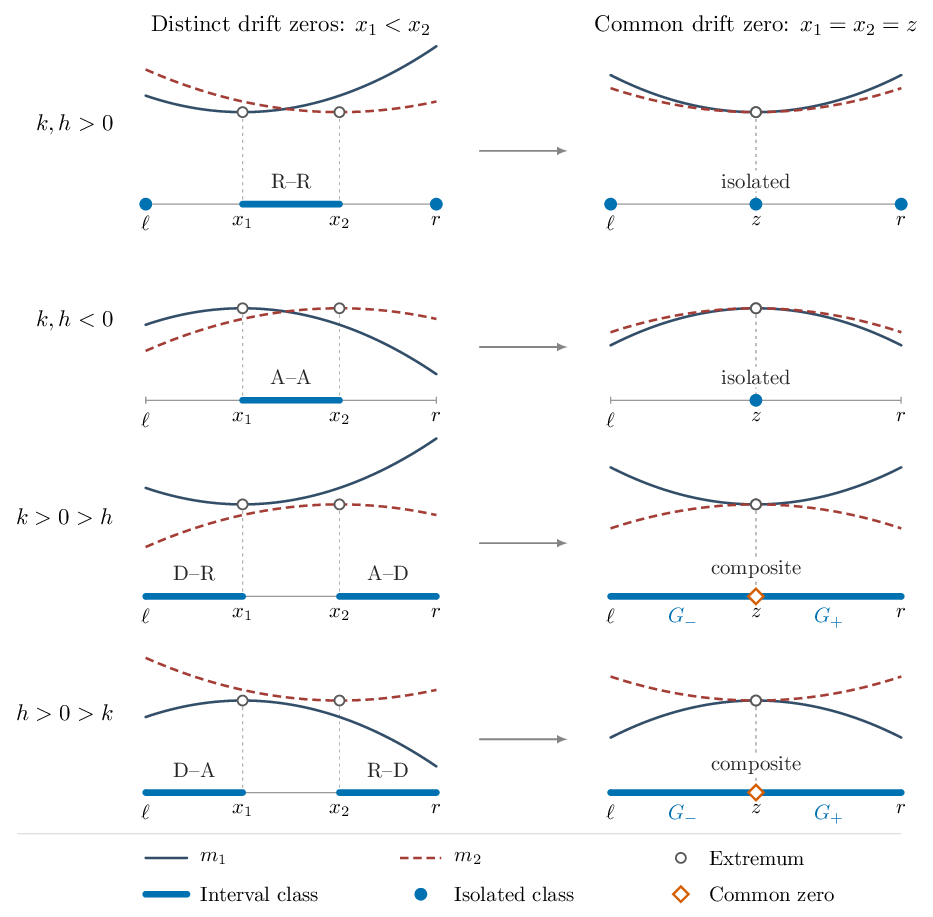}
\caption{\small Illustration for Corollary \ref{cor:quadratic-potentials}. Thin solid and dashed curves represent functions $m_1$ and $m_2$, with schematic vertical scales. Blue segments and isolated dots represent interval and singleton classes. A dot inside a composite interval marks its common zero, separating the cells $G_-$ and $G_+$.}
\label{fig:quadratic-static-classes}
\end{figure}

Corollary~\ref{cor:quadratic-potentials} reveals that coupling can
yield interval contributions in the small-diffusion limit of
the principal eigenvalue. To illustrate this effect, we consider
case~(i), where \(k,h>0\) and \(x_1<x_2\). Take
\[
 \boldsymbol M_\mu(x)
 =\operatorname{diag}(c_1(x),c_2(x))+\mu\boldsymbol L,
 \qquad \mu\geq0,
\]
where \(c_1,c_2\) are fixed zero-order coefficients and
\(\boldsymbol L=(L_{ij})\) is a fixed matrix satisfying
\(L_{12},L_{21}<0\) and \(\sum_{j=1}^2L_{ij}=0\) for \(i=1,2\).
The matrix \(-\mu\boldsymbol L\) is the backward generator of
switching between two phenotypic states, with rate
\(-\mu L_{ij}\) from state \(i\) to state \(j\).
At \(\mu=0\), the system
decouples, and by the results in
\cite{ChenLou2012}, as 
\(d\searrow0\), the principal eigenvalue converges to the minimum of
the boundary values \(c_i(\ell),c_i(r)\), \(i=1,2\), and the interior
values \(c_1(x_1)+2k\) and \(c_2(x_2)+2h\).  Hence, only local data at
the drift zeros and domain endpoints contribute to the limit.
For every fixed \(\mu>0\),
Corollary~\ref{cor:quadratic-potentials} replaces the two
interior point values in the selection formula by
\(\Lambda_{\mathrm{RR}}^{\mathrm{reg}}([x_1,x_2])\),
computed with \(\boldsymbol M_\mu\) in
Definition~\ref{reg:def:regular-interval}.
Since \(b_1>0>b_2\) on \((x_1,x_2)\), switching reverses
the direction of transport within the interval.
The transport value may depend on the coefficients throughout
\((x_1,x_2)\), beyond their local behavior near the two drift zeros.
The eigenvalue limit is the minimum of this interval value and
the two boundary values \(\sigma(\boldsymbol M_\mu(\ell))\)
and \(\sigma(\boldsymbol M_\mu(r))\).
A natural question is whether different choices of
\(\mu=\mu(d)\searrow0\) lead to different small-diffusion limits.
This joint limit is left for future work.

The endpoint types remain visible as a regular interval shrinks.
Fix $z\in(\ell,r)$ and set $x_1=z-\delta/2$ and
$x_2=z+\delta/2$, with $k,h$ and
$\boldsymbol M$ fixed.
When $k$ and $h$ share the same sign,
\begin{equation}\label{eq:intro-same-sign-coalescence}
\begin{aligned}
 \lim_{\delta\searrow0}
 \Lambda_{\mathrm{RR}}^{\mathrm{reg}}
   ([z-\delta/2,z+\delta/2])
 &=\sigma\!\left(
   \boldsymbol M(z)+\operatorname{diag}(2k,2h)\right),
 &&k,h>0,\\
 \lim_{\delta\searrow0}
 \Lambda_{\mathrm{AA}}^{\mathrm{reg}}
   ([z-\delta/2,z+\delta/2])
 &=\sigma(\boldsymbol M(z)),
 &&k,h<0.
\end{aligned}
\end{equation}
Indeed, the rescaling $x=x_1+\delta s$ reduces
both limits to the corresponding frozen transport problems on $(0,1)$.
Hence, the
repelling-repelling interval retains the drift-derivative
contributions, whereas the attracting-attracting interval does not.
Both limits coincide with the effective value of the isolated class
obtained at $x_1=x_2=z$.

The opposite-sign collision alters the class
structure: as the gap closes, two regular classes merge into a
composite class.
Corollary~\ref{cor:quadratic-potentials} determines the effective
values for each fixed configuration.
It is natural to expect that the minimum of the two regular-class values converges to
the composite-class value as the two zeros coalesce. {Establishing} this
continuity would require a separate and more delicate
analysis, {which we do not pursue here}. In this limit, the drifts vary and
the repelling or attracting endpoint conditions themselves degenerate.

\begin{figure}[H]
\centering
\includegraphics[width=0.85\linewidth]
{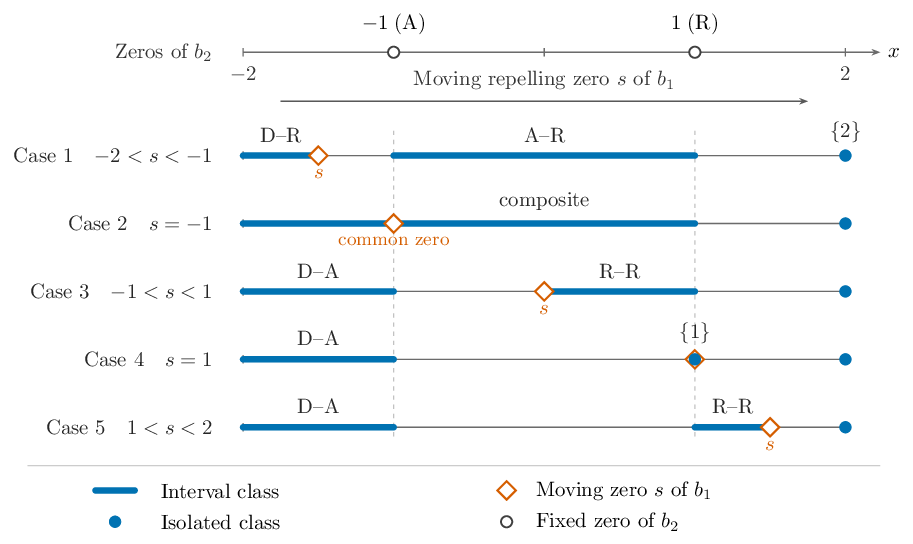}
\caption{\small Evolution of the static classes as the repelling zero
$s$ of $b_1$ moves past the attracting zero $-1$ and the
repelling zero $1$ of $b_2$. Blue segments and filled dots denote
interval and isolated classes, respectively. At $s=-1$, two regular
classes form a composite class with a common zero. At $s=1$, the
$\mathrm R$--$\mathrm R$ interval contracts to the isolated class
$\{1\}$ and then reappears to its right.}
\label{fig:moving-minimum-classes}
\end{figure}

Both geometric transitions can be realized in a single one-parameter family without losing nondegeneracy. For example, on $[-2,2]$, let
\begin{equation}\label{eq:intro-moving-minimum}
 m_1(x;s)=(x-s)^2/2,
 \qquad m_2(x)=x^3/3-x,
 \qquad -2<s<2.
\end{equation}
The first potential possesses a single minimum
at $s$, whereas the second has a maximum at $-1$ and a
minimum at $1$. Consequently,
the zero $s$ of $b_1$ is repelling, whereas the zeros $-1$ and
$1$ of $b_2$ are attracting and repelling, respectively. Since
$b_1(x;s)=x-s$ and $b_2(x)=x^2-1$, the sign of
$(x-s)(x^2-1)$ directly determines the interval classes. The
right endpoint $\{2\}$ remains an additional singleton for all
$s$. Figure \ref{fig:moving-minimum-classes} tracks the three drift
zeros as $s$ increases and displays both transitions at a glance.

We first consider the case $s<-1$. The two interval classes are
$[-2,s]$ of type $\mathrm D$-$\mathrm R$ and
$[-1,1]$ of type $\mathrm A$-$\mathrm R$. At $s=-1$,
they merge into the composite class $[-2,1]$, whose common zero is
$-1$. Its value is the minimum of the two cell values on
$(-2,-1)$ and $(-1,1)$, together with
$\omega_{-1}=\max_{0\leq\vartheta\leq1}
\sigma(\boldsymbol M(-1)+\vartheta\operatorname{diag}(1,-2))$.
Once $s$ {exceeds} $-1$, the classes become $[-2,-1]$ and
$[s,1]$, now of types $\mathrm D$-$\mathrm A$ and
$\mathrm R$-$\mathrm R$. Thus, the passage through $s=-1$
is a regular- to-composite-to-regular transition, with both endpoint
types changing across the collision.

A different phenomenon occurs at $s=1$. As $s\nearrow1$, the
repelling-repelling interval $[s,1]$ contracts to the isolated
class $\{1\}$, whose value is
$\sigma(\boldsymbol M(1)+\operatorname{diag}(1,2))$. For
$1<s<2$, the interval reopens as $[1,s]$, while
$[-2,-1]$ and $\{2\}$ remain unchanged. This is the same
geometric transition as in the collision of two quadratic minima. Hence,
moving a single minimum realizes both the regular-to-composite and
regular-to-isolated mechanisms. Replacing $m_1$ by
$-\tfrac12(x-s)^2$ yields the attracting-attracting counterpart:
near $s=-1$, the interval between the two maxima contracts to the
isolated point $-1$ and then reopens on the other side.

In Section~\ref{sec:HJ-Aubry-selection}, we derive
the common logarithmic limit, characterize the Aubry geometry, and
establish the finite-class selection principle.
Section~\ref{sec:isolated-classes}
identifies the effective values of isolated classes.
Sections~\ref{sec:regular-intervals}
and~\ref{sec:composite-intervals} determine the values of regular and
composite interval classes, respectively.
Together these results establish
Theorem~\ref{thm:main-one-dimensional}.

\section{\bf Hamilton-Jacobi limit and Aubry geometry}
\label{sec:HJ-Aubry-selection}

\subsection{Logarithmic limit of the principal eigenfunction}

In this section, we investigate the asymptotic behavior of the principal
eigenfunction of problem \eqref{eq:main-system} in the general case.
Let \(\Omega\subset\mathbb R^N\) be a bounded \(C^2\) domain with unit
normal vector \(\nu(x)\) on \(\partial\Omega\). Let \(m\geq1\) and
\(\boldsymbol b_i\) (\(1\leq i\leq m\)) be \(C^1\) vector fields.
Assume that \(\boldsymbol M\in C^1(\overline{\Omega};\mathbb R^{m\times m})\) with
\(M_{ij}<0\) for \(i\neq j\).
Consider the following eigenvalue problem
\begin{equation}\label{liu-20260913-1}
 \begin{cases}
 -d\Delta\varphi_{i}-\boldsymbol b_i(x)\cdot\nabla\varphi_{i}
 +\displaystyle\sum_{j=1}^mM_{ij}(x)\varphi_{j}
 =\lambda\varphi_{i},&x\in\Omega,\quad1\leq i\leq m,\\
 \partial_\nu\varphi_{i}=0,&x\in\partial\Omega.
 \end{cases}
\end{equation}

\begin{theorem}\label{thm:common-HJ-limit}
Let $\boldsymbol\varphi_d=(\varphi_{1,d},\cdots,\varphi_{m,d})\gg 0$ be the principal eigenfunction of \eqref{liu-20260913-1}, which is normalize by
$\max_{i,x}\varphi_{i,d}(x)=1$, and set
$W_{i,d}\coloneqq -d\log\varphi_{i,d}$.  
For every sequence $d_k\searrow0$, there are a subsequence, still
denoted by $d_k$, and a function
$W\in W^{1,\infty}(\Omega)\cap C(\overline{\Omega})$ such
that $W_{i,d_k}\to W$ uniformly on $\overline{\Omega}$ for
each $i$.  The common limit is a viscosity solution of
\begin{equation}\label{main:eq:HJ-Neumann}
 \begin{cases}
 \medskip
 \displaystyle\max_{1\leq i\leq m}\left\{|\nabla W|^2- \boldsymbol b_i(x)\cdot \nabla W\right\}=0,\quad x\in \Omega,\\
 \displaystyle
 \partial_\nu W=0, \,\, x\in \partial \Omega, \qquad 
 \min_{x\in\overline\Omega}W(x)=0.\end{cases}
\end{equation}
\end{theorem}
\begin{proof}

Set
\(H_i(x,p)\coloneqq |p|^2-\boldsymbol b_i(x)\cdot p\) and
\(H(x,p)\coloneqq \max_{1\leq i\leq m}H_i(x,p)\). A direct calculation shows that \(W_{i,d}\) satisfies
\begin{equation}\label{eq:logarithmic-system}
 -d\Delta W_{i,d}+H_i(x,\nabla W_{i,d})
 -d\sum_{j\ne i}M_{ij}
 \left(\exp\left[\frac{W_{i,d}-W_{j,d}}d\right]-1\right)-d\sum_{j=1}^mM_{ij}
 =-d\lambda(d).
\end{equation}


\noindent\emph{Step 1.}
We first establish the uniform convergence of \(W_{i,d}\). 
By the maximum principle, we show that \(\lambda(d)\) is uniformly bounded.
We then compare the two components. Suppose that
\(B_{2d}(x_0)\subset\Omega\), and set
\(u_i(y)=\varphi_{i,d}(x_0+dy)\). Then
$$
 -\Delta u_i-\boldsymbol b_i(x_0+dy)\cdot\nabla u_i
 +d\bigl(M_{ii}(x_0+dy)-\lambda(d)\bigr)u_i
 =-d\sum_{j\ne i}M_{ij}(x_0+dy)u_j.
$$
Fix \(p\in(0,1)\) 
and define
\(U=\max_{1\leq k\leq m}u_k\). The uniform coefficient bounds imply
that \(U\) is a nonnegative viscosity subsolution of
$$
 -\Delta U-C|\nabla U|\leq CdU
 \quad\text{in }B_2.
$$ 
Since \(U^p\leq\sum_k u_k^p\), 
we apply the 
scalar weak Harnack inequality for \(u_i\) to obtain
\begin{equation}\label{liu-20260914-1}
    \max_{1\leq j\leq m}\sup_{B_1}u_j
 \leq C\sum_{k=1}^m\|u_k\|_{L^p(B_{3/2})},
 \qquad
 \|u_i\|_{L^p(B_{3/2})}
 \leq C\inf_{B_{1/2}}u_i.
\end{equation}
For each fixed \(k\ne i\), the uniform strict negativity of the
off-diagonal entries yields
$$
 \begin{aligned}
 &-\Delta u_i-\boldsymbol b_i(x_0+dy)\cdot\nabla u_i
 +d\bigl(M_{ii}(x_0+dy)-\lambda(d)\bigr)u_i\geq cd\,u_k.
 \end{aligned}
$$
The zero-order coefficient of the scalar operator on the left is
\(O(d)\). Hence, for all sufficiently small \(d\), this operator
satisfies the maximum principle on \(B_2\), and its Dirichlet Green
kernel has a uniform positive lower bound on
\(B_{1/2}\times B_{3/2}\). Green-kernel comparison therefore yields
$$
 \inf_{B_{1/2}}u_i
 \geq cd\int_{B_{3/2}}u_k(z)\,\mathrm{d}z
 \geq cd\|u_k\|_{L^p(B_{3/2})},
 \qquad k\ne i,
$$
where the second inequality follows from \(p<1\). Combining this with \eqref{liu-20260914-1} yields {that}
$$
 d\sum_{k=1}^m\|u_k\|_{L^p(B_{3/2})}
 \leq C\inf_{B_{1/2}}u_i, \quad 1\leq i\leq m.
$$
Applying \eqref{liu-20260914-1} again, we
obtain that
\begin{equation}\label{eq:quantitative-component-Harnack}
 \max_{1\leq j\leq m}\sup_{\Omega\cap B_d(x_0)}\varphi_{j,d}
 \leq\frac Cd
 \min_{1\leq i\leq m}
 \inf_{\Omega\cap B_{d/2}(x_0)}\varphi_{i,d}.
\end{equation}
If \(B_{2d}(x_0)\) meets \(\partial\Omega\), \eqref{eq:quantitative-component-Harnack} {follows} by flattening the boundary and 
reducing the argument to a half-ball. 
{For any}
\(x\in\overline\Omega\) and \(1\leq i,j\leq m\), we {deduce} that
$$
 \begin{aligned}
 \varphi_{i,d}(x)
 \leq
 \max_k\sup_{\Omega\cap B_d(x)}\varphi_{k,d}\leq\frac Cd
 \min_k\inf_{\Omega\cap B_{d/2}(x)}\varphi_{k,d}
 \leq\frac Cd\,\varphi_{j,d}(x).
 \end{aligned}
$$
Since \(i\) and \(j\) are arbitrary, we obtain that
\begin{equation}\label{eq:component-phase-gap}
 \frac dC\leq
 \frac{\varphi_{i,d}(x)}{\varphi_{j,d}(x)}
 \leq\frac Cd,
 \qquad
 \|W_{i,d}-W_{j,d}\|_{L^\infty(\Omega)}
 \leq d\log(C/d)=o(1).
\end{equation}

We next derive a uniform gradient bound {using} the classical Bernstein
method \cite{EvansIshii1985}. We present only the
modifications required by the Neumann boundary condition and the
coupling terms. The following calculation is justified by standard difference quotients, where the constants depend only on the prescribed
\(C^1\)-coefficient bounds and the \(C^2\) geometry of
\(\partial\Omega\).
Choose \(\rho\in C^2(\overline\Omega)\) such that
\(\partial_\nu\rho=1\) on \(\partial\Omega\), and set
\(
 F_i:=\mathrm e^{-A\rho}|\nabla W_{i,d}|^2.
\)
The Neumann condition and the boundedness of the second fundamental
form imply that
$$
 \partial_\nu|\nabla W_{i,d}|^2
 \leq C_{\rm b}|\nabla W_{i,d}|^2
 \quad\text{on }\partial\Omega.
$$
Taking \(A>C_{\rm b}\), we have
\(\partial_\nu F_i<0\) wherever \(F_i>0\) on \(\partial\Omega\).
Since the outward derivative is nonnegative at a boundary maximum,
{any} positive maximum of \(F_i\) {over} all components and
\(\overline\Omega\) {must be attained} in \(\Omega\).
Let \((i_*,x_*)\) be a maximizing pair and set
\(P=|\nabla W_{i_*,d}(x_*)|\). We may assume {\(P>0\)}.
At \(x_*\), the maximality implies that 
$$
 |\nabla W_{j,d}|\leq P,
 \qquad
 \nabla W_{i_*,d}\cdot
 (\nabla W_{i_*,d}-\nabla W_{j,d})
 \geq P\bigl(P-|\nabla W_{j,d}|\bigr)\geq0.
$$
The weighted maximum identities also yield that at \(x_*\), 
$$
 \begin{aligned}
 \nabla|\nabla W_{i_*,d}|^2
 =AP^2\nabla\rho,\quad
 \Delta|\nabla W_{i_*,d}|^2
 \leq CP^2,\quad 
 2D^2W_{i_*,d}\nabla W_{i_*,d}
 =AP^2\nabla\rho.
 \end{aligned}
$$

Differentiating \eqref{eq:logarithmic-system} and taking the scalar product
with \(\nabla W_{i_*,d}\), we {deduce} from \eqref{eq:component-phase-gap} that 
$$
 \begin{aligned}
 d|D^2W_{i_*,d}|^2
 &-\sum_{j\ne i_*}M_{i_*j}
  \exp\!\left[\frac{W_{i_*,d}-W_{j,d}}d\right]\\
 &\qquad{}\times
  \nabla W_{i_*,d}\cdot
  (\nabla W_{i_*,d}-\nabla W_{j,d})
 \leq C(1+P^3) \quad \text{at } x=x_*.
 \end{aligned}
$$
Hence, 
\(
 d|D^2W_{i_*,d}(x_*)|^2\leq C(1+P^3).
\)
Combining the two matrix sums in
\eqref{eq:logarithmic-system}, we find  
$$
 P^2
 =d\Delta W_{i_*,d}
  +\boldsymbol b_{i_*}\cdot\nabla W_{i_*,d}
  +dM_{i_*i_*}
  +d\sum_{j\ne i_*}M_{i_*j}
   \exp\!\left[\frac{W_{i_*,d}-W_{j,d}}d\right]
  -d\lambda(d).
$$
The off-diagonal sum is nonpositive, and thus
\(
 P^2\leq d|\Delta W_{i_*,d}(x_*)|+CP+C.
\) 
Hence, 
$$
 P^4
 \leq Cd^2|D^2W_{i_*,d}(x_*)|^2
 \leq Cd(1+P^3),
$$
which implies \(P\leq C\). Since \(\rho\) is bounded and
\((i_*,x_*)\) maximizes the weighted gradients,
$$
 \max_{1\leq i\leq m}
 \|\nabla W_{i,d}\|_{L^\infty(\Omega)}\leq C.
$$

It remains to establish the boundedness of \(W_{i,d}\). By the normalization, there exist
\(i_d\) and \(x_d\in\overline\Omega\) such that
\(W_{i_d,d}(x_d)=0\) {and} \(W_{i,d}\geq0\) in
\(\overline\Omega\). It follows from \eqref{eq:component-phase-gap} that
\[
 0\leq W_{i,d}(x_d)\leq d\log(C/d).
\] 
Then for every \(x\in\overline\Omega\), we have
$$
 \begin{aligned}
 W_{i,d}(x)
 \leq
 |W_{i,d}(x)-W_{i,d}(x_d)|+W_{i,d}(x_d)\leq
 C_\Omega\|\nabla W_{i,d}\|_{L^\infty(\Omega)}
 +d\log(C/d)
 \leq C.
 \end{aligned}
$$
Thus \(0\leq W_{i,d}\leq C\) uniformly in \(i\) and \(d\).
Therefore, by the Arzel\`a--Ascoli theorem and \eqref{eq:component-phase-gap}, for every sequence \(d_k\searrow0\), there exists 
a subsequence along which all components converge uniformly to the same function \(W\). The
normalization ensures that
\(\min_{\overline\Omega}W=0\).

\medskip
\noindent\emph{Step 2.}
We identify the limiting function $W$ as a viscosity solution of \eqref{main:eq:HJ-Neumann}. 
Fix a sequence $d_k\to0$ along which
$W_{i,d_k}\to W$ uniformly on $\overline{\Omega}$ for every
$1\leq i\leq m$.  We use the half-relaxed limits to verify the two
viscosity inequalities and their boundary alternatives separately.
Define
\begin{equation}\label{eq:lower-half-relaxed-limit}
 \underline U_i(x)\coloneqq
 \liminf_{k\to\infty,\,x'\to x}W_{i,d_k}(x'),
 \qquad
 U_*(x)\coloneqq \min_{1\leq i\leq m}\underline U_i(x).
\end{equation}
By Step 1, the functions
$\underline U_i$ are well defined, and $U_*$ is lower
semicontinuous.  Uniform convergence identifies each of them with
$W$, so $\underline U_i=U_*=W$.  The formulation
\eqref{eq:lower-half-relaxed-limit} is retained to select a component
at each contact point.

Let $\phi\in C^2(\overline{\Omega})$, and suppose that
$U_*-\phi$ has a strict local minimum at $x_0$.  We must prove
\begin{equation}\label{eq:HJ-supersolution-test}
 \begin{cases}
  H(x_0,\nabla\phi(x_0))\geq0,
       &x_0\in\Omega,\\[1mm]
  \max\{H(x_0,\nabla\phi(x_0)),
         \partial_\nu\phi(x_0)\}\geq0,
       &x_0\in\partial\Omega.
 \end{cases}
\end{equation}
Choose any vector $\boldsymbol v=(v_1,\ldots,v_m)\gg0$.
Since $d_k\log v_i\to0$, the definition of $U_*$ is unchanged if
$W_{i,d_k}$ is replaced by $W_{i,d_k}+d_k\log v_i$.
The perturbed-test-function lemma therefore provides an index
$\ell$ and points $x_k\to x_0$ such that
$W_{\ell,d_k}+d_k\log v_\ell-\phi$ has a local minimum at $x_k$,
and
\begin{equation}\label{eq:perturbed-component-minimum}
 W_{\ell,d_k}(x_k)+d_k\log v_\ell
 \leq W_{j,d_k}(x_k)+d_k\log v_j,
 \qquad 1\leq j\leq m.
\end{equation}
After passing to a subsequence, $\ell$ may be taken independent of
$k$.  Then \eqref{eq:perturbed-component-minimum} implies that
\begin{equation}\label{eq:perturbed-ratio-bound}
 \exp\left[\frac{W_{\ell,d_k}(x_k)-W_{j,d_k}(x_k)}{d_k}\right]
 \leq\frac{v_j}{v_\ell},
 \qquad 1\leq j\leq m.
\end{equation}

Suppose first that $x_0\in\Omega$.  Then $x_k\in\Omega$ for
large $k$.  At the minimum point,
$\nabla W_{\ell,d_k}=\nabla\phi$ and
$\Delta W_{\ell,d_k}\geq\Delta\phi$.  Substituting
\eqref{eq:perturbed-ratio-bound} into the $\ell$-th equation of
\eqref{eq:logarithmic-system}, we find
\[
 \begin{aligned}
 -d_k\lambda(d_k)
 \leq-d_k\Delta\phi(x_k)
   +H_\ell(x_k,\nabla\phi(x_k))
   +d_k\sum_{j\ne\ell}a_{\ell j}(x_k)
       \left(\frac{v_j}{v_\ell}-1\right)
   -d_kc_\ell(x_k).
 \end{aligned}
\]
Passing to the limit, we deduce from the eigenvalue bound that
$H_\ell(x_0,\nabla\phi(x_0))\geq0$.  Since $H\geq H_\ell$, this
proves the first line of \eqref{eq:HJ-supersolution-test}.

Now assume that $x_0\in\partial\Omega$.  If
$\partial_\nu\phi(x_0)\geq0$, the second line of
\eqref{eq:HJ-supersolution-test} is immediate.  It remains to consider
$\partial_\nu\phi(x_0)<0$.  In this case $x_k$ must lie in
$\Omega$ for all sufficiently large $k$.  Indeed, if
$x_k\in\partial\Omega$, the one-sided minimum condition and the
Neumann boundary condition would imply
$
 0=\partial_\nu W_{\ell,d_k}(x_k)
 \leq\partial_\nu\phi(x_k),
$ 
contrary to $\partial_\nu\phi(x_0)<0$.  We may therefore repeat the
interior calculation and obtain
$H(x_0,\nabla\phi(x_0))\geq0$.   

It remains to prove the viscosity subsolution property.
Set
\begin{equation}\label{eq:upper-half-relaxed-limit}
 \overline U_i(x)\coloneqq
 \limsup_{k\to\infty,\,x'\to x}W_{i,d_k}(x'),
 \qquad
 U^*(x)\coloneqq \max_{1\leq i\leq m}\overline U_i(x).
\end{equation}
The function $U^*$ is upper semicontinuous.  Uniform convergence
identifies $\overline U_i$ and $U^*$ with $W$ for every $i$. 
Let $U^*-\phi$ have a strict local maximum at $x_0$.  We claim
that
\begin{equation}\label{eq:HJ-subsolution-test}
 \begin{cases}
  H(x_0,\nabla\phi(x_0))\leq0,
       &x_0\in\Omega,\\[1mm]
  \min\{H(x_0,\nabla\phi(x_0)),
         \partial_\nu\phi(x_0)\}\leq0,
       &x_0\in\partial\Omega.
 \end{cases}
\end{equation}
Fix $1\leq\ell\leq m$.  Since
$W_{\ell,d_k}\to U^*$ uniformly, there are points
$x_{\ell,k}\to x_0$ at which
$W_{\ell,d_k}-\phi$ has a local maximum.  If
$x_0\in\Omega$, then $x_{\ell,k}\in\Omega$ for large $k$.
At these maximum points, it follows from \eqref{eq:logarithmic-system} that
\[
 \begin{aligned}
 -d_k\lambda(d_k)
 &\geq-d_k\Delta\phi(x_{\ell,k})
   +H_\ell(x_{\ell,k},\nabla\phi(x_{\ell,k}))\\
 &\quad
   +d_k\sum_{j\ne\ell}a_{\ell j}(x_{\ell,k})
    \left(
     \mathrm{e}^{(W_{\ell,d_k}-W_{j,d_k})(x_{\ell,k})/d_k}-1
    \right)
   -d_kc_\ell(x_{\ell,k}).
 \end{aligned}
\]
Using $\mathrm{e}^r-1\geq r$, we have
\[
 d_k a_{\ell j}(x_{\ell,k})
 \left(\mathrm{e}^{(W_{\ell,d_k}-W_{j,d_k})(x_{\ell,k})/d_k}-1\right)
 \geq
 a_{\ell j}(x_{\ell,k})
 \bigl(W_{\ell,d_k}-W_{j,d_k}\bigr)(x_{\ell,k}).
\]
The right-hand side tends to zero because all components converge
uniformly to $W$.  Hence
$H_\ell(x_0,\nabla\phi(x_0))\leq0$.  Since $\ell$ was arbitrary,
$H(x_0,\nabla\phi(x_0))\leq0$, which is the first line of
\eqref{eq:HJ-subsolution-test}.

Finally, suppose that $x_0\in\partial\Omega$.  If
$\partial_\nu\phi(x_0)\leq0$, the second line of
\eqref{eq:HJ-subsolution-test} already holds.  When
$\partial_\nu\phi(x_0)>0$, the points $x_{\ell,k}$ are interior
for large $k$. Otherwise the one-sided maximum condition would
force
$
 0=\partial_\nu W_{\ell,d_k}(x_{\ell,k})
 \geq\partial_\nu\phi(x_{\ell,k}),
$
which contradicts $\partial_\nu\phi(x_0)>0$.  Applying the
interior argument for every $\ell$ proves the boundary
subsolution inequality.  Since $U_*=U^*=W$, the relaxed Neumann
conditions stated above are satisfied.  Together with
$\min_{\overline{\Omega}}W=0$, this completes the proof.
\end{proof}

\subsection{The Aubry geometry and finite-class selection principle}
We now return to \(N=1\), \(m=2\), and problem
\eqref{eq:main-system}. The vector fields \(\boldsymbol b_i\) in \eqref{liu-20260913-1} reduce to the scalar functions \(b_i\), and
\(H(x,p)=\max_{i=1,2}\{p^2-b_i(x)p\}\). 
Let \(\mathscr S_N\) be the Lipschitz viscosity subsolutions of the
critical Hamilton--Jacobi equation, including the subsolution part of
the relaxed Neumann condition, and define
\(d_H(x,y)\coloneqq \sup_{u\in\mathscr S_N}\{u(x)-u(y)\}\).
The Neumann Aubry set consists of the points \(y\) at which
\(d_H(\,\cdot\,,y)\) also satisfies the supersolution condition.
Two Aubry points belong to the same static class provided that
\(d_H(x,y)+d_H(y,x)=0\). For every critical solution, the Neumann representation theorem
\cite[Theorem~6.8]{Ishii2011} yields
\begin{equation}\label{eq:Aubry-representation}
 W(x)=\min_{y\in\mathcal A_N}\{W(y)+d_H(x,y)\}.
\end{equation}
In particular, \(W\) is constant on each static class.

\begin{theorem}\label{thm:exact-one-dimensional-Aubry}
The Hamiltonian $H(x,p)$ of \eqref{main:eq:HJ-Neumann} satisfies
\[
 H(x,p)=
 \begin{cases}
  p^2-b_-(x)p,&p\geq0,\\
  p^2-b_+(x)p,&p\leq0,
 \end{cases}
 \qquad
 \{p:H(x,p)\leq0\}
 =[b_+(x)\wedge0,\,b_-(x)\vee0].
\]
Moreover, we have
\begin{equation}\label{eq:explicit-critical-distance}
 d_H(x,\xi)=
 \begin{cases}
  \displaystyle\int_\xi^x(b_-(s))_+\,\mathrm{d}s,&x\geq\xi,\\[3mm]
  \displaystyle\int_x^\xi(-b_+(s))_+\,\mathrm{d}s,&x\leq\xi.
 \end{cases}
\end{equation}
The set $\{p:H(x,p)\leq0\}$ reduces to \(\{0\}\) exactly on
\(\mathcal A_0\), and the full Neumann Aubry set is
\(\mathcal A_N\), with both sets defined in
\eqref{eq:main-Aubry-set}. Every connected component of
\(\mathcal A_0\) constitutes a single static class, and every endpoint in
\(\mathcal A_N\setminus\mathcal A_0\) forms an additional singleton class.
\end{theorem}

\begin{proof}

We first derive the explicit Hamiltonian and the critical
semidistance.
For \(p\geq0\), the maximum defining \(H\) selects \(b_-\), whereas
for \(p\leq0\) it selects \(b_+\). This establishes the stated formula for
\(H\) and its zero sublevel. Hence, every critical subsolution \(u\) satisfies
\[
 b_+(s)\wedge0\leq u'(s)\leq b_-(s)\vee0
 \quad\text{for a.e. }s.
\]
If \(x>\xi\), integration yields
\(u(x)-u(\xi)\leq\int_\xi^x(b_-(s))_+\,\mathrm{d}s\). Equality is attained by taking
\(u'=(b_-)_+\) on \((\xi,x)\) and zero elsewhere. The case
\(x<\xi\) is analogous, using \(u'=b_+\wedge0\) on \((x,\xi)\).
The resulting Lipschitz functions are viscosity subsolutions since
\(H(x,\cdot)\) is convex and the differential inequality holds almost
everywhere. This establishes \eqref{eq:explicit-critical-distance}.

We next characterize the Aubry points. 
Set \(v_y(x)\coloneqq d_H(x,y)\), with \(y\in(\ell,r)\).
By \eqref{eq:explicit-critical-distance},
\[
 v_y'(x)=
 \begin{cases}
  b_+(x)\wedge0,&x<y,\\
  b_-(x)\vee0,&x>y.
 \end{cases}
\]
Hence, \(H(x,v_y'(x))=0\) for \(x\ne y\). At \(\ell\),
any test function \(\phi\) touching \(v_y\) from below
satisfies \(\phi'(\ell)\leq v_y'(\ell+)\leq0\).
At \(r\), such a test function satisfies
\(\phi'(r)\geq v_y'(r-)\geq0\).
Thus, the boundary supersolution inequalities in
\eqref{eq:HJ-supersolution-test} hold at both endpoints.
It remains to verify the supersolution condition at \(y\).
Set \(q_-=b_+(y)\wedge0\) and \(q_+=b_-(y)\vee0\).
If \(\phi\) touches \(v_y\) from below at \(y\), its
derivative satisfies
\(q_-\leq\phi'(y)\leq q_+\).
If \(b_-(y)\leq0\leq b_+(y)\), both bounds vanish,
and hence \(H(y,\phi'(y))=H(y,0)=0\).
Otherwise, \(q_-<q_+\). We choose
\(p=(q_-+q_+)/2\) and \(\phi(x)=p(x-y)\).
Since \(v_y(y)=0\) and
\(v_y'(y-)=q_-<p<q_+=v_y'(y+)\), this affine function
touches \(v_y\) from below at \(y\).
The explicit formula for \(H\) yields
\[
 H(y,p)=
 \begin{cases}
  -b_-(y)^2/4<0,&b_-(y)>0,\\
  -b_+(y)^2/4<0,&b_+(y)<0.
 \end{cases}
\]
The supersolution condition therefore fails.
Consequently, an interior point \(y\) belongs to
\(\mathcal A_N\) precisely when
\(b_-(y)\leq0\leq b_+(y)\), or equivalently,
\(b_1(y)b_2(y)\leq0\).

At the left endpoint, \eqref{eq:explicit-critical-distance}
yields \(v_\ell(\ell)=0\) and
\(v_\ell'(\ell+)=(b_-(\ell))_+\).
If \(\phi\in C^2\) touches \(v_\ell\) from below at \(\ell\), then
\(\phi'(\ell)\leq(b_-(\ell))_+\).
The boundary supersolution condition
\eqref{eq:HJ-supersolution-test} reads
\(
 \max\{H(\ell,\phi'(\ell)),-\phi'(\ell)\}\geq0.
\) 
If \(b_-(\ell)\leq0\), we have \(\phi'(\ell)\leq0\),
so this inequality holds. If \(b_-(\ell)>0\), take
\(\phi(x)=\tfrac12 b_-(\ell)(x-\ell)\).
In view of
\(v_\ell(x)=b_-(\ell)(x-\ell)+o(x-\ell)\), we see that \(\phi\) touches \(v_\ell\) from below at \(\ell\).
However, \(H(\ell,\phi'(\ell))=-{b_-(\ell)^2}/{4}<0\) and \(-\phi'(\ell)=-{b_-(\ell)}/{2}<0\). 
Hence, \(\ell\in\mathcal A_N\) if and only if
\(b_-(\ell)\leq0\). 
At the right endpoint, a similar argument yields that \(r\in\mathcal A_N\) if and only if \(b_+(r)\geq0\).
Together with the interior characterization, this proves
\eqref{eq:main-Aubry-set}.

It remains to determine the static classes. For \(x<y\),
it follows from \eqref{eq:explicit-critical-distance} that
\[
 d_H(y,x)+d_H(x,y)
 =\int_x^y\bigl((b_-(s))_++(-b_+(s))_+\bigr)\,\mathrm{d}s.
\]
The integral vanishes precisely when \([x,y]\subset\mathcal A_0\), which
establishes the assertion about static classes. The proof is complete.
\end{proof}
Under assumption (H),
Theorem~\ref{thm:exact-one-dimensional-Aubry} establishes that the
three types of static classes described in
Section~\ref{sec:introduction-main} are exhaustive.
An isolated interior component of \(\mathcal A_0\) consists of a single common zero at which the two drift derivatives have the
same sign. For every interval component of positive length,
the drifts have opposite signs except at finitely many common
zeros. If no common zero lies in its interior, a fixed
labeling ensures that \(b_p>0>b_n\) throughout the interior, and the
class is regular. Otherwise, by assumption (H), both drifts change sign at each
interior common zero. Hence, at these zeros the labels interchange, so the class is composite.


For the remainder of this section, set
\[
 (\mathcal L_d\boldsymbol U)_i
 \coloneqq -dU_i''-b_iU_i'+\sum_{j=1}^2M_{ij}U_j,
 \qquad i=1,2.
\]

\begin{definition}\label{def:local-effective-value}
An isolated compact static class \(K\) has local effective value
\(\Lambda_K\) if, for every \(\eta>0\) and every sufficiently small
relative neighborhood \(U\) of \(K\) containing no other static class,
the following properties hold.

\emph{(i)} {There exist} a relative neighborhood \(V\), with
\(K\subset V\subset\overline V\subset U\), and, for all sufficiently
small \(d\), a vector function
\(\boldsymbol\Psi_{d,\eta}\in
C(\overline V;(0,\infty)^2)\) such that
\[
 \mathcal L_d\boldsymbol\Psi_{d,\eta}
 \geq(\Lambda_K-\eta)\boldsymbol\Psi_{d,\eta}
 \quad\text{in }V
\]
in the distributional sense. The function is piecewise \(C^2\), and satisfies
\(\partial_\nu\boldsymbol\Psi_{d,\eta}\geq0\) on
\(\partial V\cap\partial\Omega\). At an interior interface, the one-sided derivatives
{may exhibit a downward jump, satisfying}
\((\Psi_{d,\eta,i}')^->(\Psi_{d,\eta,i}')^+\).
{Upon normalization by a common positive scalar}, its phases
\(S_{d,\eta,i}\coloneqq -d\log\Psi_{d,\eta,i}\) satisfy
\(\max_i\|S_{d,\eta,i}\|_{C(K)}\to0\) and, for some \(\delta_0>0\),
\begin{equation}\label{eq:local-phase-gap}
 \max_{i=1,2}S_{d,\eta,i}(x)
 \leq d_H(x,K)-\delta_0
 \quad\text{on }\partial V\cap(\ell,r)
\end{equation}
for all sufficiently small \(d\).

\emph{(ii)} There exist a relative neighborhood \(V^Q\), with
\(K\subset V^Q\subset\overline{V^Q}\subset U\), and, for all
sufficiently small \(d\), a nonzero, nonnegative, compactly supported
subsolution \(\boldsymbol Q_{d,\eta}\) of
\[
 \mathcal L_d\boldsymbol Q_{d,\eta}
 \leq(\Lambda_K+\eta)\boldsymbol Q_{d,\eta}
\]
for the direct Neumann problem on \((\ell,r)\), with
\(\operatorname{supp}\boldsymbol Q_{d,\eta}
\subset\overline{V^Q}\).
\end{definition}


\begin{theorem}\label{thm:finite-class-selection}
Suppose that
$
 \mathcal A_N=\mathop{\bigcup}_{K\in\mathcal C}K
$
is a finite disjoint union of compact static classes and that every
$K\in\mathcal C$ has a local effective value in the sense of
Definition~{\rm\ref{def:local-effective-value}}.  Then
\begin{equation}\label{eq:finite-class-limit}
 \lim_{d\searrow0}\lambda(d)
 =\min_{K\in\mathcal C}\Lambda_K.
\end{equation}
If this minimum is attained at a unique class $K_*$, then
$W_{i,d}\to d_H(\,\cdot\,,K_*)$ uniformly on $[\ell,r]$ for
$i=1,2$.
\end{theorem}

\begin{proof}
We first describ the static classes that may occur in a
subsequential phase. Fix \(d_k\searrow0\). Passing to a
subsequence if necessary, we may assume by Theorem~\ref{thm:common-HJ-limit}
that 
\(\lambda(d_k)\to\lambda\) and \(W_{i,d_k}\to W\) uniformly. By
\eqref{eq:Aubry-representation}, we have
\[
 W(x)=\min_{K\in\mathcal C}\{W(K)+d_H(x,K)\}.
\]
For two static classes, we set \(d_H(L,K)\coloneqq d_H(y,x)\), where
\(x\in K\) and \(y\in L\). This definition is independent of the
choice of representatives since both directed costs within a static class
vanish.
On \(\mathcal C\), we define
\[
 K\preceq_W L
 \quad\Longleftrightarrow\quad
 W(L)=W(K)+d_H(L,K).
\]
If \(K\preceq_W L\) and \(L\preceq_W P\),
the triangle inequality and the Aubry representation
yield opposite inequalities between \(W(P)\) and \(W(K)+d_H(P,K)\).
Hence, the equality must hold, and the relation is transitive.
If both \(K\preceq_W L\) and \(L\preceq_W K\), then
the round-trip distance between \(K\) and \(L\) vanishes, so
that \(K=L\). Hence
\(\preceq_W\) is a partial order.
Let \(K\) be a minimal class. Then
\(W(K)<W(L)+d_H(K,L)\) for every \(L\ne K\). Since \(\mathcal C\)
is finite, these strict inequalities remain valid in a
neighborhood of \(K\) in the relative topology, and hence
\begin{equation}\label{eq:significant-class}
 W(x)=W(K)+d_H(x,K)
\end{equation}
holds there. In particular, every minimal class is
significant.

We next establish the lower bound in \eqref{eq:finite-class-limit}. Fix \(\eta>0\), and take the local
supersolution from Definition~\ref{def:local-effective-value} with tolerance
\(\eta/2\). On \(K\),
\[
 d_k\log\frac{\varphi_{i,d_k}}{\Psi_{d_k,\eta/2,i}}
 =S_{d_k,\eta/2,i}-W_{i,d_k}\to-W(K)
\]
uniformly in \(i\). On the artificial boundary, these quantities
lie strictly below \(-W(K)\), in view of
\eqref{eq:local-phase-gap} and \eqref{eq:significant-class}. Hence
the simultaneous
maximum of \(\varphi_{i,d_k}/\Psi_{d_k,\eta/2,i}\) is attained away
from the artificial boundary.
Suppose that \(\lambda(d_k)<\Lambda_K-\eta\), and set
\[
 t_k\coloneqq \max_{i=1,2}\max_{x\in\overline V}
 \frac{\varphi_{i,d_k}(x)}{\Psi_{d_k,\eta/2,i}(x)},
 \qquad
 \boldsymbol Z_k\coloneqq
 t_k\boldsymbol\Psi_{d_k,\eta/2}-\boldsymbol\varphi_{d_k}.
\]
Then \(\boldsymbol Z_k\geq0\), with one component
vanishing at a point
away from the artificial boundary. Moreover, it follows that
\[
 \bigl(\mathcal L_{d_k}\boldsymbol Z_k\bigr)_i
 -(\Lambda_K-\eta/2)Z_{k,i}
 \geq[\Lambda_K-\eta/2-\lambda(d_k)]\varphi_{i,d_k}>0.
\]
If \(Z_{k,i}\) vanished at an interior point, then
\[
 (-d_k\partial_{xx}-b_i\partial_x+M_{ii}-\Lambda_K+\eta/2)Z_{k,i}
 \geq[\Lambda_K-\eta/2-\lambda(d_k)]\varphi_{i,d_k}
      -\sum_{j\ne i}M_{ij}Z_{k,j}>0,
\]
which the strong maximum principle rules out.
At a domain endpoint, the Hopf boundary point lemma would imply
\(\partial_\nu Z_{k,i}<0\), contrary to
\(\partial_\nu Z_{k,i}\geq0\). This contradiction shows
that
\(\lambda(d_k)\geq\Lambda_K-\eta\) for all large \(k\).
Passing to the limit \(k\to\infty\) and then
\(\eta\searrow0\), we obtain
\[
 \lambda\geq\Lambda_K\geq
 \Lambda_{\min}\coloneqq \min_{L\in\mathcal C}\Lambda_L.
\]
Every subsequential phase has a minimal class, and thus
\(\liminf_{d\searrow0}\lambda(d)\geq\Lambda_{\min}\).
For the reverse inequality, choose \(K_0\) with
\(\Lambda_{K_0}=\Lambda_{\min}\) and take the compactly supported
subsolution supplied by Definition~\ref{def:local-effective-value}(ii).
By the standard comparison principle for compactly supported
subsolutions of cooperative Neumann systems,
\(\lambda(d)\leq\Lambda_{\min}+\eta\). Hence
\(\limsup_{d\searrow0}\lambda(d)\leq\Lambda_{\min}\),
which establishes \eqref{eq:finite-class-limit}.

It remains to identify the limiting phase when the minimizing class is
unique.
Suppose that \(K_*\) is the unique minimizing class.
Every minimal class of a subsequential phase is
significant, and by
\eqref{eq:finite-class-limit}, every significant class has
value \(\Lambda_{\min}\). Hence \(K_*\) is the unique minimal element
of \((\mathcal C,\preceq_W)\). In a finite partially
ordered set, every element lies above a minimal element, and hence
\(W(K)=W(K_*)+d_H(K,K_*)\) for all \(K\in\mathcal C\).
By the Aubry representation and the triangle inequality,
\(W(x)=W(K_*)+d_H(x,K_*)\). Since \(\min W=0\),
we have \(W(K_*)=0\). The limit is thus
independent of the chosen subsequence, and the convergence holds for
the full family. This completes the proof.
\end{proof}

\section{\bf Nondegenerate isolated classes}
\label{sec:isolated-classes}

This section determines the local effective value of an isolated static
class. The resulting values will be combined with the interval-class
values {via} Theorem~\ref{thm:finite-class-selection}.

Let \(\xi\in(\ell,r)\) be an isolated static class such that
\begin{equation}\label{iso:eq:interior-geometry}
 b_1(\xi)=b_2(\xi)=0,
 \qquad
 \beta_i\coloneqq b_i'(\xi)\ne0,
 \qquad
 \beta_1\beta_2>0.
\end{equation}
Under the simple-zero hypothesis, the last condition distinguishes an
isolated common zero from a point in a nondegenerate interval class.
Define
\[
 \Lambda_\xi
 \coloneqq \sigma\!\left(
 \boldsymbol M(\xi)+
 \operatorname{diag}\bigl((\beta_1)_+,(\beta_2)_+\bigr)
 \right).
\]
We also consider a strict boundary singleton. The two
possible configurations are
\begin{equation}\label{iso:eq:strict-boundary-geometry}
 \xi=\ell,\quad b_1(\ell), b_2(\ell)<0,
 \qquad\text{or}\qquad
 \xi=r,\quad b_1(r)0, b_2(r)>0.
\end{equation}
For either configuration, set
\(\Lambda_\xi\coloneqq \sigma(\boldsymbol M(\xi))\).
At a strict noncharacteristic boundary class, the frozen zero-order
matrix {dictates} the local problem. 

The localized supersolution requires an integrable whole-line
Ornstein-Uhlenbeck profile with quantitative Gaussian decay. The next
lemma establishes this estimate.

\begin{lemma}\label{iso:lem:localized-OU-profile}
Let $\beta_1,\beta_2>0$, set
$\beta_*\coloneqq \min\{\beta_1,\beta_2\}$, and let
\[
 (\mathscr L_\xi\boldsymbol U)_i
 \coloneqq -U_i''-\beta_i yU_i'
   +\sum_{j=1}^2M_{ij}(\xi)U_j.
\]
There is a postive function
$\boldsymbol U\in C^2(\mathbb R)\cap
L^1(\mathbb R;\mathbb R^2)$
such that
\[
 \mathscr L_\xi\boldsymbol U=\Lambda_\xi\boldsymbol U
 \quad\text{in }\mathbb R,
 \qquad
 \Lambda_\xi
 =\sigma\!\left(
 \boldsymbol M(\xi)+\operatorname{diag}(\beta_1,\beta_2)
 \right).
\]
Moreover, for every $0<\gamma<\beta_*$, there is
$C_\gamma>0$ such that
\begin{equation}\label{iso:eq:OU-Gaussian-bound}
 U_i(y)+|U_i'(y)|
 \le C_\gamma(1+|y|)\mathrm{e}^{-\gamma y^2/2},
 \qquad y\in\mathbb R,\quad i=1,2.
\end{equation}
\end{lemma}

\begin{proof}
 Let
\((\Lambda_R,\boldsymbol U^R)\) be the principal eigenpair of the following Dirichlet problem
\begin{equation*}
 \mathscr L_\xi\boldsymbol U=\Lambda_\xi\boldsymbol U, \,\,y\in(-R,R), \quad \boldsymbol U(\pm R)=0, 
 \quad
 \sum_i\int_{-R}^R U_i {\rm d}y=1,
\end{equation*}
whose existence, positivity, and algebraic simplicity follow from the
cooperative maximum principle on a bounded interval. 
For any \(\gamma>0\), set
\(\rho_\gamma(y)=\exp(-\gamma y^2/2)\). Then
\begin{equation}\label{eq:OU-Gaussian-calculation}
 -\rho_\gamma''-\beta_i y\rho_\gamma'
 =
 \bigl[\gamma+\gamma(\beta_i-\gamma)y^2\bigr]\rho_\gamma, \quad \forall y\in \mathbb{R}.
\end{equation}
Choose \(0<\gamma<\min_i\beta_i\) and fix \(\boldsymbol e\gg0\).
By
\eqref{eq:OU-Gaussian-calculation} and the uniform boundedness of \(\Lambda_R\), there exists \(R_0\in(0,R)\),
independently of \(R\), such that
\[
 (\mathscr L_\xi-\Lambda_R)
 (\rho_\gamma\boldsymbol e)
 \ge c(1+y^2)\rho_\gamma\boldsymbol e
 \quad\text{for }R_0\le |y|<R,
\]
with some \(c>0\). For \(R\ge R_0+1\), the normalization, combined
with the local \(L^1\)-to-\(L^\infty\) estimate, provides a uniform bound for
\(\boldsymbol U^R(\pm R_0)\). By the quotient maximum argument for this strict positive supersolution, applied separately on \((R_0,R)\) and \((-R,-R_0)\), comparing \(\boldsymbol U^R\) with a fixed multiple of \(\rho_\gamma \boldsymbol e\) yields
\begin{equation}\label{eq:OU-tail}
 U_i^R(y)\leq C \exp(-\gamma y^2/2)
 \quad (|y|\geq R_0,\ i=1,2),
\end{equation}
which in particular implies that
\[
 \sup_R\sum_i\int_{\{R_1<|y|<R\}}U_i^R(y){\rm d}y
 \to0
 \quad\text{as }R_1\to\infty.
\]
This is the asserted uniform tightness.
In view of \eqref{eq:OU-tail}, applying the interior Harnack and Schauder estimates for cooperative
systems yields
\(\|\boldsymbol U^R\|_{C^{2,\alpha}(\tilde{I})}\leq C\) for any compact interval \(\tilde{I}\).
Thus, along a sequence \(R_j\to\infty\), there exist \(\Lambda\in\mathbb{R}\) and \(\boldsymbol U\in C^2(\mathbb{R})^2\) such that \(\Lambda_{R_j}\to\Lambda\) and \(\boldsymbol U^{R_j}\to \boldsymbol U\) in \((C^2_{\rm loc}(\mathbb{R}))^2\). Moreover, the limiting function \(\boldsymbol U\) solves \(\mathscr L_\xi\boldsymbol U=\Lambda\boldsymbol U\).
The tightness ensures that \(\sum_i\int_{\mathbb{R}} U_i=1\), while the strong maximum
principle yields \(\boldsymbol U\gg0\).

Next, we apply
\eqref{eq:OU-tail} to establish that
\begin{equation}\label{liu-20260728-2}
     U_i'(y)+\beta_i yU_i(y)\to 0
 \quad\text{as }y\to\pm\infty.
\end{equation}
Indeed, {for} sufficiently large \(|y|\), {we set}
\begin{equation}\label{liu-20260728-4}
h_y:=\frac{1}{1+|y|},\qquad
V_i(s):=U_i(y+h_y s),\quad |s|<2.
\end{equation}
A direct calculation shows that \(V=(V_i)\) satisfies
\begin{equation}\label{liu-20260728-5}
-V_i''
-\beta_i h_y(y+h_y s)V_i'
+h_y^2\sum_j
M_{ij}(\xi)V_j=h_y^2\Lambda V_i, \quad |s|<2.
\end{equation}
Since
\(|h_y(y+h_y s)|\le C\) for \(|s|\leq 2\), all coefficients of the above equation are uniformly bounded. {By} standard {interior} gradient estimates{, we obtain}
\[
|V_i'(0)|
\leq C\max_j\|V_j\|_{L^\infty(-2,2)},
\]
where the constant \(C\) is independent of \(y\). Therefore,
\begin{equation}\label{liu-20260728-3}
|U_i'(y)|
=\frac{|V_i'(0)|}{h_y}
\le C(1+|y|)
\max_j\sup_{|x-y|\le 2h_y}|U_j(x)|.
\end{equation}
In view of \(|x^2-y^2|\le C\) whenever \(|x-y|\le 2h_y\), by \eqref{eq:OU-tail} we have
\[
\sup_{|x-y|\le 2h_y}U_j(x)
\le C e^{-\gamma y^2/2}.
\]
Consequently, \eqref{liu-20260728-3} yields
$
|U_i'(y)|
\le C(1+|y|)e^{-\gamma y^2/2}$.
On the other hand,
in view of
\(|\beta_i yU_i(y)|
\le C|y|e^{-\gamma y^2/2}\), we deduce that
\[
\left|U_i'(y)+\beta_i yU_i(y)\right|
\le C(1+|y|)e^{-\gamma y^2/2}
\to 0
\qquad (y\to\pm\infty),
\]
which establishes the claim \eqref{liu-20260728-2}.

To proceed further, define
$
 \mathcal F_iU_i:=U_i''+(\beta_i yU_i)'$. Then
\[
 -U_i''-\beta_i yU_i'
 =-\mathcal F_iU_i+\beta_iU_i,
\]
By \eqref{liu-20260728-2}, we have \(\int_{\mathbb{R}}F_iU_i{\rm d}y=0\). Integrating \(\mathscr L_\xi\boldsymbol U=\Lambda\boldsymbol U\) then yields
\[
 \left(
 \boldsymbol M(\xi)+\operatorname{diag}(\beta_1,\beta_2)
 \right)\overline{\boldsymbol U}
 =\Lambda\overline{\boldsymbol U},
 \qquad
 \overline U_i\coloneqq \int_{\mathbb R}U_i(y)\,\mathrm{d}y.
\]
Since \(\overline{\boldsymbol U}\gg0\), the defining Perron property of
\(\sigma\) requires \(\Lambda=\Lambda_\xi\). This completes the proof.
\end{proof}

\begin{theorem}\label{iso:thm:isolated-local-value}
Let $K=\{\xi\}$ satisfy either
\eqref{iso:eq:interior-geometry} or
\eqref{iso:eq:strict-boundary-geometry}, and define $\Lambda_\xi$
as above. Then $K$ has local effective value
$\Lambda_\xi=\Lambda_K^{\mathrm{iso}}$ in the sense of
Definition~\ref{def:local-effective-value}.
\end{theorem}

\begin{proof}
Fix \(\eta>0\) and a sufficiently small relative neighborhood \(U\)
as in Definition~\ref{def:local-effective-value}. Fix a relative
neighborhood \(V^Q\) with
\(K\subset V^Q\subset\overline{V^Q}\subset U\). All compactly
supported subsolutions constructed below {are} supported in
\(\overline{V^Q}\) for small \(d\).

We first treat the case \(\beta_1,\beta_2>0\)
and construct the
positive local supersolution.
Let \(\boldsymbol U\) be given by
Lemma~\ref{iso:lem:localized-OU-profile}, and set
\(z=x-\xi\), \(y=z/\sqrt d\), and
\(\boldsymbol F_d(x)\coloneqq \boldsymbol U(y)\). Then
\begin{equation}\label{iso:eq:variable-coefficient-error}
 \begin{split}
 \bigl((\mathcal L_d-\Lambda_\xi)\boldsymbol F_d\bigr)_i
 =E_{d,i}\coloneqq {}&
 -\frac{b_i(\xi+z)-\beta_i z}{\sqrt d}\,U_i'(y)\\
 &+\sum_{j=1}^2
  \bigl(M_{ij}(\xi+z)-M_{ij}(\xi)\bigr)U_j(y).
 \end{split}
\end{equation}
Choose
\(0<q<\gamma<\beta_*\coloneqq \min_i\beta_i\), let
\(\boldsymbol v\gg0\) satisfy
\(\boldsymbol M(\xi)\boldsymbol v=\sigma(\boldsymbol M(\xi))\boldsymbol v\), and
define the vector function \(\boldsymbol G_d\) by
\(G_{d,i}(x)\coloneqq v_i\mathrm{e}^{-qz^2/(2d)}\).
By continuity of \(b_i'\) and \(\boldsymbol M\),
we choose a sufficiently small
\(\delta>0\) such that, for both
\(\lambda=\Lambda_\xi-\eta\) and
\(\lambda=\Lambda_\xi+\eta\),
\begin{equation}\label{iso:eq:slow-Gaussian-strict}
 (\mathcal L_d-\lambda)\boldsymbol G_d
 \ge c_0(1+y^2)\boldsymbol G_d
 \quad\text{if }|y|\ge R_0,\quad |z|<\delta,
\end{equation}
where \(c_0>0\) and \(R_0>1\) are independent of \(d\). Indeed,
\[
 \frac{((\mathcal L_d-\lambda)\boldsymbol G_d)_i}{G_{d,i}}
 =q+\frac{q b_i(\xi+z)z-q^2z^2}{d}
  +\frac{(\boldsymbol M(\xi+z)\boldsymbol v)_i}{v_i}-\lambda,
\]
and the coefficient of \(y^2\) tends to
\(q(\beta_i-q)>0\).
By the differentiability of \(b_i\) at \(\xi\), the continuity of
\(\boldsymbol M\), and
\eqref{iso:eq:OU-Gaussian-bound}--%
\eqref{iso:eq:variable-coefficient-error}, we obtain
\[
 |E_{d,i}|
 \leq o_\delta(1)(1+|y|)^2
       \mathrm{e}^{-\gamma y^2/2},
 \qquad |z|<\delta,
\]
where \(o_\delta(1)\to0\) as \(\delta\searrow0\), uniformly in \(d\)
and \(i\). By the Gaussian upper bound for \(F_{d,i}\)
and the slower Gaussian lower bound for \(G_{d,i}\), and since
\(q<\gamma\), we deduce that
\begin{equation}\label{iso:eq:error-absorbed-by-slow-tail}
 \frac{|E_{d,i}|+\eta F_{d,i}}{G_{d,i}}
 \leq C\eta(1+|y|)^2
 \mathrm{e}^{-(\gamma-q)y^2/2},
 \qquad |z|<\delta.
\end{equation}
For every fixed \(R>0\), it follows from the relation \(z=\sqrt d\,y\) and the local
uniform positivity of the Ornstein-Uhlenbeck profile
that
\(\max_{1\leq i\leq m}\sup_{|y|\leq R}
  \frac{|E_{d,i}|}{F_{d,i}}\to 0\) as \(d\searrow0\).
In particular, once \(R_0\) is fixed, we have \(|E_{d,i}|\leq\frac{\eta}{4}F_{d,i}\) when \(|y|\leq R_0\)
for all small \(d\). On the same fixed region,
by the
coefficient bounds and the positivity of \(F_{d,i}\),
\[
 \max_{\lambda\in\{\Lambda_\xi-\eta,\Lambda_\xi+\eta\}}
 \max_{1\leq i\leq m}\sup_{|y|\leq R_0}
 \frac{\bigl|((\mathcal L_d-\lambda)\boldsymbol G_d)_i\bigr|}
      {F_{d,i}}
 \leq C_{R_0,\eta}.
\]
We may then choose \(\varepsilon>0\) so small that
\[
 \varepsilon
 \bigl|((\mathcal L_d-\lambda)\boldsymbol G_d)_i\bigr|
 \leq\frac{\eta}{4}F_{d,i},
 \qquad
 |y|\leq R_0,\quad
 \lambda=\Lambda_\xi\pm\eta,
\]
and
\(U_i(0)-\varepsilon v_i>0\) for \(i=1,2\).
Next choose \(R\geq R_0\)
such that
\[
 C(1+|y|)^2
 \mathrm{e}^{-(\gamma-q)y^2/2}
 \leq\frac{\varepsilon c_0}{2}(1+|y|^2)
 \qquad (|y|\geq R),
\]
where \(C\) is the constant in
\eqref{iso:eq:error-absorbed-by-slow-tail}. We then take \(d\)
sufficiently small that \(R\sqrt d<\delta\) and \(|E_{d,i}|\leq{\eta}F_{d,i}/2\) for \(|y|\leq R\).
On \(|y|\leq R_0\), the variable-coefficient error is at most
\(\eta F_{d,i}/2\), while the term generated by
\(\varepsilon\boldsymbol G_d\) has absolute value at most
\(\eta F_{d,i}/4\). The terms \(\pm\eta F_{d,i}\) therefore determine
the required signs. On \(R_0\leq|y|\leq R\), the same conclusion
follows from \(|E_{d,i}|\leq\eta F_{d,i}/2\) and
\eqref{iso:eq:slow-Gaussian-strict}. Finally, for
\(R\leq|y|<\delta/\sqrt d\), the choice of \(R\)
ensures
\[
 |E_{d,i}|+\eta F_{d,i}
 \leq\frac{\varepsilon c_0}{2}(1+|y|^2)G_{d,i},
\]
which is absorbed by the lower bound in
\eqref{iso:eq:slow-Gaussian-strict}. Hence
\begin{equation}\label{iso:eq:two-local-inequalities}
 \begin{aligned}
 [\mathcal L_d-(\Lambda_\xi-\eta)]
   (\boldsymbol F_d+\varepsilon\boldsymbol G_d)&\geq\boldsymbol0,\\
 [\mathcal L_d-(\Lambda_\xi+\eta)]
   (\boldsymbol F_d-\varepsilon\boldsymbol G_d)&\leq\boldsymbol0
 \end{aligned}
 \qquad\text{in }(\xi-\delta,\xi+\delta).
\end{equation}
Take
\(V=(\xi-\delta,\xi+\delta)\) and
\(\boldsymbol\Psi_{d,\eta}^{-}
 \coloneqq \boldsymbol F_d+\varepsilon\boldsymbol G_d\). The first inequality
in \eqref{iso:eq:two-local-inequalities}
is precisely the supersolution
requirement in Definition~\ref{def:local-effective-value}. Moreover,
by
\eqref{iso:eq:OU-Gaussian-bound},
\( -d\log\Psi_{d,\eta,i}^{-}(x)
 \to \frac q2(x-\xi)^2\)
locally uniformly away from
\(\xi\).
The explicit semidistance in Theorem \ref{thm:exact-one-dimensional-Aubry}
yields the expansion
\[
 d_H(x,\xi)
 =\frac{\beta_*}{2}(x-\xi)^2
  +o(|x-\xi|^2)
 \quad(x\to\xi).
\]
After decreasing \(\delta\), we may arrange that
\[
 d_H(x,\xi)-\frac q2(x-\xi)^2
 \ge\frac{\beta_*-q}{4}(x-\xi)^2
 \qquad (0<|x-\xi|\le\delta).
\]
The phase convergence at the two artificial endpoints then implies
\eqref{eq:local-phase-gap}, for example with
\(\delta_0=(\beta_*-q)\delta^2/8\). Multiplication by the common
factor \([\max_i\{U_i(0)+\varepsilon v_i\}]^{-1}\) changes these phases
by \(O(d)\) and makes their minimum at \(\xi\) equal to zero. This
establishes the first part of Definition~\ref{def:local-effective-value}.

For the compactly supported subsolution, we set
\[
 \boldsymbol q_d\coloneqq \boldsymbol F_d-\varepsilon\boldsymbol G_d,
 \qquad
 \boldsymbol Q_{d,\eta}\coloneqq
 \bigl((q_{d,1})_+,(q_{d,2})_+\bigr)^{\mathsf{T}}.
\]
Since \(q<\gamma\), the estimate
\eqref{iso:eq:OU-Gaussian-bound} implies that
\(F_{d,i}/G_{d,i}\to0\) as \(|y|\to\infty\). Hence,
\(\boldsymbol Q_{d,\eta}\not\equiv0\), and
\[
\operatorname{supp}\boldsymbol Q_{d,\eta}
 \subset[\xi-Y\sqrt d,\xi+Y\sqrt d]\Subset V^Q
\]
for some \(Y\) independent of the small  \(d\). We select \(d\) small enough so that
\(Y\sqrt d<\min\{\delta,
\operatorname{dist}(\xi,\overline\Omega\setminus V^Q)\}\).
By the second inequality in \eqref{iso:eq:two-local-inequalities},
\[
 [\mathcal L_d-(\Lambda_\xi+\eta)]\boldsymbol Q_{d,\eta}
 \leq\boldsymbol0.
\]
Its support is compactly contained in the region where the local
inequality holds, and hence \(\boldsymbol Q_{d,\eta}\) is the required
compactly supported subsolution.

It remains to carry out the two constructions in the zero-action
configurations.
Suppose that either \(\beta_1,\beta_2<0\) or that \(\xi\) is a strict
boundary singleton.
In all these configurations, there is an affine local coordinate
\(s=s(x)\), with \(|s'|=1\), \(s(\xi)=0\), and a constant
\(\kappa>0\) such that
\begin{equation}\label{iso:eq:zero-action-drift}
 b_i(x)s'(x)s(x)\le-\kappa s(x)^2,
 \qquad x\in V,\quad i=1,2.
\end{equation}
For an interior point, {set} \(s=x-\xi\). At \(\ell\), {set}
\(s=x-\ell\), {and at} \(r\), {set} \(s=r-x\).
The explicit semidistance formula further {yields}
\(d_H(x,\xi)=0\) for \(x\in V\).
Let \(\boldsymbol v\gg0\) satisfy
\(\boldsymbol M(\xi)\boldsymbol v=\Lambda_\xi\boldsymbol v\). Shrink \(V\)
so that
\[
 \left|\frac{(\boldsymbol M(x)\boldsymbol v)_i}{v_i}-\Lambda_\xi\right|
 \le\frac\eta4,
 \qquad x\in V,\quad i=1,2.
\]
We choose \(0<a<\min\{\kappa,\eta/2\}\) and {define}
\begin{equation}\label{iso:eq:zero-action-supersolution}
 \Psi_{d,\eta,i}^{-}(x)
 \coloneqq v_i\exp\!\left(\frac{a s(x)^2}{2d}\right).
\end{equation}
Since \(s''=0\), direct differentiation, together with
\eqref{iso:eq:zero-action-drift}, yields
\[
 \frac{(\mathcal L_d\boldsymbol\Psi_{d,\eta}^{-})_i}
      {\Psi_{d,\eta,i}^{-}}
 \ge\Lambda_\xi-\frac\eta4-a
   +a(\kappa-a)\frac{s^2}{d}
 \ge\Lambda_\xi-\eta.
\]
At a domain endpoint, the functions in
\eqref{iso:eq:zero-action-supersolution} satisfy the Neumann
condition exactly. Their normalized phases converge to \(-as^2/2\). {Since} \(d_H(x,\xi)=0\),
\eqref{eq:local-phase-gap} holds after fixing the artificial boundary of
\(V\).
For the compactly supported subsolution, fix \(c=1/2\) and {select}
\(\alpha>0\) {sufficiently} small {such} that \(\alpha/\kappa\le\log(4/3)\) and \(\alpha\le {3\eta/32}\).
Set \(g_d(x)\coloneqq \mathrm{e}^{-\alpha s(x)^2/(2d)}\) and
\(\boldsymbol Q_{d,\eta}\coloneqq \boldsymbol v\,(g_d-c)_+\).
On \(D_d\coloneqq \{g_d>c\}\), with \(t=s^2/d\),
{a direct calculation yields}
\begin{align*}
 \frac{([\mathcal L_d-(\Lambda_\xi+\eta)]
       [\boldsymbol v(g_d-c)])_i}{v_i}
 &\le
 \bigl[\alpha-\alpha(\kappa+\alpha)t\bigr]g_d
 -\frac{3\eta}{4}(g_d-c).
\end{align*}
If \(t\ge2/\kappa\), the right-hand side is nonpositive. If
\(t\le2/\kappa\), then
\(g_d\ge3/4\) {because} \(\alpha/\kappa\le\log(4/3)\), and the right-hand side is at most
\(\alpha-3\eta/16<0\). {Thus, the desired inequality holds in the classical sense on \(D_d\).}
If \(\xi\) is a domain endpoint, then
\(g_d'(\xi)=0\), {and hence} the Neumann condition is satisfied. The support
has diameter \(O(\sqrt d)\) and therefore lies in \(V^Q\) for
sufficiently small \(d\). Thus \(\boldsymbol Q_{d,\eta}\) is the required compactly
supported subsolution.
\end{proof}

\section{\bf Regular interval classes}
\label{sec:regular-intervals}

Throughout this section, the hypotheses and notation of
Section~\ref{sec:introduction-main} remain in force.  In particular,
all drift zeros are simple and the domain endpoints are
noncharacteristic.

\subsection{Endpoint geometry and the transport value}

The endpoint geometry, the transport problem, and its admissibility
conditions were specified in
Definition~\ref{reg:def:regular-interval}, and in particular
\[
 b_p(x)>0>b_n(x),\qquad a<x<b.
\]
Under the hypotheses of
Theorem~\ref{thm:main-one-dimensional}, every nondegenerate component
of \(\{b_1b_2\leq0\}\) without an interior common zero is regular.


We now prove that \(\Lambda_I^{\mathrm{reg}}\) in Definition~\ref{reg:def:regular-interval} is well defined. The
argument relies on the positive characteristic operators associated
with the two transport directions. Repelling-repelling intervals are
treated via the continuous adjoint flux problem, while for every other
endpoint combination, the primal one-sided operators can be composed
on a continuous component. Following the notation in Definition~\ref{reg:def:regular-interval}, we define
 \[
  \lambda_a\coloneqq
  \begin{cases}
  M_{pp}(a)+b_p'(a),&a\text{ is }\mathrm R,\\
   M_{nn}(a),&a\text{ is }\mathrm A\text{ or }\mathrm D,
  \end{cases}
  \qquad
  \lambda_b\coloneqq
  \begin{cases}
   M_{nn}(b)+b_n'(b),&b\text{ is }\mathrm R,\\
   M_{pp}(b),&b\text{ is }\mathrm A\text{ or }\mathrm D,
  \end{cases}
 \]
 and set $\lambda_{\partial,I}\coloneqq
 \min\{\lambda_a,\lambda_b\}$.
 Throughout this subsection, we set $B=-b_n$.
 For
$\lambda<\lambda_{\partial,I}$ and $a<x<b$, we define the primal characteristic operators
 as 
\begin{align}
 (\mathcal H_{n,\lambda}f)(x)
 &\coloneqq{}
 \int_a^x
 \exp\left[-\int_y^x
 \frac{M_{nn}(s)-\lambda}{-b_n(s)}\,\mathrm{d}s\right]
 \frac{-M_{np}(y)}{B(y)}f(y)\,\mathrm{d}y,
 \label{reg:eq:Hn}\\
 (\mathcal H_{p,\lambda}g)(x)
 &\coloneqq{}
 \int_x^b
 \exp\left[-\int_x^y
 \frac{M_{pp}(s)-\lambda}{b_p(s)}\,\mathrm{d}s\right]
 \frac{-M_{pn}(y)}{b_p(y)}g(y)\,\mathrm{d}y.
 \label{reg:eq:Hp}
\end{align}
If $a=\ell$ is a domain endpoint, the term
\begin{equation}\label{reg:eq:left-domain-characteristic-term}
 \frac{-M_{np}(\ell)}{M_{nn}(\ell)-\lambda}f(\ell)
 \exp\left[-\int_\ell^x
 \frac{M_{nn}(s)-\lambda}{-b_n(s)}\,\mathrm{d}s\right]
\end{equation}
is added to \eqref{reg:eq:Hn}. If $b=r$ is a domain endpoint, the term
\begin{equation}\label{reg:eq:right-domain-characteristic-term}
 \frac{-M_{pn}(r)}{M_{pp}(r)-\lambda}g(r)
 \exp\left[-\int_x^r
 \frac{M_{pp}(s)-\lambda}{b_p(s)}\,\mathrm{d}s\right]
\end{equation}
is added to \eqref{reg:eq:Hp}.  The endpoint values are defined by
continuity.  More precisely,
\begin{equation}\label{reg:eq:characteristic-endpoint-values}
 \begin{aligned}
 (\mathcal H_{n,\lambda}f)(a)
 &=
 \begin{cases}
 0,&\mathrm R,\\[1mm]
 \displaystyle\frac{-M_{np}(a)}{M_{nn}(a)-\lambda}f(a),
       &\mathrm A\text{ or }\mathrm D,
 \end{cases}\\[2mm]
 (\mathcal H_{p,\lambda}g)(b)
 &=
 \begin{cases}
 0,&\mathrm R,\\[1mm]
 \displaystyle\frac{-M_{pn}(b)}{M_{pp}(b)-\lambda}g(b),
       &\mathrm A\text{ or }\mathrm D.
 \end{cases}
 \end{aligned}
\end{equation}
At an attracting endpoint, the integral in \eqref{reg:eq:Hn} or
\eqref{reg:eq:Hp} is improper.  For example, at a left attracting
endpoint, set $h=x-a$, $\gamma=-b_n'(a)>0$, and
\[
 \beta(h)\coloneqq
 \frac{M_{nn}(a+h)-\lambda}{[-b_n(a+h)]/h},
 \qquad
 f_0(h)\coloneqq
 \frac{-M_{np}(a+h)}{[-b_n(a+h)]/h}f(a+h).
\]
Writing $\beta(h)=\beta_0+hr(h)$ and
$E(h)=\exp(\int_0^h r(s)\,\mathrm{d}s)$, we have
\begin{equation}\label{reg:eq:left-attracting-representation}
 (\mathcal H_{n,\lambda}f)(a+h)
 =\frac1{E(h)}
 \int_0^1t^{\beta_0-1}E(th)f_0(th)\,\mathrm{d}t,
\end{equation}
where
$\beta_0=[M_{nn}(a)-\lambda]/\gamma>0$. 
This formula gives the endpoint value in
\eqref{reg:eq:characteristic-endpoint-values}.  The right attracting
formula follows by reflection and interchange of $p$ and $n$.

When both endpoints are repelling, define, for $g,f\in C(I)$,
\begin{align}
 (\mathcal T_{p,\lambda}g)(x)
 &\coloneqq
 \frac1{b_p(x)}
 \int_a^x
 \exp\left[-\int_y^x
 \frac{M_{pp}(s)-\lambda}{b_p(s)}\,\mathrm{d}s\right]
 \bigl(-M_{np}(y)\bigr)g(y)\,\mathrm{d}y,
 &&a<x\leq b,
 \label{reg:eq:Tp}\\
 (\mathcal T_{n,\lambda}f)(x)
 &\coloneqq
 \frac1{B(x)}
 \int_x^b
 \exp\left[-\int_x^y
 \frac{M_{nn}(s)-\lambda}{-b_n(s)}\,\mathrm{d}s\right]
 \bigl(-M_{pn}(y)\bigr)f(y)\,\mathrm{d}y,
 &&a\leq x<b.
 \label{reg:eq:Tn}
\end{align}
Both expressions admit continuous extensions to the missing
endpoints.  We set
\begin{equation}\label{reg:eq:repelling-feedback}
 \mathcal K_{I,\lambda}
 \coloneqq
 \mathcal T_{p,\lambda}\mathcal T_{n,\lambda}
 \colon C(I)\to C(I).
\end{equation}

If the two endpoints are not both repelling, using
\eqref{reg:eq:Hn}-\eqref{reg:eq:right-domain-characteristic-term}, we define
\begin{equation*}
 \mathcal K_{I,\lambda}
 \coloneqq
 \begin{cases}
 \mathcal H_{p,\lambda}\mathcal H_{n,\lambda},
       &\text{if the left endpoint is not repelling},\\
 \mathcal H_{n,\lambda}\mathcal H_{p,\lambda},
       &\text{if the left endpoint is repelling and the right endpoint
       is not repelling}.
 \end{cases}
\end{equation*}
The composition acts on the component that remains continuous on
$I$.  
In every case, write
\begin{equation}\label{reg:eq:transport-spectral-radius}
 R_I(\lambda)\coloneqq r(\mathcal K_{I,\lambda}).
\end{equation}

\begin{lemma}
\label{reg:lem:characteristic-operators}
For every regular interval $I$ and every
$\lambda<\lambda_{\partial,I}$, the characteristic feedback operator
$\mathcal K_{I,\lambda}$ is compact and positive.  Its positive
spectral radius is simple.  Moreover, $R_I$ is continuous and
strictly increasing on $( -\infty,\lambda_{\partial,I})$, and
\[
 \lim_{\lambda\to-\infty}R_I(\lambda)=0,
 \qquad
 \lim_{\lambda\nearrow\lambda_{\partial,I}}R_I(\lambda)=+\infty.
\]
\end{lemma}

\begin{proof}
We first verify the characteristic formulations, their endpoint
convergence, and compactness.  In the repelling--repelling case, consider the adjoint problem
\begin{equation}\label{reg:eq:interval-adjoint}
 \begin{cases}
 (b_i v_i)'+\displaystyle\sum_{j=p,n}M_{ji}v_j
       =\lambda v_i,
       &a<x<b,\quad i=p,n,\\[1mm]
 (b_pv_p)(a)=0,\qquad (b_nv_n)(b)=0.
 \end{cases}
\end{equation}
Then by integrating the $p$-equation from $a$ to $x$ and the
$n$-equation from $x$ to $b$, we observe that $v_p=\mathcal T_{p,\lambda}v_n$ and $v_n=\mathcal T_{n,\lambda}v_p$. 
Moreover, the expansions $b_p(a+h)=b_p'(a)h+o(h)$ and $-b_n(b-h)=b_n'(b)h+o(h)$, 
together with $q_a(\lambda),q_b(\lambda)>-1$, {yield}
\[
 (\mathcal T_{p,\lambda}g)(a)
 =
 \frac{-M_{np}(a)}
 {M_{pp}(a)+b_p'(a)-\lambda}\,g(a),
 \qquad
 (\mathcal T_{n,\lambda}f)(b)
 =
 \frac{-M_{pn}(b)}
 {M_{nn}(b)+b_n'(b)-\lambda}\,f(b).
\]
The denominators are positive, so both characteristic maps extend
continuously to $I$. Conversely, differentiating the two integral
identities recovers \eqref{reg:eq:interval-adjoint}. {Hence} \eqref{reg:eq:interval-adjoint} has a
nontrivial nonnegative solution if and only if $v_p$ is a fixed
point of $\mathcal K_{I,\lambda}$, with
$v_n=\mathcal T_{n,\lambda}v_p$.

At a right attracting endpoint, write
$b_p(x)=\gamma_b(b-x)+O((b-x)^2)$, where
$\gamma_b=-b_p'(b)>0$, and set
$\nu_b=[M_{pp}(b)-\lambda]/\gamma_b>0$.  With $t=b-x$ and
$y=b-tz$, formula~\eqref{reg:eq:Hp} becomes
\begin{equation}\label{reg:eq:right-attracting-rescaling}
 \begin{aligned}
 (\mathcal H_{p,\lambda}g)(b-t)
 &={}
 \int_0^1
 \exp\left[-\int_{b-t}^{b-tz}
 \frac{M_{pp}(s)-\lambda}{b_p(s)}\,\mathrm{d}s\right]
 \frac{-tM_{pn}(b-tz)}{b_p(b-tz)}g(b-tz)\,\mathrm{d}z\\
 &={}
 \int_0^1
 \left[z^{\nu_b}(1+O(t))\right]
 \left[\frac{-M_{pn}(b)}{\gamma_bz}+O(t)\right]
 g(b-tz)\,\mathrm{d}z,
 \end{aligned}
\end{equation}
which is uniform for $0<z\leq1$.  Since $\nu_b>0$, the
integrand is dominated by an integrable multiple of
$z^{\nu_b-1}$.  {By} the dominated convergence theorem, we have
\[
 \lim_{x\nearrow b}(\mathcal H_{p,\lambda}g)(x)
 =\frac{-M_{pn}(b)}{M_{pp}(b)-\lambda}g(b).
\]
The analogous verification at a left attracting endpoint {yields}
\eqref{reg:eq:left-attracting-representation}.

At a left repelling endpoint, set
$q_a(\lambda)=[M_{pp}(a)-\lambda]/b_p'(a)>-1$.  By the change of
variables $y=a+z(x-a)$ in \eqref{reg:eq:Tp}, we see that its singular
part is a bounded perturbation of
\[
 \frac{-M_{np}(a)}{b_p'(a)}
 \int_0^1z^{q_a(\lambda)}g(a+z(x-a))\,\mathrm{d}z.
\]
It therefore has the continuous endpoint value in
\eqref{reg:eq:Tp}.  The same calculation at a right repelling endpoint
applies to \eqref{reg:eq:Tn}, with
$q_b(\lambda)=[M_{nn}(b)-\lambda]/b_n'(b)>-1$.  For a 
feedback operator with one
repelling endpoint, the corresponding power is locally integrable and
$\mathcal K_{I,\lambda}$ has a continuous
endpoint trace.  The domain endpoint terms
\eqref{reg:eq:left-domain-characteristic-term} and
\eqref{reg:eq:right-domain-characteristic-term} are regular.

We next prove the compactness of $\mathcal K_{I,\lambda}$.  Let
$\mathcal B$ be the unit ball of $C(I)$.  Suppose first that both
endpoints are repelling.  For $f\in\mathcal B$, set
$g=\mathcal T_{n,\lambda}f$ and
$h=\mathcal T_{p,\lambda}g$.  By \eqref{reg:eq:Tn},
$\|g\|_\infty\leq C_\lambda$.  Since $-b_n$ is bounded away from
zero near $a$, it follows that
$\|g\|_{W^{1,\infty}(a,a+\delta)}\leq C_\lambda$.  Moreover,
\begin{equation}\label{reg:eq:Tp-compactness-ode}
 b_ph'+(b_p'+M_{pp}-\lambda)h=-M_{np}g.
\end{equation}
The endpoint identity {reads}
$
 \bigl(b_p'(a)+M_{pp}(a)-\lambda\bigr)h(a)
 =-M_{np}(a)g(a).
$
Subtracting this  from
\eqref{reg:eq:Tp-compactness-ode}, and using
$b_p'(a)+M_{pp}(a)-\lambda>0$, {yields}
\[
 |h(x)-h(a)|+(x-a)|h'(x)|
 \leq C_\lambda(x-a),\qquad a<x<a+\delta.
\]
Away from $a$, equation~\eqref{reg:eq:Tp-compactness-ode} directly
bounds $h'$.  {Hence} $\mathcal K_{I,\lambda}$ maps
$\mathcal B$ into a bounded subset of $W^{1,\infty}(I)$, and is
compact by the Arzel\`a--Ascoli theorem.
It remains to consider the cases in which the endpoints are not both
repelling.  Suppose first that the left endpoint is not repelling, so that
$\mathcal K_{I,\lambda}
=\mathcal H_{p,\lambda}\mathcal H_{n,\lambda}$.  For
$f\in\mathcal B$, set
$g_f=\mathcal H_{n,\lambda}f$.  On each compact subinterval of
$[a,b)$, the family $\{g_f:f\in\mathcal B\}$ is uniformly
bounded.  At a left attracting endpoint, this follows from
\eqref{reg:eq:left-attracting-representation}. At a left domain
endpoint, it follows from the regular boundary term
\eqref{reg:eq:left-domain-characteristic-term}. 
Assume that the right endpoint is attracting.  Since $b_n(b)<0$, there
exist $\delta,c_0>0$ such that
$-b_n(x)\geq c_0$ for $b-\delta\leq x\leq b$.  Differentiating
\eqref{reg:eq:Hn} {yields}
\[
 -b_n(\mathcal H_{n,\lambda}f)'
 +(M_{nn}-\lambda)\mathcal H_{n,\lambda}f=-M_{np}f.
\]
It follows that
$
 \sup_{f\in\mathcal B}
 \|(\mathcal H_{n,\lambda}f)'\|_{L^\infty(b-\delta,b)}
 <\infty.
$ 
{Thus} $\{g_f:f\in\mathcal B\}$ is uniformly Lipschitz near $b$.
Using \eqref{reg:eq:right-attracting-rescaling}, with $t=b-x$, we
obtain
\[
 \begin{aligned}
 &\left|
 (\mathcal K_{I,\lambda}f)(b-t)
 -(\mathcal K_{I,\lambda}f)(b)
 \right|\\
 &\qquad\leq
 C\int_0^1z^{\nu_b-1}
 |g_f(b-tz)-g_f(b)|\,\mathrm{d}z+Ct
 =O(t),
 \end{aligned}
\]
uniformly for $f\in\mathcal B$.  {Hence}
$\{\mathcal K_{I,\lambda}f:f\in\mathcal B\}$ is equicontinuous at
$b$.
If the endpoint is a domain endpoint, the verification is similar.

We next establish positivity and the spectral properties.  Every
kernel in \eqref{reg:eq:Hn}, \eqref{reg:eq:Hp},
\eqref{reg:eq:Tp}, and \eqref{reg:eq:Tn} is positive.  In the
repelling-repelling case, if $0\leq f\not\equiv0$, then
$\mathcal K_{I,\lambda}f\gg0$ on $I$.  {Thus} the operator is
strongly positive.  In every other case, its integral kernel is
strictly positive for $x,y\in(a,b)$. A possible zero at a repelling
endpoint is forced only by the prescribed incoming trace.  The
standard Perron argument for a compact positive integral operator,
or equivalently the Krein--Rutman theorem on its invariant principal
ideal \cite{KR1950}, shows that $R_I(\lambda)>0$ is simple and has an
eigenfunction strictly positive in $(a,b)$.

Let $\lambda_1<\lambda_2<\lambda_{\partial,I}$.  Inspection of the
four characteristic formulas and the two domain endpoint terms
shows that increasing $\lambda$ strictly increases every positive
kernel and endpoint factor.  {Hence} the positive kernel of
$\mathcal K_{I,\lambda_2}$ strictly dominates that of
$\mathcal K_{I,\lambda_1}$ in the interior.  The
Collatz-Wielandt inequalities {therefore yield}
$R_I(\lambda_1)<R_I(\lambda_2)$.  {Consequently,}
\[
 \|\mathcal K_{I,\lambda'}-\mathcal K_{I,\lambda}\|_{C(I)\to C(I)}
 \to0
 \qquad\text{as }\lambda'\to\lambda.
\]
The simplicity and isolation of the positive eigenvalue {yield} the
continuity of $R_I$.

It remains to determine the limits of the spectral radius.  For
$\lambda\ll-1$, the zeroth-order coefficients in the first-order
equations defining the characteristic maps are bounded below by
$|\lambda|/2$.  By the maximum principle, we arrive at
\[
 \|\mathcal K_{I,\lambda}\|_{C(I)\to C(I)}
 \leq\frac{C}{|\lambda|^2}.
\]
{Consequently,} $R_I(\lambda)\to0$ as $\lambda\to-\infty$.

We finally prove that
$R_I(\lambda)\to+\infty$ as
$\lambda\nearrow\lambda_{\partial,I}$.  We first assume that 
$\lambda_{\partial,I}=\lambda_a$. 
Suppose that the left endpoint is attracting or is a domain endpoint.
In either
case, $\lambda_a=M_{nn}(a)$ and
$\mathcal K_{I,\lambda}
 =\mathcal H_{p,\lambda}\mathcal H_{n,\lambda}$.
Choose $a<c<d<b$, with $c$ and $d$ close to
$a$, and a function
$\chi\in C(I)$ such that $0\leq\chi\leq1$,
$\chi=1$ near $a$, and $\chi=0$ on $[c,b]$.
At an attracting endpoint, by
\eqref{reg:eq:left-attracting-representation},
\[
 \inf_{x\in[c,d]}
 (\mathcal H_{n,\lambda}\chi)(x)
 \geq \frac{C}{M_{nn}(a)-\lambda}
\]
for $\lambda$ sufficiently close to $M_{nn}(a)$.  
At a domain endpoint, the same lower bound follows directly from
the boundary term
\eqref{reg:eq:left-domain-characteristic-term}. Its exponential
factor is bounded below on $[c,d]$, uniformly for $\lambda$ near
$M_{nn}(a)$.
The kernel of $\mathcal H_{p,\lambda}$ is bounded below by a
positive constant for $a\leq x\leq c$ and $c\leq y\leq d$.
{Hence} $ (\mathcal K_{I,\lambda}\chi)(x)
 \geq \frac{C}{M_{nn}(a)-\lambda}$ for $a\leq x\leq c$.
On $[c,b]$, we have $\chi=0$ and
$\mathcal K_{I,\lambda}\chi\geq0$.  {It therefore follows} that $ \mathcal K_{I,\lambda}\chi
 \geq \frac{C}{M_{nn}(a)-\lambda}\chi$ on $I$, from which 
by iteration {we deduce} that 
$
 \mathcal K_{I,\lambda}^{\,m}\chi
 \geq
 \left(\frac{C}{M_{nn}(a)-\lambda}\right)^m\chi
$ for all $m\geq1$. 
The spectral-radius formula {thus gives}
\[
 R_I(\lambda)
 \geq\frac{C}{M_{nn}(a)-\lambda}
 \to+\infty.
\]

Suppose next that the left endpoint is repelling.  Then
$\lambda_a=M_{pp}(a)+b_p'(a)$, and
$q_a(\lambda)+1\to0$.  If the right endpoint is also repelling, the
argument in the original repelling--repelling construction applies
directly.  For $\lambda$ in a fixed left neighborhood of
$\lambda_a$,
$\mathcal T_{n,\lambda}\mathbf 1
 \geq c_0\mathbf 1$.  The rescaled kernel in
\eqref{reg:eq:Tp} then {yields}
$
 \inf_{x\in I}
 (\mathcal K_{I,\lambda}\mathbf 1)(x)
 \geq\frac{c_1}{q_a(\lambda)+1}-c_2.
$
{It follows by iteration} that
\[
 R_I(\lambda)
 \geq\frac{c_1}{q_a(\lambda)+1}-c_2
 \to+\infty.
\]

It remains to consider a left repelling endpoint whose right endpoint
is not repelling.  In this case,
$\mathcal K_{I,\lambda}
 =\mathcal H_{n,\lambda}\mathcal H_{p,\lambda}$.
Choose $\delta>0$ and
$\chi\in C(I)$, $0\leq\chi\leq1$, supported in
$[a+2\delta,b-\delta]$ and equal to one on a nondegenerate
subinterval.  Since
\[
 \frac{M_{pp}(a+h)-\lambda}{b_p(a+h)}
 =\frac{q_a(\lambda)}h+O(1)
 \qquad\text{as }h\searrow0,
\]
we can apply formula~\eqref{reg:eq:Hp}, restricted to the subinterval on which
$\chi=1$, to obtain 
\[
 (\mathcal H_{p,\lambda}\chi)(a+h)
 \geq c h^{q_a(\lambda)},
 \qquad 0<h<\delta,
\]
where $c>0$ is independent of $\lambda$ near $\lambda_a$.
The kernel of $\mathcal H_{n,\lambda}$ is uniformly bounded below
when $a<y<a+\delta$ and
$x\in\operatorname{supp}\chi$.  {Hence}
\[
 \begin{aligned}
 (\mathcal K_{I,\lambda}\chi)(x)
 \geq c\int_0^\delta h^{q_a(\lambda)}\,\mathrm{d}h=\frac{c\,\delta^{q_a(\lambda)+1}}
         {q_a(\lambda)+1},
 \qquad x\in\operatorname{supp}\chi.
 \end{aligned}
\]
Outside $\operatorname{supp}\chi$, the right-hand side of
$
 \mathcal K_{I,\lambda}\chi
 \geq
 \frac{c\,\delta^{q_a(\lambda)+1}}
      {q_a(\lambda)+1}\chi
$
vanishes.  Iteration and the spectral-radius formula {now show} that
$
 R_I(\lambda)
 \geq
 \frac{c\,\delta^{q_a(\lambda)+1}}
      {q_a(\lambda)+1}
 \to+\infty.
$

If the right endpoint ceiling is active, the same arguments apply
after reflecting the interval and interchanging $p$ and $n$.
They {yield} the corresponding bounds with
$M_{pp}(b)-\lambda$ at an attracting or domain endpoint and with
$q_b(\lambda)+1$ at a repelling endpoint.  When both endpoint ceilings
are active, either one of the preceding estimates suffices.  This
completes the proof.
\end{proof}

\begin{lemma}
\label{reg:lem:transport-value}
For every regular interval $I$, there is a unique
$\Lambda_I^{\mathrm{reg}}<\lambda_{\partial,I}$ such that
$R_I(\Lambda_I^{\mathrm{reg}})=1$.  Moreover, this is precisely the unique
value for which \eqref{reg:eq:transport-system} admits a positive
admissible profile that satisfies
\eqref{reg:eq:endpoint-conditions}. Such a
profile is unique up to multiplication by a positive constant.
\end{lemma}

\begin{proof}

We first determine the spectral parameter and construct the primal
profile when the two endpoints are not both repelling. By
Lemma~\ref{reg:lem:characteristic-operators}, the continuity and strict
monotonicity of the spectral radius in
\eqref{reg:eq:transport-spectral-radius}, together with its two
limiting values, ensure the existence of a
unique \(\Lambda_I^{\mathrm{reg}}<\lambda_{\partial,I}\) such that
\(R_I(\Lambda_I^{\mathrm{reg}})=1\).
Suppose that the left endpoint is not repelling. Let \(u_p\) be the
positive eigenfunction of
\(\mathcal H_{p,\Lambda_I^{\mathrm{reg}}}
 \mathcal H_{n,\Lambda_I^{\mathrm{reg}}}\) at the eigenvalue \(1\),
and {define}
\(u_n=\mathcal H_{n,\Lambda_I^{\mathrm{reg}}}u_p\). If the left
endpoint is repelling and the right endpoint is not, {we instead consider}
the positive eigenfunction \(u_n\) of
\(\mathcal H_{n,\Lambda_I^{\mathrm{reg}}}
 \mathcal H_{p,\Lambda_I^{\mathrm{reg}}}\) and {define}
\(u_p=\mathcal H_{p,\Lambda_I^{\mathrm{reg}}}u_n\). In both cases,
the two characteristic identities prove
\eqref{reg:eq:transport-system}. The endpoint limits in
\eqref{reg:eq:characteristic-endpoint-values}, together with the two
domain endpoint terms, {precisely yield}
\eqref{reg:eq:endpoint-conditions}.
The kernels are strictly positive in the interior. {Consequently,}
\(\boldsymbol u\gg0\) on \((a,b)\). At a repelling endpoint, the
intermediate component has exponent greater than \(-1\), {yielding}
\(\boldsymbol u\in L^1((a,b);\mathbb R^2)\).

We next establish the primal profile when both endpoints are repelling.
Let
\(\mathcal K=\mathcal K_{I,\Lambda_I^{\mathrm{reg}}}
=\mathcal T_{p,\Lambda_I^{\mathrm{reg}}}
 \mathcal T_{n,\Lambda_I^{\mathrm{reg}}}\).
By Lemma~\ref{reg:lem:characteristic-operators}, \(\mathcal K\) is
compact and strongly positive on \(C(I)\), and
\(r(\mathcal K)=1\). The Krein--Rutman theorem and its dual form
\cite{KR1950} imply that \(1\) is algebraically simple and that
\(\mathcal K^*\) has a strictly positive eigenfunctional \(\mu\).

Since \(\mathcal K\) possesses a kernel with respect to Lebesgue
measure, {\(\mathcal K^*\) maps finite measures to absolutely continuous measures. Furthermore, the relation} \(\mu=\mathcal K^*\mu\) {implies that}
\(\mu=w\,\mathrm{d}x\) for some \(w>0\) almost everywhere. {Define}
\(z\,\mathrm{d}x\coloneqq
\mathcal T_{p,\Lambda_I^{\mathrm{reg}}}^*(w\,\mathrm{d}x)\) and set
\[
 u_n\coloneqq\frac{w}{-M_{pn}},
 \qquad
 u_p\coloneqq\frac{z}{-M_{np}}.
\]
Since
\(\mathcal T_{n,\Lambda_I^{\mathrm{reg}}}^*
 (z\,\mathrm{d}x)=w\,\mathrm{d}x\), {applying Fubini's theorem, we obtain}
\begin{align}
 u_n(x)
 &={}
 \int_a^x
 \exp\left[-\int_y^x
 \frac{M_{nn}(s)-\Lambda_I^{\mathrm{reg}}}{-b_n(s)}\,\mathrm{d}s\right]
 \frac{-M_{np}(y)}{B(y)}u_p(y)\,\mathrm{d}y,
 \label{reg:eq:primal-volterra-n}\\
 u_p(x)
 &={}
 \int_x^b
 \exp\left[-\int_x^y
 \frac{M_{pp}(s)-\Lambda_I^{\mathrm{reg}}}{b_p(s)}\,\mathrm{d}s\right]
 \frac{-M_{pn}(y)}{b_p(y)}u_n(y)\,\mathrm{d}y.
 \label{reg:eq:primal-volterra-p}
\end{align}
From \eqref{reg:eq:primal-volterra-n} and
\eqref{reg:eq:primal-volterra-p} we deduce {that} \(\boldsymbol u\gg0\), that the incoming traces
satisfy \(u_n(a)=u_p(b)=0\), and that equation~\eqref{reg:eq:transport-system} holds upon differentiation on compact
subintervals. The endpoint
rescaling in the proof of
Lemma~\ref{reg:lem:characteristic-operators} ensures local integrability.

It remains to prove the converse and uniqueness. If the two endpoints
are not both repelling, every positive admissible profile satisfies the
two identities associated with \(\mathcal H_{n,\lambda}\) and
\(\mathcal H_{p,\lambda}\). Its continuous component is therefore a
positive fixed point of \(\mathcal K_{I,\lambda}\). If both endpoints
are repelling, Fubini's theorem implies that
\((-M_{pn})u_n\,\mathrm{d}x\) is a positive eigenfunctional of
\(\mathcal K_{I,\lambda}^*\) at the eigenvalue \(1\). In either case,
simplicity of the positive spectral radius yields
\(R_I(\lambda)=1\) and proves uniqueness of the profile up to a
positive factor.

Finally, an admissible positive profile necessarily has
\(\lambda<\lambda_a\) and \(\lambda<\lambda_b\). At an attracting or
domain endpoint, this follows immediately from the positive
algebraic trace in
\eqref{reg:eq:characteristic-endpoint-values}. At a repelling endpoint,
the regular-singular expansion in
Lemma~\ref{reg:lem:repelling-endpoint-asymptotics} shows that local
integrability is equivalent to
\([M_{pp}(a)-\lambda]/b_p'(a)>-1\) on the left or
\([M_{nn}(b)-\lambda]/b_n'(b)>-1\) on the right. Hence
\(\lambda<\lambda_{\partial,I}\), and the strict monotonicity of
\(R_I\) forces \(\lambda=\Lambda_I^{\mathrm{reg}}\). The proof is
complete.
\end{proof}

The endpoint powers used above also yield the matching data for the
viscous endpoint local models.  We record them explicitly.

\begin{lemma}
\label{reg:lem:repelling-endpoint-asymptotics}
Let \(\lambda\in\mathbb R\), and let
\(\boldsymbol u\in C^1_{\mathrm{loc}}((a,b);\mathbb R^2)
\cap L^1((a,b);\mathbb R^2)\) be a positive solution of
\eqref{reg:eq:transport-system}.  At a left repelling endpoint, assume
that \(u_n(a)=0\) in the trace sense and set
\(q_a(\lambda)=[M_{pp}(a)-\lambda]/b_p'(a)\).  Then
\(q_a(\lambda)>-1\), and there exists \(C_a>0\) such that, as \(h\searrow0\),
\begin{equation}\label{reg:eq:left-repelling-asymptotics}
 \begin{aligned}
 u_p(a+h)&=C_ah^{q_a(\lambda)}\bigl(1+O(h)\bigr),\\
 u_n(a+h)&=C_a
 \frac{-M_{np}(a)}{[-b_n(a)](q_a(\lambda)+1)}
 h^{q_a(\lambda)+1}\bigl(1+O(h)\bigr),\\
 u_p'(a+h)&=C_aq_a(\lambda)h^{q_a(\lambda)-1}
 +O(h^{q_a(\lambda)}),\\
 u_n'(a+h)&=C_a\frac{-M_{np}(a)}{-b_n(a)}
 h^{q_a(\lambda)}\bigl(1+O(h)\bigr).
 \end{aligned}
\end{equation}
At a right repelling endpoint, assume that \(u_p(b)=0\) in the trace
sense and set
\(q_b(\lambda)=[M_{nn}(b)-\lambda]/b_n'(b)\).  Then
\(q_b(\lambda)>-1\), and there exists \(C_b>0\) such that, as \(h\searrow0\),
\begin{equation}\label{reg:eq:right-repelling-asymptotics}
 \begin{aligned}
 u_n(b-h)&=C_bh^{q_b(\lambda)}\bigl(1+O(h)\bigr),\\
 u_p(b-h)&=C_b
 \frac{-M_{pn}(b)}{b_p(b)(q_b(\lambda)+1)}
 h^{q_b(\lambda)+1}\bigl(1+O(h)\bigr),\\
 u_n'(b-h)&=-C_bq_b(\lambda)h^{q_b(\lambda)-1}
 +O(h^{q_b(\lambda)}),\\
 u_p'(b-h)&=-C_b\frac{-M_{pn}(b)}{b_p(b)}
 h^{q_b(\lambda)}\bigl(1+O(h)\bigr).
 \end{aligned}
\end{equation}
At either endpoint, the positive locally integrable solution ray
that satisfies the incoming trace is unique up to multiplication by a
positive constant.
\end{lemma}

\begin{proof}
We first show that integrability forces the exponent at a left
repelling endpoint to
exceed \(-1\), and then establish the expansion. Set \(h=x-a\) and
\(q=q_a(\lambda)\). Since \(-b_n\) is bounded away from zero near
\(a\), the \(n\)-equation and local integrability guarantee
that \(u_n\) admits the prescribed finite trace. Fix a small \(h_0>0\), and define
\[
 \mu(h)\coloneqq
 \exp\left[-\int_{h_0}^{h}
 \frac{M_{pp}(a+s)-\lambda}{b_p(a+s)}\,\mathrm{d}s\right].
\]
Then \(\mu(h)=c h^{-q}(1+O(h))\) for some \(c>0\), {whereas} the first
equation {yields}
\[
 \frac{\mathrm{d}}{\mathrm{d}h}
 \bigl(\mu(h)u_p(a+h)\bigr)
 =\mu(h)\frac{M_{pn}(a+h)}{b_p(a+h)}u_n(a+h)<0.
\]
Then \(u_p(a+h)\geq c_0h^q\) on \((0,h_0]\) for some
\(c_0>0\). Local integrability thus requires \(q>-1\).

{We now write}
\(u_p(a+h)=h^qA(h)\) and \(u_n(a+h)=h^{q+1}B_0(h)\). By the variation of
constants in the \(n\)-equation, with \(u_n(a)=0\), we {obtain}
\begin{align}
 B(h)
 &=\int_0^1 K_a(h,r)r^{q_a}A(hr)\,{\rm d}r,
 \label{reg:eq:left-repelling-volterra}\\
 \frac{b_p(a+h)}{h}A'(h)
 &=c_a(h)A(h)+M_{pn}(a+h)B(h),
 \label{reg:eq:left-repelling-weighted-p}
\end{align}
where
\[
\begin{aligned}
 K_a(h,r)
 &:={}
 \exp\left[-\int_{hr}^{h}
 \frac{M_{nn}(a+s)-\Lambda_I}{-b_n(a+s)}\,{\rm d}s\right]
 \frac{-M_{np}(a+hr)}{-b_n(a+hr)},\\
 c_a(h)
 &:={}
 \frac{M_{pp}(a+h)-\Lambda_I
       -q_ab_p(a+h)/h}{h}.
\end{aligned}
\]
The quotient \(b_p(a+h)/h\) extends continuously to \(h=0\), where
its value is \(b_p'(a)>0\). The coefficient of \(A(h)\) in
\eqref{reg:eq:left-repelling-weighted-p} also has a removable singularity
at \(h=0\), since \(M_{pp}(a)-\lambda=qb_p'(a)\). The remaining
coefficients in \eqref{reg:eq:left-repelling-volterra} are continuous for
\(0\leq hs\leq h\leq h_0\), whereas \(s^q\in L^1(0,1)\).

Substituting \eqref{reg:eq:left-repelling-volterra} into
\eqref{reg:eq:left-repelling-weighted-p} and integrating in \(h\), we obtain a nonsingular Volterra equation of the second kind for
\(A\). Conversely, variation of constants in the \(n\)-equation and
removal of the regular-singular integrating factor in the
\(p\)-equation {show that} every locally integrable solution {is a}
solution of this Volterra equation. Picard iteration shows that its
continuous solution is uniquely determined by \(A(0)\). The positive
admissible solution therefore has
\(A,B\in C^1([0,h_0])\). If \(A(0)=0\), the uniqueness would yield the
zero solution near \(a\), which contradicts the positivity. Hence
\(C_a=A(0)>0\). 
Applying the dominated convergence theorem {to}
\eqref{reg:eq:left-repelling-volterra}, we {obtain} 
\[
 B(0)
 =C_a\frac{-M_{np}(a)}{-b_n(a)}
   \int_0^1r^{q_a}\,{\rm d}r
 =C_a\frac{-M_{np}(a)}{[-b_n(a)](q_a+1)}.
\]
{It then follows} from \eqref{reg:eq:left-repelling-volterra} and
\eqref{reg:eq:left-repelling-weighted-p} that
\(A(h)=C_a+O(h)\) and \(B_0(h)=B_0(0)+O(h)\). Moreover, we have
\(A'(h),B_0'(h)=O(1)\). Substituting them into the weighted components
{establishes} \eqref{reg:eq:left-repelling-asymptotics}.

It remains to treat a right repelling endpoint. Set \(h=b-x\) and
interchange the roles of \(p\) and \(n\). The reflected integrating
factor argument first shows that local integrability forces
\(q_b(\lambda)>-1\). The same Volterra
calculation applies, with \(b_p(a+h)\) replaced by \(-b_n(b-h)\) and
\(-b_n(a+h)\) replaced by \(b_p(b-h)\). Since
\(\mathrm{d}/\mathrm{d}x=-\mathrm{d}/\mathrm{d}h\), the two derivative
signs in \eqref{reg:eq:right-repelling-asymptotics} are negative. The
calculation establishes the right-hand expansion and uniqueness of the
right local {solution}. The proof is complete.
\end{proof}

\subsection{Three endpoint layers}

To realize \(\Lambda_I^{\mathrm{reg}}\) as a local effective value, we require
separate layers at repelling, attracting, and domain endpoints.
Throughout this construction, \(\mathcal L_d\) is the operator defined
in Section~\ref{sec:HJ-Aubry-selection}.

The repelling-endpoint layer is constructed from the following explicit
scalar profile, whose two asymptotic regimes yield the inner and outer
matching data.

\begin{lemma}
\label{reg:lem:repelling-profile}
Let $q>-1$ and $\alpha>0$, and let
\[
 Z_{q,\alpha}(y)
 \coloneqq
 \int_0^\infty
 t^q
 \exp\left[-\frac{\alpha}{2}(t-y)^2\right]\,\mathrm{d}t.
\]
Then $Z_{q,\alpha}\in C^\infty(\mathbb R)$ is strictly positive and
satisfies
\begin{equation}\label{reg:eq:Z-ode}
 -Z_{q,\alpha}''
 -\alpha yZ_{q,\alpha}'
 +\alpha qZ_{q,\alpha}
 =0
 \qquad\text{in }\mathbb R.
\end{equation}
Moreover, the following asymptotic behaviors hold:
\begin{align}
 Z_{q,\alpha}(y)
 &=
 \sqrt{\frac{2\pi}{\alpha}}\,
 y^q\bigl(1+O(y^{-2})\bigr),
 &&y\to+\infty,
 \label{reg:eq:Z-plus}\\
 Z_{q,\alpha}(y)
 &=
 \frac{\Gamma(q+1)}
      {(\alpha|y|)^{q+1}}
 \exp\left(-\frac{\alpha y^2}{2}\right)
 \bigl(1+O(|y|^{-2})\bigr),
 &&y\to-\infty.
 \label{reg:eq:Z-minus}
\end{align}
\end{lemma}
\begin{proof}
Since \(q>-1\), the defining integrand is locally integrable at
\(t=0\). On every compact \(y\)-interval, its derivatives with respect
to \(y\) are bounded by
\(C(1+t)^N t^q\exp(-ct^2)\) for suitable \(C,c>0\) and
\(N\in\mathbb N\). Differentiation under the integral sign is therefore
valid to every order. Hence \(Z_{q,\alpha}\in C^\infty(\mathbb R)\).
Its strict positivity is immediate.

Write \(Z=Z_{q,\alpha}\). Direct differentiation yields
\[
 \begin{aligned}
 Z'(y)
 &=\alpha\int_0^\infty t^q(t-y)
   \exp\left(-\frac{\alpha}{2}(t-y)^2\right)\,\mathrm{d}t,\\
 Z''(y)
 &=\alpha^2\int_0^\infty t^q(t-y)^2
   \exp\left(-\frac{\alpha}{2}(t-y)^2\right)\,\mathrm{d}t-\alpha Z(y).
 \end{aligned}
\]
Since \(q+1>0\), the boundary terms vanish in the integration by
parts
\[
 0=\int_0^\infty\frac{\mathrm{d}}{\mathrm{d}t}
 \left[t^{q+1}
 \exp\left(-\frac{\alpha}{2}(t-y)^2\right)\right]\,\mathrm{d}t.
\]
Hence, we derive that
\[
 \alpha\int_0^\infty t^{q+1}(t-y)
 \exp\left(-\frac{\alpha}{2}(t-y)^2\right)\,\mathrm{d}t
 =(q+1)Z(y).
\]
Combining \(t(t-y)=(t-y)^2+y(t-y)\) with the derivative formulas
above yields \eqref{reg:eq:Z-ode}.

For \(y\to+\infty\), the substitution \(s=t-y\) yields
\[
 Z(y)=\int_{-y}^{\infty}(y+s)^q
       \exp\left(-\frac{\alpha s^2}{2}\right)\,\mathrm{d}s.
\]
Note that on \(|s|\leq y/2\), 
\[
 (y+s)^q
 =y^q\left(1+\frac{qs}{y}
       +O\left(\frac{s^2}{y^2}\right)\right),
\]
uniformly in \(s\). The odd term integrates to zero, while Gaussian
tail estimates {yield}
\[
 \int_{[-y,-y/2]\cup[y/2,\infty)}
 (y+s)^q\exp\left(-\frac{\alpha s^2}{2}\right)\,\mathrm{d}s
 =O\left((1+y)^{|q|+1}\exp(-cy^2)\right)
\]
for some \(c>0\). Consequently, we obtain
\[
 Z(y)=y^q\left(\sqrt{\frac{2\pi}{\alpha}}+O(y^{-2})\right),
\]
which establishes \eqref{reg:eq:Z-plus}.
Finally, let \(y<0\). Under the substitution \(z=-\alpha yt\),
\[
 Z(y)=
 \frac{\exp(-\alpha y^2/2)}{(-\alpha y)^{q+1}}
 \int_0^\infty z^q\exp(-z)
 \exp\left(-\frac{z^2}{2\alpha y^2}\right)\,\mathrm{d}z.
\]
The inequality \(0\leq1-\exp(-r)\leq r\) for \(r\geq0\) yields
\[
 0\leq
 \Gamma(q+1)-
 \int_0^\infty z^q\exp(-z)
 \exp\left(-\frac{z^2}{2\alpha y^2}\right)\,\mathrm{d}z
 \leq\frac{\Gamma(q+3)}{2\alpha y^2}.
\]
Thus, the integral equals
\(\Gamma(q+1)(1+O(y^{-2}))\), and
\eqref{reg:eq:Z-minus} follows.
\end{proof}

\begin{lemma}
\label{reg:lem:repelling-module}
Let \(a\) be a left repelling endpoint, and let \(\Lambda=\Lambda_I^{\mathrm{reg}}\) and \(\boldsymbol u\) denote the
positive transport solution given by Lemma {\rm\ref{reg:lem:transport-value}}. Set
\(q=[M_{pp}(a)-\Lambda]/b_p'(a)>-1\). For
\(\delta_d=d^{2/5}\), there exist sequences \(\Delta_d\searrow0\) with
\(\delta_d\ll\Delta_d\) and positive functions
\(\boldsymbol F_d^a\in C^2([a-3\delta_d,a+2\Delta_d];\mathbb{R}^2)\) such that
\begin{align}\label{reg:eq:repelling-residual}
 \boldsymbol F_d^a \equiv \boldsymbol u
    \text{ near }a+2\Delta_d,\qquad 
 \max_{i=p,n}\sup_{a-3\delta_d\leq x\leq a+2\Delta_d}
 \frac{|[(\mathcal L_d-\Lambda)\boldsymbol F_d^a]_i|}
      {F_{d,i}^a}\to0.
\end{align}
Furthermore, under a common scalar normalization, we have uniformly for \(i=p,n\),
\begin{equation}\label{reg:eq:repelling-phase}
 -d\log F_{d,i}^a(x)
 =\int_a^x b_p(s)\,\mathrm{d}s+o(\delta_d^2)
 \quad(a-3\delta_d\leq x\leq a-\delta_d).
\end{equation}
The {analogous} module at a right repelling endpoint is obtained by {reflecting}
the coordinate and interchanging \(p\) and \(n\).
\end{lemma}

\begin{proof}
Throughout the proof, we fix the constants \(0<\vartheta<\frac{q+1}{q+2}\), \(h_d=d^{7/10}\), and \(\Delta_d=\delta_d^\vartheta\).
This ensures \(h_d\ll\delta_d\ll\Delta_d\ll1\).

\medskip
\noindent\emph{Step 1.}
We first construct the corrected layer near the repelling endpoint.
Set
\[
 \alpha\coloneqq b_p'(a),\qquad c\coloneqq-b_n(a)>0,
 \qquad y\coloneqq\frac{x-a}{\sqrt d},
 \qquad R_d\coloneqq\frac{\delta_d}{\sqrt d}=d^{-1/10}.
\]
We abbreviate $Z\coloneqq Z_{q,\alpha}$ and set
\[
 Y(y)\coloneqq
 \frac{-M_{np}(a)}{c(q+1)}Z_{q+1,\alpha}(y).
\]
By differentiating the integral defining $Z_{q+1,\alpha}$ and
integrating by parts, we obtain the identity
\[
 Z_{q+1,\alpha}'=(q+1)Z_{q,\alpha}.
\]
Hence, \(Y'=k_0Z\), where
$k_0\coloneqq-M_{np}(a)/c$, and \((Z,Y)\) solves the frozen system
\[
 -Z''-\alpha yZ'+\alpha qZ=0,
 \qquad cY'+M_{np}(a)Z=0.
\]

By the scalings $\psi_p=d^{q/2}P(y)$ and $\psi_n=d^{(q+1)/2}Q(y)$,
the equation $(\mathcal L_d-\Lambda)\boldsymbol\psi=0$ becomes
\begin{equation}\label{reg:eq:scaled-repelling-system}
\begin{aligned}
 -P''-\frac{b_p(a+\sqrt d\,y)}{\sqrt d}P'
 +(M_{pp}(a+\sqrt d\,y)-\Lambda)P
 +\sqrt d\,M_{pn}(a+\sqrt d\,y)Q&=0,\\
 -\sqrt d\,Q''-b_n(a+\sqrt d\,y)Q'
 +M_{np}(a+\sqrt d\,y)P
 +\sqrt d\,(M_{nn}(a+\sqrt d\,y)-\Lambda)Q&=0.
\end{aligned}
\end{equation}
We correct $(Z,Y)$ on $[-3R_d,3R_d]$ via a finite-interval
fixed-point argument. Define
\[
 \mu (y):=Z(y)^2{\rm e}^{\alpha y^2/2},
 \qquad
 c_d(y):=-b_n(a+\sqrt d\,y),
\qquad
 k_d(y):={-M_{np}(a+\sqrt d\,y)\over c_d(y)}.
\]
Restricting to a sufficiently small neighborhood of \(a\), we have
$c_d\geq c_0>0$. The first equation in
\eqref{reg:eq:scaled-repelling-system} can be expressed as
$
 -P''-\alpha yP'+\alpha qP=f_{p,d}[P,Q],
$
where
\[
\begin{aligned}
 f_{p,d}[P,Q]
 \coloneqq{}&
 \left(\frac{b_p(a+\sqrt d\,y)}{\sqrt d}-\alpha y\right)P'
 -\bigl(M_{pp}(a+\sqrt d\,y)-M_{pp}(a)\bigr)P\\
 &-\sqrt d\,M_{pn}(a+\sqrt d\,y)Q.
\end{aligned}
\]
Furthermore, for any $C^2$ function $w$, direct calculations {yield}
\[
 \bigl(-\partial_{yy}-\alpha y\partial_y+\alpha q\bigr)(Zw)
 =-Z\mu^{-1}(\mu w')',
\]
which motivates the definition of the map $(P,Q)\mapsto \widetilde P$ given by 
\begin{equation}\label{reg:eq:repelling-green-p}
 \widetilde P(y)\coloneqq Z(y)\left[1-
 \int_0^y\mu(t)^{-1}\int_{-3R_d}^t
 \frac{\mu(s)}{Z(s)}f_{p,d}[P,Q](s)
 \,\mathrm{d}s\,\mathrm{d}t\right].
\end{equation}
This map imposes $\widetilde P(0)=Z(0)$ and
$(\widetilde P/Z)'(-3R_d)=0$.

To handle the equation for $Q$ in \eqref{reg:eq:scaled-repelling-system}, we introduce $S\coloneqq Q'-k_dP$, then we have
\[
 S'-\frac{c_d}{\sqrt d}S
 =(M_{nn}(a+\sqrt d\,y)-\Lambda)Q-k_d'P-k_dP'.
\]
Imposing the terminal condition $S(3R_d)=0$, we obtain the stable
representation
\begin{equation}\label{reg:eq:repelling-green-n}
\begin{aligned}
 S_d[P,Q](y)
 \coloneqq{}&\int_y^{3R_d}
 \exp\left[-\frac1{\sqrt d}\int_y^s c_d(t)\,\mathrm{d}t\right]\\[-1mm]
 &\quad\times
 \left[k_d'(s)P(s)+k_d(s)P'(s)
 -\bigl(M_{nn}(a+\sqrt d\,s)-\Lambda\bigr)Q(s)\right]
 \,\mathrm{d}s.
\end{aligned}
\end{equation}

We now close the map in relative variables. Write
$P=Z(1+\rho_p)$ and $Q=Y(1+\rho_n)$. {Since} $Y'=k_0Z$,
the identity $S=Q'-k_dP$ is equivalent to
\begin{equation*}
 [Y(\rho_n-\rho_p)]'
 =(k_d-k_0)Z(1+\rho_p)+S-Y\rho_p'.
\end{equation*}
Given a trial pair $(\rho_p,\rho_n)$, the formula
\eqref{reg:eq:repelling-green-p}, with
$(P,Q)=(Z(1+\rho_p),Y(1+\rho_n))$, first yields
$\widetilde P=Z(1+\widetilde\rho_p)$. Formula
\eqref{reg:eq:repelling-green-n} then yields
$\widetilde S=S_d[\widetilde P,Q]$, and we define
\begin{equation}\label{liu-20260908-1} \widetilde\rho_n(y)\coloneqq\widetilde\rho_p(y)
 +\frac1{Y(y)}\int_{-3R_d}^y
 \bigl[(k_d-k_0)Z(1+\widetilde\rho_p)
       +\widetilde S-Y\widetilde\rho_p'\bigr](s)\,\mathrm{d}s.
\end{equation}
Thus, the common relative amplitude is fixed by
$(\widetilde\rho_n-\widetilde\rho_p)(-3R_d)=0$.

We define the following space
\[
 \begin{aligned}
 \mathcal X_d\coloneqq\bigl\{(\rho_p,\rho_n)&\in
 C^1([-3R_d,3R_d])^2:\ 
 \rho_p(0)=0,\ \rho_p'(-3R_d)=0,\,\,  (\rho_n-\rho_p)(-3R_d)=0\bigr\},
 \end{aligned}
\]
endowed with the norm
\[
 \|(\rho_p,\rho_n)\|_{\mathcal X_d}
 \coloneqq\sum_{j=0}^1\sup_{|y|\leq3R_d}\langle y\rangle^j
 \bigl(|\partial_y^j\rho_p|+|\partial_y^j\rho_n|\bigr).
\]
In conjunction with \eqref{reg:eq:repelling-green-p} and \eqref{liu-20260908-1}, for each $(\rho_p,\rho_n)\in {\mathcal X_d}$, we define
$
 \mathfrak T_d(\rho_p,\rho_n):=
 \left(\widetilde\rho_p,\,
       \widetilde\rho_n\right).
$ 
We next establish the estimates needed for the contraction argument. 

From the integral representation in
Lemma~\ref{reg:lem:repelling-profile} and the definition of $Y$, we observe that
\begin{equation}\label{reg:eq:repelling-profile-ratios}
 \frac{Z'}Z=
 \begin{cases}
  O(\langle y\rangle^{-1}),&y\geq0,\\
  O(\langle y\rangle),&y\leq0,
 \end{cases}
 \qquad
 \frac YZ=
 \begin{cases}
  \asymp\langle y\rangle,&y\geq0,\\
  \asymp\langle y\rangle^{-1},&y\leq0,
 \end{cases}
\end{equation}
with the corresponding differentiated and reciprocal bounds. Moreover,
\begin{equation}\label{reg:eq:repelling-kernel-bounds}
\begin{aligned}
 \sup_{|y|\leq3R_d}\langle y\rangle\mu(y)^{-1}
 \int_{-3R_d}^{y}\mu(s)\,\mathrm{d}s&\leq C,\\
 \sup_{|y|\leq3R_d}\int_y^{3R_d}
 \exp\left[-\frac1{\sqrt d}\int_y^s c_d(t)\,\mathrm{d}t\right]
 \,\mathrm{d}s&\leq C\sqrt d.
\end{aligned}
\end{equation}
For $k=0,1,2$, splitting the integral at $0$ and using
\eqref{reg:eq:Z-plus} and \eqref{reg:eq:Z-minus}, we derive 
\begin{equation}\label{reg:eq:repelling-Z-integrals}
 \int_{-3R_d}^{y}\langle s\rangle^kZ(s)\,\mathrm{d}s
 \leq C
 \begin{cases}
  \langle y\rangle^{k-1}Z(y),&y\leq-1,\\
  \langle y\rangle^{k+1}Z(y),&y\geq1.
 \end{cases}
\end{equation}
The estimates on $[-1,1]$ are absorbed into the same constant.
Uniformly for $|y|\leq3R_d$, 
\[
\begin{aligned}
 \left|\frac{b_p(a+\sqrt d\,y)}{\sqrt d}-\alpha y\right|
 &\leq C\sqrt d\,y^2,\\
 |M_{ij}(a+\sqrt d\,y)-M_{ij}(a)|
 +|b_n(a+\sqrt d\,y)-b_n(a)|
 &\leq C\sqrt d\,|y|.
\end{aligned}
\]
Within a fixed ball of $\mathcal X_d$, these estimates
yield
\[
 \frac{|f_{p,d}[P,Q](y)|}{Z(y)}
 \leq C\sqrt d \bigl(1+\|(\rho_p,\rho_n)\|_{\mathcal X_d}\bigr) \times
 \begin{cases}
  \langle y\rangle^3,&y\leq0,\\
  \langle y\rangle,&y\geq0,
 \end{cases}
\]
\[
 \frac{|S_d[\widetilde P,Q](y)|}{Z(y)}
 \leq C\sqrt d\,\langle y\rangle
 \bigl(1+\|(\rho_p,\rho_n)\|_{\mathcal X_d}\bigr).
\]
Substituting these into the formula for $\widetilde\rho_n$, we obtain
\[
 \|\mathfrak T_d(\rho_p,\rho_n)\|_{\mathcal X_d}
 \leq C\bigl(\sqrt d\,R_d+\sqrt d\,R_d^3
 +\sqrt d\log(2+R_d)\bigr)
 \bigl(1+\|(\rho_p,\rho_n)\|_{\mathcal X_d}\bigr).
\]
The analogous calculation for two trial pairs yields the corresponding
difference estimate. 
Using \eqref{reg:eq:repelling-profile-ratios}-\eqref{reg:eq:repelling-Z-integrals} in the three formulas defining
$(\widetilde\rho_p,\widetilde S,\widetilde\rho_n)$, we deduce that
\[
 \|\mathfrak T_d(0,0)\|_{\mathcal X_d}\leq C\kappa_d,
 \quad
 \|\mathfrak T_d(z_1)-\mathfrak T_d(z_2)\|_{\mathcal X_d}
 \leq C\kappa_d\|z_1-z_2\|_{\mathcal X_d},
 \qquad z_1,z_2\in\mathcal X_d,
\]
where
$
 \kappa_d\coloneqq\delta_d+{\delta_d^3}/{d}\to 0.
$
For clarity, the term $\delta_d^3/d$ arises from the drift
remainder:
\[
 \sqrt d\,|y|^2\frac{|Z'(y)|}{Z(y)}
 \leq C\sqrt d\,R_d^3=C{\delta_d^3}/{d}.
\]
All other coefficient errors are $O(\delta_d)$, and thus the stable-kernel
bound in \eqref{reg:eq:repelling-kernel-bounds} controls the apparent
$d^{-1/2}$ factor in the equation for $S$.

For sufficiently small $d$, the operator $\mathfrak T_d$ is a contraction on
the closed ball of radius $2C\kappa_d$. Let
$(\widehat\rho_{p,d},\widehat\rho_{n,d})$ be its fixed point and set
\[
 \widehat Z_d\coloneqq Z(1+\widehat\rho_{p,d}),
 \qquad \widehat Y_d\coloneqq Y(1+\widehat\rho_{n,d}).
\]
The fixed-point identities immediately yield
$\widehat Z_d,\widehat Y_d\in C^2([-3R_d,3R_d])$, and
\begin{equation}\label{reg:eq:repelling-corrected-C1}
 \sum_{j=0}^1\sup_{|y|\leq3R_d}\langle y\rangle^j
 \left[
 \left|\partial_y^j\left(\frac{\widehat Z_d}{Z}-1\right)\right|
 +\left|\partial_y^j\left(\frac{\widehat Y_d}{Y}-1\right)\right|
 \right]
 \leq C\kappa_d.
\end{equation}
We next derive the second-derivative estimate needed in the matching
region. At the fixed point, it follows from \eqref{reg:eq:repelling-green-p} that $-Z\mu^{-1}(\mu \widehat\rho_{p,d}')'=f_{p,d}[\widehat Z_d,\widehat Y_d]$, which 
implies
\[
 \widehat\rho_{p,d}''
 =-\left(2\frac{Z'}Z+\alpha y\right)\widehat\rho_{p,d}'
 -\frac{f_{p,d}[\widehat Z_d,\widehat Y_d]}Z.
\]
Thus, applying \eqref{reg:eq:repelling-corrected-C1} and
\eqref{reg:eq:repelling-profile-ratios} we have
$|\widehat\rho_{p,d}''|\leq C\kappa_d$ on
$[R_d,2R_d]$. To estimate the other component $\widehat\rho_{n,d}$,  set
$
 \widehat S_d
 \coloneqq S_d[\widehat Z_d,\widehat Y_d].
$
The fixed-point identities imply
$
 \widehat Y_d'=k_d\widehat Z_d+\widehat S_d.
$ Define 
\[
 B_d(y)\coloneqq
 k_d(y)\widehat Z_d'(y)
 -\bigl(M_{nn}(a+\sqrt d\,y)-\Lambda\bigr)\widehat Y_d(y)\quad\text{on }R_d\leq y\leq3R_d.
\]
Using the first equation in \eqref{reg:eq:scaled-repelling-system},
together with \eqref{reg:eq:repelling-corrected-C1} and
\eqref{reg:eq:repelling-profile-ratios}, we deduce that
\[
 |B_d'(y)|\leq CZ(y),
 \qquad
 |k_d'(y)\widehat Z_d(y)|
 \leq C\sqrt d\,Z(y).
\]
Integrating the $B_d$-term in
\eqref{reg:eq:repelling-green-n} by parts, 
we find that the terminal contribution at $s=3R_d$ is
exponentially small for $R_d\leq y\leq2R_d$,
while
\eqref{reg:eq:repelling-kernel-bounds} controls the remaining integral.
Hence, $|\widehat S_d'(y)|
 \leq C\sqrt d\,y^2Z(y)$ for $R_d\leq y\leq2R_d$.  
Differentiating
$\widehat Y_d'=k_d\widehat Z_d+\widehat S_d$ then yields
\[
 Y\widehat\rho_{n,d}''
 =
 k_d'\widehat Z_d+k_d\widehat Z_d'
 +\widehat S_d'
 -Y''(1+\widehat\rho_{n,d})
 -2Y'\widehat\rho_{n,d}'.
\]
Since $Y/Z\asymp y$, $Y'/Y=O(y^{-1})$, and
$Y''/Y=O(y^{-2})$ on this interval, in view of
$R_d^{-2}=\delta_d^3/d$ and
$\sqrt d\,R_d=\delta_d$, we obtain
\begin{equation}\label{reg:eq:repelling-corrected-C2}
 \sup_{R_d\leq y\leq2R_d}
 \left(
 |\widehat\rho_{p,d}''(y)|
 +|\widehat\rho_{n,d}''(y)|
 \right)
 \leq C\kappa_d.
\end{equation} 
On the rest of $[-3R_d,3R_d]$, classical ODE regularity provides a
finite $C^2$-bound, while no small second-derivative estimate there is used
below. It follows from
\eqref{reg:eq:repelling-corrected-C1} that $\widehat Z_d$ and
$\widehat Y_d$ are positive for small $d$, and the fixed-point
identities ensure that they solve
\eqref{reg:eq:scaled-repelling-system} exactly.

Let $C_a>0$ be the coefficient in
\eqref{reg:eq:left-repelling-asymptotics}, and define
\begin{equation}\label{reg:eq:repelling-inner-field}
 \boldsymbol\Phi_d(x)\coloneqq
 C_a\sqrt{\frac{\alpha}{2\pi}}
 \begin{pmatrix}
  d^{q/2}\widehat Z_d((x-a)/\sqrt d)\\[1mm]
  d^{(q+1)/2}\widehat Y_d((x-a)/\sqrt d)
 \end{pmatrix}.
\end{equation}
Then we conclude that
\begin{equation}\label{reg:eq:repelling-inner-exact}
 \boldsymbol\Phi_d\gg0,
 \qquad
 (\mathcal L_d-\Lambda)\boldsymbol\Phi_d=0
 \quad\text{for }|x-a|\leq3\delta_d.
\end{equation}

\medskip
\noindent\emph{Step 2.}
We next match the layer to a transport continuation.
Let \(h=x-a\). Combining the large-\(y\) asymptotics in
Lemma~\ref{reg:lem:repelling-profile}, the definition of \(Y\), and with
\eqref{reg:eq:repelling-corrected-C1}, we deduce, uniformly for
\(\delta_d\leq h\leq2\delta_d\),
\[
\begin{aligned}
 \Phi_{d,p}(a+h)=C_ah^q(1+o(1)),\quad
 \Phi_{d,n}(a+h)=
 C_a\frac{-M_{np}(a)}{c(q+1)}h^{q+1}(1+o(1)).
\end{aligned}
\]
By Lemma~\ref{reg:lem:repelling-endpoint-asymptotics}, the function \(\boldsymbol u\) {shares} the same leading terms, which also hold after one differentiation. Hence, there exists a quantity \(\omega_d=o(1)\) such that
\begin{equation}\label{reg:eq:repelling-layer-transport-match}
 \max_{i=p,n}\left\{
 \left|\frac{\Phi_{d,i}(a+h)}{u_i(a+h)}-1\right|
 +h\left|\partial_x
 \frac{\Phi_{d,i}(a+h)}{u_i(a+h)}\right|
 \right\}
 \leq C\omega_d,
 \qquad \delta_d\leq h\leq2\delta_d.
\end{equation}
Moreover, the differentiated profile estimates, including
\eqref{reg:eq:repelling-corrected-C2} on
\(R_d\leq y\leq2R_d\), also yield
\begin{equation}\label{reg:eq:repelling-layer-derivatives}
 |\partial_x^j\Phi_{d,i}(a+h)|
 \leq Ch^{-j}\Phi_{d,i}(a+h),
 \qquad j=1,2,\quad i=p,n.
\end{equation}

Let \(\boldsymbol w_d\) be the solution of the transport system
\eqref{reg:eq:transport-system} with
\(\boldsymbol w_d(a+\delta_d)
 =\boldsymbol\Phi_d(a+\delta_d)\),
extended to \(a+2\Delta_d\). On
\(J_{d}\coloneqq[a+\delta_d,a+\delta_d+h_d]\), set
\(\boldsymbol D_d\coloneqq\boldsymbol w_d-
\boldsymbol\Phi_d\). By
\eqref{reg:eq:repelling-inner-exact}, we have
\begin{equation}\label{reg:eq:repelling-D-equation} 
\begin{cases}
 -b_i(D_{d,i})'+(M_{ii}-\Lambda)D_{d,i}
 +M_{ij}D_{d,j}=-d(\Phi_{d,i})'',
 &\{i,j\}=\{p,n\},\\
 \boldsymbol D_d(a+\delta_d)=0.
\end{cases}
\end{equation}

We now estimate this mismatch. On \(J_{d}\), where \(x-a\asymp\delta_d\), \eqref{reg:eq:repelling-layer-transport-match}-\eqref{reg:eq:repelling-layer-derivatives} yield
\begin{equation}\label{reg:eq:repelling-matching-profile-bounds}
 b_p\asymp\delta_d,\quad -b_n\asymp1,\quad
 {\Phi_{d,n}\over\Phi_{d,p}}\asymp\delta_d,\quad
 {\Phi_{d,p}\over\Phi_{d,n}}\asymp\delta_d^{-1},\quad
 { |\partial_x^j\Phi_{d,i}|\over\Phi_{d,i}}
 \leq C\delta_d^{-j}\quad(j=1,2).
\end{equation}
Set \(E_i\coloneqq D_{d,i}/\Phi_{d,i}\). Substituting
\(D_{d,i}=\Phi_{d,i}E_i\) into
\eqref{reg:eq:repelling-D-equation} {and applying}
\eqref{reg:eq:repelling-inner-exact}, we obtain
\[
 -b_iE_i'
 +d\frac{(\Phi_{d,i})''}{\Phi_{d,i}}E_i
 +M_{ij}\frac{\Phi_{d,j}}{\Phi_{d,i}}(E_j-E_i)
 =-d\frac{(\Phi_{d,i})''}{\Phi_{d,i}},
 \qquad \{i,j\}=\{p,n\}.
\]
It follows from \eqref{reg:eq:repelling-matching-profile-bounds} that
\begin{align}
 |E_p'|
 &\leq C\left[
 \frac d{\delta_d^3}(1+|E_p|)+|E_n-E_p|
 \right], \quad
 |E_n'|
 \leq C\left[
 \frac d{\delta_d^2}(1+|E_n|)
 +\frac1{\delta_d}|E_p-E_n|
 \right].
 \label{reg:eq:repelling-relative-mismatch}
\end{align}
Integrating these inequalities over \(J_{d}\) and absorbing the
same-component terms, which is permissible because \({dh_d}/{\delta_d^3}=d^{1/2}\), \( {h_d}/{\delta_d}=d^{3/10}\), and \({dh_d}/{\delta_d^2}=d^{9/10}\), we obtain
\[
\begin{aligned}
 \|E_p\|_{L^\infty(J_{d})}
 \leq C\left[
 \frac{dh_d}{\delta_d^3}
 +h_d\|E_n\|_{L^\infty(J_{d})}\right],\quad
 \|E_n\|_{L^\infty(J_{d})}
 \leq C\left[
 \frac{dh_d}{\delta_d^2}
 +\frac{h_d}{\delta_d}
  \|E_p\|_{L^\infty(J_{d})}\right].
\end{aligned}
\]
Substituting these inequalities into one another and subsequently applying
\eqref{reg:eq:repelling-relative-mismatch}, we deduce
\begin{equation}\label{reg:eq:repelling-D-estimates}
\begin{aligned}
 \frac{|D_{d,p}|}{\Phi_{d,p}}
 &\leq C\frac{dh_d}{\delta_d^3},&
 \frac{|(D_{d,p})'|}{\Phi_{d,p}}
 &\leq C\frac d{\delta_d^3},\\
 \frac{|D_{d,n}|}{\Phi_{d,n}}
 &\leq C\left(
 \frac{dh_d}{\delta_d^2}
 +\frac{dh_d^2}{\delta_d^4}\right),&
 \frac{|(D_{d,n})'|}{\Phi_{d,n}}
 &\leq C\left(
 \frac d{\delta_d^2}
 +\frac{dh_d}{\delta_d^4}\right).
\end{aligned}
\end{equation}

Choose \(\chi\in C^\infty(\mathbb R;[0,1])\) such that
\(\chi=0\) on \(( -\infty,0]\) and \(\chi=1\) on
\([1/2,\infty)\), and {define}
\[
 \chi_{a,d}^{(1)}(x)\coloneqq
 \chi\left(\frac{x-a-\delta_d}{h_d}\right),
 \qquad
 \boldsymbol F_d^{a,1}\coloneqq
 (1-\chi_{a,d}^{(1)})\boldsymbol\Phi_d
 +\chi_{a,d}^{(1)}\boldsymbol w_d
 \quad\text{on }J_{d}.
\]
The estimates in \eqref{reg:eq:repelling-D-estimates} ensure that
\(\boldsymbol F_{d}^{a,1}\gg0\) and yield
\begin{equation}\label{reg:eq:repelling-first-comparability}
 F_{d,i}^{a,1}\asymp\Phi_{d,i}\asymp w_{d,i},
 \qquad i=p,n.
\end{equation}

It remains to control the diffusion applied to
\(\boldsymbol w_d\). In view of
\eqref{reg:eq:repelling-first-comparability} and
\eqref{reg:eq:repelling-matching-profile-bounds}, we obtain
\(w_{d,n}\asymp(x-a)w_{d,p}\) on \(J_{d}\). The transport
equations and their first derivatives therefore yield
\begin{equation}\label{reg:eq:repelling-transport-derivatives}
 |\partial_x^jw_{d,i}(a+h)|
 \leq Ch^{-j}w_{d,i}(a+h),
 \qquad j=1,2,\quad i=p,n.
\end{equation}
Write
$\boldsymbol F_d^{a,1}=\boldsymbol\Phi_d+
\chi_{a,d}^{(1)}\boldsymbol D_d$. Using the fact that
$\boldsymbol w_d(a+\delta_d)
 =\boldsymbol\Phi_d(a+\delta_d)$, we derive
\[
\begin{aligned}
 [ (\mathcal L_d-\Lambda)\boldsymbol F_d^{a,1}]_i
 =-d\chi_{a,d}^{(1)}(w_{d,i})''
 -2d(\chi_{a,d}^{(1)})'(D_{d,i})'-\bigl[d(\chi_{a,d}^{(1)})''
          +b_i(\chi_{a,d}^{(1)})'\bigr]D_{d,i}.
\end{aligned}
\]
Combining the  bounds
$|(\chi_{a,d}^{(1)})'|\leq Ch_d^{-1}$ and
$|(\chi_{a,d}^{(1)})''|\leq Ch_d^{-2}$ with
\eqref{reg:eq:repelling-D-estimates}-\eqref{reg:eq:repelling-transport-derivatives}, we deduce
\begin{equation}\label{reg:eq:first-repelling-splice-error}
 \max_{i=p,n}
 \frac{|[(\mathcal L_d-\Lambda)\boldsymbol F_d^{a,1}]_i|}
      {F_{d,i}^{a,1}}
 \leq C\left[
 \frac d{\delta_d^2}
 +\frac{dh_d}{\delta_d^4}
 +\frac{d^2}{h_d\delta_d^3}
 +\frac{d^2}{\delta_d^4}
 \right]=o(1).
\end{equation}

\medskip
\noindent\emph{Step 3.}
We next compare the transport continuation directly with
\(\boldsymbol u\) to perform the second matching.
For \(\delta_d\leq h\leq2\Delta_d\), define
\[
 R_{d,p}(h)\coloneqq\frac{w_{d,p}(a+h)}{u_p(a+h)},
 \qquad
 R_{d,n}(h)\coloneqq\frac{w_{d,n}(a+h)}{u_n(a+h)}.
\]
Since both \(\boldsymbol w_d\) and \(\boldsymbol u\) solve the same
transport system, a direct calculation yields
\begin{equation}\label{reg:eq:repelling-ratio-system}
 R_{d,p}'=\gamma_p(h)(R_{d,n}-R_{d,p}),
 \qquad
 R_{d,n}'=\gamma_n(h)(R_{d,p}-R_{d,n}),
\end{equation}
where the coefficients \(\gamma_p\) and \(\gamma_n\) are {defined} by
\[
 \gamma_p(h)\coloneqq
 \frac{M_{pn}(a+h)u_n(a+h)}{b_p(a+h)u_p(a+h)},
 \qquad
 \gamma_n(h)\coloneqq
 \frac{M_{np}(a+h)u_p(a+h)}{b_n(a+h)u_n(a+h)}.
\]
The endpoint asymptotics established in
Lemma~\ref{reg:lem:repelling-endpoint-asymptotics} yield
\begin{equation}\label{reg:eq:repelling-ratio-coefficient-bounds}
 |\gamma_p(h)|\leq C,\qquad
 |\gamma_n(h)|\leq Ch^{-1},\qquad
 h\frac{|u_i'(a+h)|}{u_i(a+h)}\leq C.
\end{equation}
Set \(Z_d\coloneqq R_{d,p}-R_{d,n}\). It then follows from
\eqref{reg:eq:repelling-ratio-system} that
$
 Z_d'=-(\gamma_p+\gamma_n)Z_d.
$ 
On the other hand, the \(n\)-equation for \(\boldsymbol u\) implies
\[
 \gamma_n(h)=\frac{u_n'(a+h)}{u_n(a+h)}
 -\frac{M_{nn}(a+h)-\Lambda}{b_n(a+h)}.
\]
Therefore, integrating the equation for \(Z_d\) from \(\delta_d\) to \(h\), we deduce that
\begin{equation}\label{reg:eq:repelling-ratio-difference-formula}
\begin{aligned}
 Z_d(h)=Z_d(\delta_d)
 \frac{u_n(a+\delta_d)}{u_n(a+h)}\exp\left\{
 \int_{\delta_d}^h
 \left[
 \frac{M_{nn}(a+s)-\Lambda}{b_n(a+s)}-\gamma_p(s)
 \right]\,\mathrm{d}s\right\}.
\end{aligned}
\end{equation}
Choose \(h_0>0\) sufficiently small such that the endpoint expansion {yields}
\[
 \frac{u_n(a+s)}{u_n(a+t)}
 \leq C\left({s}/{t}\right)^{q+1},
 \qquad 0<s\leq t\leq h_0.
\]
Since \(2\Delta_d\leq h_0\) for small \(d\), and the exponential factor in
\eqref{reg:eq:repelling-ratio-difference-formula} remains uniformly bounded,
it follows from \eqref{reg:eq:repelling-layer-transport-match} that
\begin{equation}\label{reg:eq:repelling-ratio-decay}
 |Z_d(h)|\leq C\omega_d
 \left({\delta_d}/{h}\right)^{q+1}
 \qquad \text{for }\delta_d\leq h\leq2\Delta_d.
\end{equation}
Consequently, combining \eqref{reg:eq:repelling-ratio-system} with
\eqref{reg:eq:repelling-ratio-decay}, we obtain
\[
 |R_{d,p}(h)-R_{d,p}(\delta_d)|\leq C\omega_d,
 \qquad
 |R_{d,n}(h)-R_{d,p}(h)|\leq C\omega_d\quad \text{ for }\delta_d\leq h\leq2\Delta_d.
\]
Hence, \(\boldsymbol w_d\gg0\) on
\([a+\delta_d,a+2\Delta_d]\) for all sufficiently small \(d\).

Since the entire left construction is determined only up to a common
positive scalar, we normalize it by requiring
\(R_{d,p}(2\Delta_d)=1\). Indeed,
\eqref{reg:eq:repelling-layer-transport-match} {shows that}
\(R_{d,p}(\delta_d)=1+O(\omega_d)\), while the preceding estimate yields
\(|R_{d,p}(2\Delta_d)-R_{d,p}(\delta_d)|\leq C\omega_d\). Hence, before
rescaling, \(R_{d,p}(2\Delta_d)=1+O(\omega_d)\), which is strictly positive for small \(d\). Since the rescaling {applies uniformly to}
\(\boldsymbol\Phi_d\), \(\boldsymbol w_d\), and
\(\boldsymbol F_d^{a,1}\), {it preserves} all the relative
residual estimates obtained above, as well as
\eqref{reg:eq:repelling-layer-transport-match}.

Under this normalization, for
$\Delta_d\leq h\leq2\Delta_d$,
\[
\begin{aligned}
 |R_{d,p}(h)-1|
 &\leq C\int_h^{2\Delta_d}|Z_d(s)|\,\mathrm{d}s,\\
 |R_{d,n}(h)-1|
 &\leq |R_{d,p}(h)-1|+|Z_d(h)|.
\end{aligned}
\]
Combining these estimates with
\eqref{reg:eq:repelling-ratio-system} and
\eqref{reg:eq:repelling-ratio-coefficient-bounds}, we obtain
\begin{equation}\label{reg:eq:repelling-common-match}
 \max_{i=p,n}\sum_{j=0}^1h^j
 \frac{|\partial_x^j(w_{d,i}-u_i)|}{u_i}
 \leq C\omega_d
 \left(\frac{\delta_d}{\Delta_d}\right)^{q+1},
 \qquad \Delta_d\leq h\leq2\Delta_d.
\end{equation}

We define the second cutoff and interpolation as
\[
 \chi_{a,d}^{(2)}(x)\coloneqq
 \chi\left(\frac{x-a-\Delta_d}{\Delta_d}\right),
 \qquad
 \boldsymbol F_d^{a,2}\coloneqq
 (1-\chi_{a,d}^{(2)})\boldsymbol w_d
 +\chi_{a,d}^{(2)}\boldsymbol u
\]
on $[a+\Delta_d,a+2\Delta_d]$. On this interval,
\eqref{reg:eq:repelling-common-match} and the differentiated transport
equations imply
\[
 F_{d,i}^{a,2}\asymp u_i\asymp w_{d,i},
 \qquad
 \frac{|u_i''|}{u_i}
 +\frac{|(w_{d,i})''|}{w_{d,i}}
 \leq\frac C{\Delta_d^2}.
\]
Then together with the bounds $|(\chi_{a,d}^{(2)})'|\leq C\Delta_d^{-1}$ and
$|(\chi_{a,d}^{(2)})''|\leq C\Delta_d^{-2}$, we deduce
\begin{equation}\label{reg:eq:second-repelling-splice-error}
 \max_{i=p,n}
 \frac{|[(\mathcal L_d-\Lambda)\boldsymbol F_d^{a,2}]_i|}
      {F_{d,i}^{a,2}}
 \leq C\left[
 \frac d{\Delta_d^2}
 +\frac{\omega_d\delta_d^{q+1}}{\Delta_d^{q+2}}
 \right]=o(1).
\end{equation}
Indeed, to see the last equality, we note that the first term vanishes since \(d/\Delta_d^2\to0\). For the second term, recalling $\Delta_d=\delta_d^\vartheta$, we obtain
$
 {\delta_d^{q+1}}/{\Delta_d^{q+2}}
 =\delta_d^{q+1-\vartheta(q+2)}\to 0
$
provided $\vartheta$ is chosen such that $q+1-\vartheta(q+2)>0$.
Finally, estimate \eqref{reg:eq:repelling-common-match} ensures that $\boldsymbol F_d^{a,2}\gg0$ for  small \(d\).  Since
$\chi=1$ on $[1/2,\infty)$, this interpolation coincides with
$\boldsymbol u$ near \(a+2\Delta_d\). Step 3 is complete.

\medskip
\noindent\emph{Step 4.}
It remains to assemble the pieces and verify the residual and phase
estimates. Define
\[
 \boldsymbol F_d^a(x)\coloneqq
 \begin{cases}
  \boldsymbol\Phi_d(x),
   &a-3\delta_d\leq x\leq a+\delta_d,\\
  (1-\chi_{a,d}^{(1)})\boldsymbol\Phi_d
 +\chi_{a,d}^{(1)}\boldsymbol w_d,
   &a+\delta_d\leq x\leq a+\delta_d+h_d,\\
  \boldsymbol w_d(x),
   &a+\delta_d+h_d\leq x\leq a+\Delta_d,\\
  (1-\chi_{a,d}^{(2)})\boldsymbol w_d
 +\chi_{a,d}^{(2)}\boldsymbol u,
   &a+\Delta_d\leq x\leq a+2\Delta_d.
 \end{cases}
\]
{Since} \(\chi\) {remains constant near the junctions}, the values and
the first two derivatives match perfectly at every interface. Consequently, \(\boldsymbol F_d^a\in C^2\), and it follows from the preceding estimates that \(\boldsymbol F_d^a\gg0\).
The construction of \(\boldsymbol F_d^a\) is
illustrated in Figure~\ref{reg:fig:left-repelling-module}.

\begin{figure}[htbp]
\centering
\includegraphics[width=0.85\linewidth]{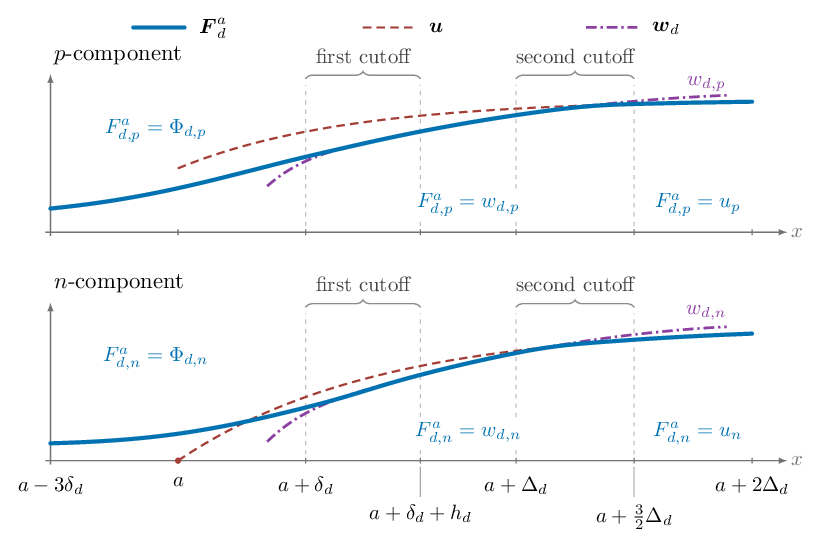}
\caption{\small The construction of $\boldsymbol F_d^a$ at a left repelling endpoint. Solid and dashed
curves represent $\boldsymbol F_d^a$ and $\boldsymbol u$,
respectively. The shaded bands mark the two cutoff intervals. 
In particular, $\boldsymbol F_d^a=\boldsymbol u$ for
$x\geq a+\tfrac32\Delta_d$.}
\label{reg:fig:left-repelling-module}
\end{figure}

The residual vanishes on the exact endpoint layer. Furthermore, by virtue of \eqref{reg:eq:first-repelling-splice-error} and \eqref{reg:eq:second-repelling-splice-error}, it is \(o(1)\) relative to \(\boldsymbol F_d^a\) on the two transition regions. On the intervening transport segment, it follows from the ratio estimates and the differentiated transport equations that
\[
 \max_{i=p,n}\frac{|(w_{d,i})''|}{w_{d,i}}
 \leq\frac C{(x-a)^2}\leq\frac C{\delta_d^2}.
\]
Since \((\mathcal L_d-\Lambda)\boldsymbol w_d
=-d(\boldsymbol w_d)''\), the relative residual is bounded by \(Cd/\delta_d^2=o(1)\). {Combining this with} the transition estimates, we establish \eqref{reg:eq:repelling-residual}.

It remains to verify \eqref{reg:eq:repelling-phase}. Let \(h=a-x\in[\delta_d,3\delta_d]\). In this regime, \(\boldsymbol F_d^a(a-h)=\boldsymbol\Phi_d(a-h)\) {and} \(y=-h/\sqrt d\to-\infty\). By applying the negative-\(y\) asymptotics from Lemma~\ref{reg:lem:repelling-profile}, together with \eqref{reg:eq:repelling-corrected-C1}, we obtain
\[
 -d\log F_{d,i}^a(a-h)
 =\frac\alpha2h^2
 +O\bigl(d|\log d|+d\kappa_d+d\omega_d\bigr).
\]
Here, the final term accounts for the common scalar rescaling {introduced} in Step 3. On the other hand, we have
\[
 \int_a^{a-h}b_p(t)\,\mathrm{d}t
 =\frac\alpha2h^2+O(h^3).
\]
In view of
$
 d|\log d|+d\kappa_d+d\omega_d+\delta_d^3
 =o(\delta_d^2),
$
we deduce \eqref{reg:eq:repelling-phase}. The proof is complete.
\end{proof}


\begin{lemma}
\label{reg:lem:attracting-domain-modules}
Let $\boldsymbol u$ be the positive profile from
Lemma {\rm\ref{reg:lem:transport-value}}. Assume that $a$ is attracting.
Then $\boldsymbol u$ has a positive $C^1$ extension across $a$
that solves the transport system on both sides. Set
$\varepsilon_d=d^{1/3}$. Then there exists a positive \(C^2\) function
$\boldsymbol F_d^a$ which agrees with this extension {for}
$|x-a|\geq2\varepsilon_d$, such that, uniformly in a neighborhood of \(a\),
\[
 \max_{i=p,n}
 \frac{\bigl|[(\mathcal L_d-\Lambda_I^{\mathrm{reg}})
                 \boldsymbol F_d^a]_i\bigr|}
      {F_{d,i}^a}
 =o(1).
\]
The analogous assertion holds at a right attracting endpoint.

If $a=\ell$ is a domain endpoint, fix a cutoff function $\chi_\ell$ that equals one near $\ell$, and define
\begin{equation}\label{reg:eq:left-domain-layer}
 B_{p,d}^\ell(x)
 \coloneqq
 d\frac{u_p'(\ell)}{b_p(\ell)}
 \exp\left[-\frac{b_p(\ell)(x-\ell)}{d}\right]\chi_\ell(x),
 \qquad
 \boldsymbol F_d^\ell
 \coloneqq(u_p+B_{p,d}^\ell,u_n)^{\mathsf T}.
\end{equation}
At a right domain endpoint, we prescribe the \(n\)-component as
$B_{n,d}^r=d u_n'(r)b_n(r)^{-1}
\exp[b_n(r)(r-x)/d]\chi_r$, where
$\chi_r=1$ near \(r\).
For {sufficiently} small \(d\), each corrected profile is positive, {satisfies the Neumann condition exactly}, and exhibits a relative residual $O(d)$.
\end{lemma}

\begin{proof}
Set $\Lambda=\Lambda_I^{\mathrm{reg}}$. We treat a left attracting
endpoint. Let $h=x-a$ 
and
$\beta_0=[M_{nn}(a)-\Lambda]/[-b_n'(a)]>0$. 
Since $[-b_n(a+h)]/h$ admits a positive
\(C^1\) extension at $h=0$, the coefficients in \eqref{reg:eq:left-attracting-representation} extend to negative $h$. For $|h|\leq\rho$, consider
\begin{equation}\label{liu-20260919-1}
 \begin{aligned}
 v_p(h)
 &=u_p(a)+\int_0^h
 \frac{[M_{pp}(a+s)-\Lambda]v_p(s)
       +M_{pn}(a+s)v_n(s)}{b_p(a+s)}\,\mathrm d s,\\
 v_n(h)
 &=\frac1{E(h)}\int_0^1t^{\beta_0-1}E(th)
 \frac{-M_{np}(a+th)}{[-b_n(a+th)]/(th)}
 v_p(th)\,\mathrm d t.
 \end{aligned}
\end{equation}
In view of $\beta_0>0$, the second line defines a bounded map from \(v_p\) to
\(v_n\) on \(C([-\rho,\rho])\). Substituting it into the first line
yields a map $\mathcal S$ satisfying
$
 \|\mathcal Sv_p-\mathcal S\widehat v_p\|_\infty
 \leq C\rho\|v_p-\widehat v_p\|_\infty.
$
Hence, $\mathcal S$ is a contraction for sufficiently small $\rho$. On
$h\geq0$, the given transport profile satisfies the same two
identities by the \(p\)-equation and
\eqref{reg:eq:left-attracting-representation}, and therefore the fixed point
yields an extension of $\boldsymbol u$. The first identity in \eqref{liu-20260919-1} implies $v_p\in C^1$. 
Together with the integrability of the kernel in the second identity, we also obtain $v_n\in C^1$. At $h=0$, the second identity in \eqref{liu-20260919-1} yields the positive
boundary condition in \eqref{reg:eq:characteristic-endpoint-values}. After reducing
$\rho$ if necessary, \(\widetilde{\boldsymbol u}(a+h)=(v_p(h),v_n(h))^{\mathsf T}\) provides the required positive extension. 

It remains to regularize the extension. In fact, the transport equations yield
\[
 |\widetilde u_p''(a+h)|\leq C,
 \qquad
 |\widetilde u_n''(a+h)|\leq\frac{C}{|h|},
 \qquad h\neq0.
\]
Let $\varrho_\varepsilon$ be an even mollifier and choose
$\chi\in C_c^\infty((-2,2))$ with $\chi=1$ on \([-1,1]\).
Define
\[
 \boldsymbol F_d^a
 \coloneqq
 \widetilde{\boldsymbol u}
 +\chi\left(\frac{x-a}{\varepsilon_d}\right)
 \bigl(\varrho_{\varepsilon_d}*
       \widetilde{\boldsymbol u}-\widetilde{\boldsymbol u}\bigr).
\]
If $\omega$ is a modulus of continuity of
$\widetilde{\boldsymbol u}'$, {then} by the preceding bounds we obtain
\[
 \begin{aligned}
 \|\boldsymbol F_d^a-\widetilde{\boldsymbol u}\|_\infty
 \leq C\varepsilon_d\omega(C\varepsilon_d),
 \,\,\,\,
 \|(\boldsymbol F_d^a)'-\widetilde{\boldsymbol u}'\|_\infty
 \leq C\omega(C\varepsilon_d),\,\,\,\,
 \|(\boldsymbol F_d^a)''\|_\infty
  \leq C\varepsilon_d^{-1}.
 \end{aligned}
\]
{Consequently, we obtain}
\[
 \left\|
 \left[-\operatorname{diag}(b_p,b_n)\partial_x
       +\boldsymbol M-\Lambda\right]\boldsymbol F_d^a
 \right\|_\infty=o(1),
 \qquad
 d\|(\boldsymbol F_d^a)''\|_\infty
 \leq C\frac d{\varepsilon_d}=o(1).
\]
{The treatment of the right attracting endpoint is analogous, taking} $h=b-x$ and interchanging \(p\) and \(n\).

Finally, suppose that $a=\ell$ is a domain endpoint. {Since} both
drifts are nonzero there, $\boldsymbol u$ is \(C^2\) near
$\ell$. The endpoint condition and the \(n\)-equation {yield}
$u_n'(\ell)=0$, while
$(B_{p,d}^\ell)'(\ell)=-u_p'(\ell)$. {Hence},
$(\boldsymbol F_d^\ell)'(\ell)=\boldsymbol0$. {On the region where}
$\chi_\ell=1$,
\[
 -d(B_{p,d}^\ell)''-b_p(x)(B_{p,d}^\ell)'
 =u_p'(\ell)[b_p(x)-b_p(\ell)]
  \exp\left[-\frac{b_p(\ell)(x-\ell)}d\right]=O(d),
\]
{since} $b_p(x)-b_p(\ell)=O(x-\ell)$ and
$\sup_{s\geq0}s\exp[-b_p(\ell)s/d]=O(d)$. {Furthermore},
$B_{p,d}^\ell=O(d)$, \(-d\boldsymbol u''=O(d)\), and the cutoff
terms are exponentially small. {Consequently}, the corrected profile has
relative residual \(O(d)\) and remains positive. The right domain
endpoint follows from the displayed formula in the statement by the
change of variables \(s=r-x\) and the interchange of \(p\) and
\(n\). The proof is complete.
\end{proof}

By assembling the appropriate modules at the two endpoints, we obtain a
positive interior profile \(\boldsymbol F_d\). It coincides with the transport
profile away from the shrinking endpoint neighborhoods, satisfies the
original Neumann condition at every domain endpoint, and satisfies
\begin{equation}\label{reg:eq:neutral-interior-profile}
 \max_{i=p,n}
 \frac{|[(\mathcal L_d-\Lambda_I^{\mathrm{reg}})
          \boldsymbol F_d]_i|}{F_{d,i}}=o(1).
\end{equation}
On \(I\), a common scalar normalization ensures that
\(-d\log F_{d,i}\to0\) uniformly in \(i\).

\subsection{Local realization of the effective value}

It remains to make the comparison strict at the artificial endpoints
of an isolating neighborhood. The following gluing argument also
reveals where the sign condition \(M_{pn},M_{np}<0\) plays an important role.

At a left repelling endpoint, we define
\[
 S_a(x)\coloneqq \int_a^x b_p(s)\,\mathrm{d}s>0,
 \qquad r_a(x)\coloneqq -b_p(x)>0,\qquad x<a.
\]
Choose a sufficiently short exterior one-sided neighborhood of \(a\)
such that \(b_n(x)<b_p(x)<0\) holds therein.
Fix \(0<\vartheta<1\) and define
\begin{equation}\label{reg:eq:realization-repelling-functions}
 \begin{aligned}
 G_{p,d}^{a,-}(x)=\mathrm{e}^{-\vartheta S_a(x)/d},\qquad 
 G_{n,d}^{a,-}(x)=\frac{d\kappa_{a}}{r_a(x)}
 \mathrm{e}^{-\vartheta S_a(x)/d}.
 \end{aligned}
\end{equation}
Choose
$
 \kappa_a>
 \frac{-M_{np}(a)}{\vartheta[-b_n(a)]}.
$
A direct calculation shows that, for
\(d^{2/5}\leq a-x\leq\rho\),
\begin{equation}\label{reg:eq:slow-tail}
 (\mathcal L_d-(\Lambda_I^{\mathrm{reg}}-\eta/4))
 \boldsymbol G_d^{a,-}
 \geq c\left(1+\frac{(a-x)^2}{d}\right)
 \boldsymbol G_d^{a,-}.
\end{equation}
Indeed, setting \(t=a-x\) and
\(\lambda_0=\Lambda_I^{\mathrm{reg}}-\eta/4\), a direct substitution gives
\[
\begin{aligned}
 \frac{[(\mathcal L_d-\lambda_0)\boldsymbol G_d^{a,-}]_p}
      {G_{p,d}^{a,-}}
 &=\frac{\vartheta(1-\vartheta)r_a^2}{d}
   +O\!\left(1+\frac d{r_a}\right),\\
 \frac{[(\mathcal L_d-\lambda_0)\boldsymbol G_d^{a,-}]_n}
      {G_{p,d}^{a,-}}
 &=\vartheta[-b_n(a)]\kappa_a+M_{np}(a)
   +O\!\left(t+\frac d{t^2}\right).
\end{aligned}
\]
Here \(r_a\asymp t\), and the {error terms are} uniform for
\(d^{2/5}\leq t\leq\rho\). {By the choice of} \(\kappa_a\) {and subsequently taking $\rho$ sufficiently small}, {both rows become} positive, {which establishes}
\eqref{reg:eq:slow-tail}. The right repelling
tails are derived analogously by reflection and interchanging the two components.
Since \((d^{2/5})^2/d=d^{-1/5}\to\infty\), {for} every fixed \(\eta>0\) and all sufficiently small \(d\), the same tail {satisfies the stronger inequality}
\begin{equation}\label{reg:eq:slow-tail-strong}
 (\mathcal L_d-(\Lambda_I^{\mathrm{reg}}+\eta))
 \boldsymbol G_d^{a,-}
 \geq \frac c2\left(1+\frac{(a-x)^2}{d}\right)
 \boldsymbol G_d^{a,-}
\end{equation}
throughout the region in which its cutoff will vary. The large
eikonal term in \eqref{reg:eq:slow-tail-strong} {dominates} the
subtraction {employed} in the compactly supported subsolution below.

At a left attracting endpoint, both drifts are positive on a short
exterior interval. {Define}
\(P_a(x)\coloneqq -\int_a^x b_n(s)\,\mathrm{d}s>0\) for \(x<a\).
For a fixed vector \(\boldsymbol e\gg0\), the tail
\(\boldsymbol G_d^{a,-}\coloneqq
\boldsymbol e\,\mathrm{e}^{\vartheta P_a/d}\) {with} \(0<\vartheta<1\)
satisfies \eqref{reg:eq:slow-tail} and
\eqref{reg:eq:slow-tail-strong}, {provided $\rho$ is decreased further}.
Its phase derivative {is given by} \(\vartheta b_n\). {Moreover}, for either fixed level
\(\lambda=\Lambda_I^{\mathrm{reg}}\pm\eta\),
\[
 \frac{[(\mathcal L_d-\lambda)
          \boldsymbol G_d^{a,-}]_i}{G_{d,i}^{a,-}}
 =\frac{\vartheta b_n(b_i-\vartheta b_n)}d+O(1),
 \qquad i=p,n.
\]
{Upon} shrinking this exterior neighborhood, both leading expressions
{become} positive. For \(i=n\), the leading term {reads}
\(\vartheta(1-\vartheta)b_n^2/d\). On {the region} \(a-x\geq d^{2/5}\), these terms dominate all zeroth-order terms. At a
right attracting endpoint, {we instead employ}
\(P_b(x)\coloneqq -\int_b^x b_p(s)\,\mathrm{d}s>0\).

\begin{lemma}
\label{reg:lem:local-realization}
Every regular interval $I$ has local effective value
$\Lambda_I^{\mathrm{reg}}$ in the sense of
Definition~{\rm\ref{def:local-effective-value}}.
\end{lemma}

\begin{proof}
\noindent\emph{Step 1.}
Fix $\eta>0$ and a sufficiently small relative neighborhood \(U\)
of \(I=[a,b]\) containing no other static class. Choose a relative
interval \(V\) satisfying \(I\subset V\subset\overline V\subset U\) such that its
artificial endpoints lie in the exterior intervals where
\(\boldsymbol G_d^{a,-}\) and \(\boldsymbol G_d^{b,-}\) satisfy
\eqref{reg:eq:slow-tail}. These endpoints remain independent of \(d\).
We use the same \(0<\vartheta<1\) in the definitions of
\(\boldsymbol G_d^{a,-}\) and \(\boldsymbol G_d^{b,-}\), and set
\(\delta_d=d^{2/5}\). Let \(\boldsymbol F_d\) be the positive function
from \eqref{reg:eq:neutral-interior-profile}, with the common
normalization specified there. In what follows, all sums over \(\xi\) range only over
the endpoints in \(\{a,b\}\cap(\ell,r)\).

Choose smooth cutoffs with values in \([0,1]\) such that
\begin{equation}\label{reg:eq:realization-cutoffs}
\begin{aligned}
 \chi_{a,d}=0&\quad\text{for }x\geq a-\delta_d,
 &\chi_{a,d}=1&\quad\text{for }x\leq a-4\delta_d/3,\\
 \chi_{b,d}=0&\quad\text{for }x\leq b+\delta_d,
 &\chi_{b,d}=1&\quad\text{for }x\geq b+4\delta_d/3.
\end{aligned}
\end{equation}
We omit the cutoff corresponding to a domain endpoint.
Set \(\chi_{F,d}=1\) on
\([a-8\delta_d/3,b+8\delta_d/3]\cap[\ell,r]\) and
\(\chi_{F,d}=0\) outside
\((a-3\delta_d,b+3\delta_d)\cap[\ell,r]\).
We require that each cutoff satisfy \(|\chi'|\leq C\delta_d^{-1}\) and
\(|\chi''|\leq C\delta_d^{-2}\), with all derivatives vanishing at the ends of its transition intervals. We extend
\(\chi_{F,d}\boldsymbol F_d\) and
\(\chi_{\xi,d}\boldsymbol G_d^{\xi,-}\) by zero wherever the
 corresponding uncut function is undefined. 

We define the positive constants $A_{a,d}$ and $A_{b,d}$ by
\begin{equation}\label{reg:eq:realization-normalization}
 A_{a,d}G_{p,d}^{a,-}(a-2\delta_d)=F_{d,p}(a-2\delta_d),
 \qquad
 A_{b,d}G_{n,d}^{b,-}(b+2\delta_d)=F_{d,n}(b+2\delta_d),
\end{equation}
omitting the relation at a domain endpoint. For repelling endpoints, write \(S_a(x)=\int_a^x b_p(s)\,\mathrm d s\) for \(x<a\) and
\(S_b(x)=\int_b^x b_n(s)\,\mathrm d s\) for \(x>b\).
Using
\eqref{reg:eq:repelling-phase} and
\eqref{reg:eq:realization-normalization}, we derive that 
\begin{equation}\label{reg:eq:realization-repelling-comparison}
\begin{aligned}
 -d\log\frac{A_{a,d}G_{i,d}^{a,-}(a-h)}{F_{d,i}(a-h)}
 &=(1-\vartheta)
   [S_a(a-2\delta_d)-S_a(a-h)]+o(\delta_d^2),\\
 -d\log\frac{A_{b,d}G_{i,d}^{b,-}(b+h)}{F_{d,i}(b+h)}
 &=(1-\vartheta)
   [S_b(b+2\delta_d)-S_b(b+h)]+o(\delta_d^2),\,\, i=p,n.
\end{aligned}
\end{equation}
Each equality is applied only when the indicated endpoint is repelling.

At an attracting endpoint, \(\boldsymbol F_d\) converges 
to the positive extension from
Lemma~\ref{reg:lem:attracting-domain-modules}, and thus
\(-d\log F_{d,i}=O(d)\). Here
\(P_a(x)=-\int_a^x b_n(s)\,\mathrm d s\) on the left and
\(P_b(x)=-\int_b^x b_p(s)\,\mathrm d s\) on the right, and
\(\boldsymbol G_d^{\xi,-}=\boldsymbol e\exp(\vartheta P_\xi/d)\)
with a fixed \(\boldsymbol e\gg0\). Substituting these into
\eqref{reg:eq:realization-normalization} yields
\begin{equation}\label{reg:eq:realization-attracting-comparison}
\begin{aligned}
 -d\log\frac{A_{a,d}G_{i,d}^{a,-}(a-h)}{F_{d,i}(a-h)}
 &=\vartheta[P_a(a-2\delta_d)-P_a(a-h)]+o(\delta_d^2),\\
 -d\log\frac{A_{b,d}G_{i,d}^{b,-}(b+h)}{F_{d,i}(b+h)}
 &=\vartheta[P_b(b+2\delta_d)-P_b(b+h)]+o(\delta_d^2).
\end{aligned}
\end{equation}
The errors in
\eqref{reg:eq:realization-repelling-comparison}
and \eqref{reg:eq:realization-attracting-comparison} hold uniformly
for \(\delta_d\leq h\leq3\delta_d\) and \(i=p,n\).
Near a left endpoint, 
\(S_a(a-h)=b_p'(a)h^2/2+O(h^3)\) in the repelling case and
\(P_a(a-h)=-b_n'(a)h^2/2+O(h^3)\) in the attracting case.
Similarly, near a right endpoint, the corresponding expressions are
\(S_b(b+h)=b_n'(b)h^2/2+O(h^3)\) and
\(P_b(b+h)=-b_p'(b)h^2/2+O(h^3)\). In all four cases, the leading quadratic coefficients are strictly positive. Since
\(h\in[\delta_d,4\delta_d/3]\) on
\(\operatorname{supp}\chi_{\xi,d}'\) and
\(h\in[8\delta_d/3,3\delta_d]\) on
\(\operatorname{supp}\chi_{F,d}'\), there exists \(c_0>0\) such that
\begin{equation}\label{reg:eq:exponential-separation}
\begin{aligned}
 A_{\xi,d}G_{i,d}^{\xi,-}
 &\leq\exp(-c_0\delta_d^2/d)F_{d,i}
 &&\text{on }\operatorname{supp}\chi_{\xi,d}',\\
 F_{d,i}
 &\leq\exp(-c_0\delta_d^2/d)A_{\xi,d}G_{i,d}^{\xi,-}
 &&\text{on }\operatorname{supp}\chi_{F,d}'
                  \cap\{\chi_{\xi,d}=1\}.
\end{aligned}
\end{equation}

Next, we define the candidate supersolution as
\begin{equation}\label{reg:eq:realization-supersolution}
 \boldsymbol\Psi_{d,\eta}^-
 =\chi_{F,d}\boldsymbol F_d+
   \sum_\xi A_{\xi,d}\chi_{\xi,d}\boldsymbol G_d^{\xi,-}.
\end{equation}
Since at least one \(\chi_{\xi,d}\) equals one whenever \(\chi_{F,d}<1\), it follows that \(\boldsymbol\Psi_{d,\eta}^-\gg0\) on \(\overline V\) and that \(\boldsymbol\Psi_{d,\eta}^-\in C(\overline V)\cap C^2(V)\).
By virtue of \eqref{reg:eq:neutral-interior-profile} and
\eqref{reg:eq:slow-tail}, we have for small \(d\),
\begin{equation}\label{reg:eq:realization-minus-residuals}
\begin{aligned}
 (\mathcal L_d-(\Lambda_I^{\mathrm{reg}}-\eta))\boldsymbol F_d
 &\geq\frac\eta2\boldsymbol F_d,\\
 (\mathcal L_d-(\Lambda_I^{\mathrm{reg}}-\eta))
     \boldsymbol G_d^{\xi,-}
 &\geq c\left(1+{\operatorname{dist}(x,I)^2}/{d}\right)
     \boldsymbol G_d^{\xi,-}.
\end{aligned}
\end{equation}
The first inequality holds wherever \(\chi_{F,d}\neq0\), and the second holds wherever \(\chi_{\xi,d}\neq0\).

To estimate the cutoff commutators, recall that for a scalar \(\chi\) and a vector \(\boldsymbol W\),
\begin{equation}\label{reg:eq:realization-commutator}
 [[\mathcal L_d,\chi]\boldsymbol W]_i
 =-2d\chi'W_i'-d\chi''W_i-b_i\chi'W_i.
\end{equation}
On \(\operatorname{supp}\chi_{F,d}'\), the estimates in
\eqref{reg:eq:repelling-profile-ratios},
\eqref{reg:eq:repelling-corrected-C1}, and the formula
\eqref{reg:eq:repelling-inner-field} give
\(|F_{d,i}'|\leq Cd^{-1}F_{d,i}\) at repelling endpoints.
At attracting endpoints, this bound follows analogously from the positive
\(C^1\) extension and its regularization in
Lemma~\ref{reg:lem:attracting-domain-modules}.
Moreover, differentiating \eqref{reg:eq:realization-repelling-functions}
and \(e_i\exp(\vartheta P_\xi/d)\) yields
\(|(G_{i,d}^{\xi,-})'|\leq Cd^{-1}G_{i,d}^{\xi,-}\) on
\(\operatorname{supp}\chi_{\xi,d}'\). Thus, it follows from
\eqref{reg:eq:realization-commutator} that
\begin{equation}\label{reg:eq:realization-commutator-bounds}
\begin{aligned}
 |[[\mathcal L_d,\chi_{F,d}]\boldsymbol F_d]_i|
 \leq C\delta_d^{-1}F_{d,i},\quad
 |[[\mathcal L_d,\chi_{\xi,d}]\boldsymbol G_d^{\xi,-}]_i|
 \leq C\delta_d^{-1}G_{i,d}^{\xi,-}.
\end{aligned}
\end{equation}
Set \(\gamma_d:=C\delta_d^{-1}\exp(-c_0\delta_d^2/d)\to0\).
{Combining}
\eqref{reg:eq:realization-minus-residuals}-\eqref{reg:eq:realization-commutator-bounds}
{with the first inequality in} \eqref{reg:eq:exponential-separation}, {we deduce that on \(\operatorname{supp}\chi_{\xi,d}'\),}
\[
 [(\mathcal L_d-(\Lambda_I^{\mathrm{reg}}-\eta))
       \boldsymbol\Psi_{d,\eta}^-]_i
 \geq(\eta/2-\gamma_d)F_{d,i}\geq0.
\]
{On the other hand, on \(\operatorname{supp}\chi_{F,d}'\cap\{\chi_{\xi,d}=1\}\), the second inequality in} \eqref{reg:eq:exponential-separation} {yields}
\[
 [(\mathcal L_d-(\Lambda_I^{\mathrm{reg}}-\eta))
       \boldsymbol\Psi_{d,\eta}^-]_i
 \geq\left[c\left(1+
       {\operatorname{dist}(x,I)^2}/{d}\right)-\gamma_d\right]
       A_{\xi,d}G_{i,d}^{\xi,-}\geq0.
\]
{Outside the derivative supports in}
\eqref{reg:eq:realization-cutoffs}, all commutators vanish,
and \eqref{reg:eq:realization-minus-residuals} {applies directly}.
{Consequently,}
\(\mathcal L_d\boldsymbol\Psi_{d,\eta}^-
\geq(\Lambda_I^{\mathrm{reg}}-\eta)\boldsymbol\Psi_{d,\eta}^-\)
in \(V\). {Finally, at every domain endpoint,}
\(\boldsymbol\Psi_{d,\eta}^-=\boldsymbol F_d\) {in a neighborhood} and satisfies
the exact Neumann condition.

\medskip
\noindent\emph{Step 2.}
We verify the normalization and boundary inequality in
Definition~\ref{def:local-effective-value}(i). By \eqref{reg:eq:realization-repelling-comparison} and \eqref{reg:eq:realization-attracting-comparison}, 
it follows from the first identity in \eqref{reg:eq:realization-normalization}
that 
\[
\begin{aligned}
 -d\log A_{a,d}
 &=-d\log F_{d,p}(a-2\delta_d)
      +d\log G_{p,d}^{a,-}(a-2\delta_d)\\
 &=\begin{cases}
 (1-\vartheta)S_a(a-2\delta_d)+o(\delta_d^2),&a\text{ repelling},\\
 \vartheta P_a(a-2\delta_d)+O(d),&a\text{ attracting}.
 \end{cases}
\end{aligned}
\]
Analogously, the second identity provides the corresponding formula for
$-d\log A_{b,d}$, with \(S_a(a-2\delta_d)\) and
\(P_a(a-2\delta_d)\) replaced by \(S_b(b+2\delta_d)\) and
\(P_b(b+2\delta_d)\), respectively. Hence, we obtain $-d\log A_{\xi,d}=O(\delta_d^2)=o(1)$.
In view of \eqref{reg:eq:realization-supersolution}, the functions \(\boldsymbol\Psi_{d,\eta}^-\) and \(\boldsymbol F_d\) coincide on \(I\). Therefore, applying the normalization
specified after \eqref{reg:eq:neutral-interior-profile}, we deduce that
\(\max_{i=p,n}\|-d\log\Psi_{d,\eta,i}^-\|_{C(I)}\to0\).

For \(x\in\partial V\cap(\ell,r)\), we set \(\xi=a\) if \(x<a\) and
\(\xi=b\) if \(x>b\). Since \(d\) is small, we have
\(\chi_{F,d}(x)=0\) and \(\chi_{\xi,d}(x)=1\), which gives
\(\Psi_{d,\eta,i}^-(x)=A_{\xi,d}G_{i,d}^{\xi,-}(x)\).
At such points, the terms \(d\kappa_a/[-b_p(x)]\) and
\(d\kappa_b/b_n(x)\) in
\eqref{reg:eq:realization-repelling-functions} contribute an error of order
\(O(d|\log d|)\) to \(-d\log G_{i,d}^{\xi,-}(x)\).
Applying \eqref{eq:explicit-critical-distance}, we obtain
\[
 -d\log\Psi_{d,\eta,i}^-(x)
 =\begin{cases}
 \vartheta d_H(x,I)+o(1),&\xi\text{ repelling},\\
 -\vartheta P_\xi(x)+o(1),&\xi\text{ attracting},
 \end{cases}
 \qquad i=p,n.
\]
In this expression, \(d_H(x,I)=S_\xi(x)>0\) in the repelling case, whereas
\(d_H(x,I)=0\) and \(P_\xi(x)>0\) in the attracting case. Consequently,
\(d_H(x,I)+d\log\Psi_{d,\eta,i}^-(x)\) {converges to}
\((1-\vartheta)d_H(x,I)>0\) or \(\vartheta P_\xi(x)>0\),
respectively. By choosing
\(\delta_0\) 
small we then have
\[
 \max_{i=p,n}[-d\log\Psi_{d,\eta,i}^-(x)]
 \leq d_H(x,I)-\delta_0
 \quad\text{on }\partial V\cap(\ell,r).
\]
When both endpoints are domain endpoints, this boundary set is empty.
This completes the verification of Definition~\ref{def:local-effective-value}(i).

\medskip
\noindent\emph{Step 3.}
Using the cutoffs in \eqref{reg:eq:realization-cutoffs} and the
constants in \eqref{reg:eq:realization-normalization}, we define
\begin{equation}\label{reg:eq:realization-subsolution}
 \boldsymbol q_d
 =\chi_{F,d}\boldsymbol F_d-
   \sum_\xi A_{\xi,d}\chi_{\xi,d}\boldsymbol G_d^{\xi,-},
 \qquad
 \boldsymbol Q_{d,\eta}=(\boldsymbol q_d)_+.
\end{equation}
The positive part is taken componentwise. On \(I\),
\(\boldsymbol q_d=\boldsymbol F_d\gg0\). For \(x\leq a-3\delta_d\)
in \(V\), \(\boldsymbol q_d=-A_{a,d}\boldsymbol G_d^{a,-}\ll0\),
and for \(x\geq b+3\delta_d\) in \(V\),
\(\boldsymbol q_d=-A_{b,d}\boldsymbol G_d^{b,-}\ll0\).
Hence, \(\boldsymbol Q_{d,\eta}\not\equiv0\) and
\(
 \operatorname{supp}\boldsymbol Q_{d,\eta}
 \subset[a-3\delta_d,b+3\delta_d]\cap[\ell,r]
 \subset V^Q,
\) 
where \(V^Q\) is a fixed relative neighborhood satisfying
\(I\subset V^Q\subset\overline{V^Q}\subset V\).

By \eqref{reg:eq:neutral-interior-profile}, for small \(d\) we have
\((\mathcal L_d-(\Lambda_I^{\mathrm{reg}}+\eta))\boldsymbol F_d
\leq-\eta\boldsymbol F_d/2\).
{Wherever} \(\chi_{\xi,d}\neq0\), {we have}
\(\operatorname{dist}(x,I)\geq\delta_d\).
Subtracting \(5\eta\boldsymbol G_d^{\xi,-}/4\) from
\eqref{reg:eq:slow-tail} and using \(\delta_d^2/d\to\infty\), we derive
\begin{equation}\label{reg:eq:realization-plus-residual}
 (\mathcal L_d-(\Lambda_I^{\mathrm{reg}}+\eta))
       \boldsymbol G_d^{\xi,-}
 \geq\frac c2\left(1+
       {\operatorname{dist}(x,I)^2}/{d}\right)
       \boldsymbol G_d^{\xi,-}
 \quad\text{where }\chi_{\xi,d}\neq0.
\end{equation}
{Applying} \eqref{reg:eq:realization-commutator} to
\eqref{reg:eq:realization-subsolution}, {and then using}
\eqref{reg:eq:exponential-separation} and
\eqref{reg:eq:realization-commutator-bounds}, {we deduce}
\[
 [(\mathcal L_d-(\Lambda_I^{\mathrm{reg}}+\eta))
       \boldsymbol q_d]_i
 \leq(-\eta/2+\gamma_d)F_{d,i}\leq0
 \quad\text{on }\operatorname{supp}\chi_{\xi,d}'.
\]
{On the other hand, on} \(\operatorname{supp}\chi_{F,d}'\cap\{\chi_{\xi,d}=1\}\),
the same formulas {together with} \eqref{reg:eq:realization-plus-residual} {yield}
\[
 [(\mathcal L_d-(\Lambda_I^{\mathrm{reg}}+\eta))
       \boldsymbol q_d]_i
 \leq\left[-\frac c2\left(1+
       {\operatorname{dist}(x,I)^2}/{d}\right)+\gamma_d\right]
       A_{\xi,d}G_{i,d}^{\xi,-}\leq0.
\]
Combining the inequality for \(\boldsymbol F_d\) with
\eqref{reg:eq:realization-plus-residual}, we conclude that
\((\mathcal L_d-(\Lambda_I^{\mathrm{reg}}+\eta))
\boldsymbol q_d\leq\boldsymbol0\) throughout \(V\).
Consequently, \(\boldsymbol Q_{d,\eta}\) satisfies
Definition~\ref{def:local-effective-value}(ii).
\end{proof}


We conclude this section by proving Corollary \ref{cor:quadratic-potentials}.

\begin{proof}[Proof of Corollary {\rm\ref{cor:quadratic-potentials}}]
Since \(b_1(x)b_2(x)=4kh(x-x_1)(x-x_2)\),
{equation} \eqref{eq:main-Aubry-set} {yields} the stated static classes. The signs
of \(b_1'=2k\) and \(b_2'=2h\) dictate their endpoint types and the
frozen matrices. By substituting this information into
\eqref{reg:eq:transport-system}-\eqref{reg:eq:endpoint-conditions}
and applying Theorem~\ref{thm:main-one-dimensional}, we deduce the
corresponding limits stated in the corollary.
\end{proof}

\section{\bf Composite interval classes and common switching points}
\label{sec:composite-intervals}

Let \(K=[\alpha,\beta]\) be a composite static class, with common zeros
\(\mathscr Z_K=\{z_1<\cdots<z_q\}\), \(q\geq1\). Set
\(z_0=\alpha\), \(z_{q+1}=\beta\), and {let} \(G_j=(z_j,z_{j+1})\). At a common zero \(z\),
simplicity and \(b_1b_2\leq0\) {on either side guarantee}
\(b_1'(z)b_2'(z)<0\). Since an interior zero of only one drift would change
the sign of \(b_1b_2\), no such zero can lie in \(K\).
Hence, each cell admits fixed indices \(\{p,n\}=\{1,2\}\) with
\(b_p>0>b_n\), which interchange across a common zero.
Throughout this section, vectors and matrices retain the original
component ordering. The symbols \(p,n\) denote indices, not a permutation of a
vector. We denote by
\(\mathscr T=-\operatorname{diag}(b_1,b_2)\partial_x+\boldsymbol M\)
the transport operator.

\subsection{The mixed Ornstein-Uhlenbeck value}

For \(z\in\mathscr Z_K\), we define
\begin{equation}\label{comp:eq:node-value}
 \rho_z(\vartheta)=\sigma(\boldsymbol M(z)+\vartheta\boldsymbol B_z),
 \qquad
 \omega_z=\max_{0\leq\vartheta\leq1}\rho_z(\vartheta),
\end{equation}
where
\(\boldsymbol B_z=\operatorname{diag}(b_1'(z),b_2'(z))\). 
Let \(\boldsymbol v_z(\vartheta)\gg0\) be the corresponding Perron
vector, normalized by \(v_{z,1}+v_{z,2}=1\).

\begin{lemma}\label{comp:lem:mixed-OU}
Let \(\omega_{z,R}\) denote the principal Dirichlet eigenvalue of the frozen operator 
\(\mathscr A_z=-\partial_{yy}-\boldsymbol B_zy\partial_y+\boldsymbol M(z)\) on \((-R,R)\). 
Then  \(\omega_{z,R}\searrow\omega_z\) as \(R\to\infty\).
Furthermore, for any \(\vartheta\in(0,1)\) and \(\nu<\rho_z(\vartheta)\), there
exists an even function \(\boldsymbol\Psi\in C^2(\mathbb R;(0,\infty)^2)\) such that
\begin{equation}\label{comp:eq:OU-power-super}
 \mathscr A_z\boldsymbol\Psi\geq\nu\boldsymbol\Psi,\qquad
 \partial_y^j\!\left[
 \Psi_i(y)-|y|^{-\vartheta}v_{z,i}(\vartheta)\right]
 =o(|y|^{-\vartheta-j}),\quad j=0,1,2,
\end{equation}
as \(|y|\to\infty\), uniformly for \(i=1,2\).
\end{lemma}

\begin{proof}
\noindent\emph{Step 1.}
We first establish the limits of \(\omega_{z,R}\). Let \(\boldsymbol U^R\) be a positive
Dirichlet eigenfunction associated with \(\omega_{z,R}\). For \(m>1\),set \(X_i:=\|U_i^R\|_{L^m(-R,R)}\).
Multiplying the \(i\)-th
equation by \((U_i^R)^{m-1}\) and integrating by parts, we obtain
\[
 \begin{aligned}
 \omega_{z,R}X_i^m
 ={}&(m-1)\int_{-R}^R
       (U_i^R)^{m-2}|(U_i^R)'|^2\,\mathrm{d}y
   +\alpha B_{ii,z}X_i^m+M_{ii,z}X_i^m \\
 &+\sum_{j\ne i}M_{ij,z}
   \int_{-R}^R U_j^R(U_i^R)^{m-1}\,\mathrm{d}y.
 \end{aligned}
\]
The Dirichlet condition eliminates the boundary terms, whereas
H\"older's inequality yields
\[
 \int_{-R}^R U_j^R(U_i^R)^{m-1}\,\mathrm{d}y
 \leq X_jX_i^{m-1}.
\]
In view of \(M_{ij,z}<0\) for \(i\ne j\), we deduce that
\[
 (\boldsymbol M_z+\alpha\boldsymbol B_z)\boldsymbol X
 \leq\omega_{z,R}\boldsymbol X,
 \qquad
 \boldsymbol X=(X_p,X_n)^{\mathsf T}\gg\boldsymbol 0.
\]
Then the matrix Collatz-Wielandt formula implies
\(\rho_z(\alpha)\leq\omega_{z,R}\). The cases \(\alpha=0,1\)
follow by continuity. Hence, \(\omega_z\leq\omega_{z,R}\).

For the reverse bound, we set \(a_i(\vartheta)=M_{ii}(z)+\vartheta B_{ii,z}\).
The formula
\[
 \rho_z(\vartheta)=\frac{a_1(\vartheta)+a_2(\vartheta)}2
 -\sqrt{\frac{(a_1(\vartheta)-a_2(\vartheta))^2}{4}
             +M_{12}(z)M_{21}(z)}
\]
shows that \(\rho_z\) is strictly concave and tends to \(-\infty\)
at both ends of the real line. Let \(\overline\vartheta\) denote its
unique maximizer. If \(\overline\vartheta\leq1\), we take
\(\widehat{\mathscr A}=\mathscr A_z\). Otherwise we take its formal adjoint,
so that
\(\widehat{\boldsymbol B}=-\boldsymbol B_z\) and
\(\widehat{\boldsymbol M}=\boldsymbol M(z)^{\mathsf T}+\boldsymbol B_z\)
in
\(\widehat{\mathscr A}=-\partial_{yy}
-\widehat{\boldsymbol B}y\partial_y+\widehat{\boldsymbol M}\).
In the two cases,
\(\widehat\rho(s)=\sigma(\widehat{\boldsymbol M}+s\widehat{\boldsymbol B})\)
equals \(\rho_z(s)\) and \(\rho_z(1-s)\), respectively. Hence, we deduce
\begin{equation}\label{comp:eq:adjoint-rho-sup}
 \sup_{s\geq0}\widehat\rho(s)=\omega_z.
\end{equation}
We choose indices \(p,n\) so that
\(\widehat B_{pp}=\beta_p>0\), \(\widehat B_{nn}=-\beta_n<0\).
We normalize the Perron vector \(\boldsymbol r_s\) by \(r_{s,n}=1\).
Using the identities
\[
 r_{s,p}=\frac{-\widehat M_{pn}}
 {\widehat M_{pp}+\beta_ps-\widehat\rho(s)},\qquad
 \widehat\rho(s)=\widehat M_{nn}-\beta_ns+\widehat M_{np}r_{s,p},
\]
we deduce that, as \(s\to\infty\),
\begin{equation}\label{comp:eq:adjoint-Perron-asymptotics}
 \begin{aligned}
 \widehat\rho(s)&=\widehat M_{nn}-\beta_ns+O(s^{-1}),&
 r_{s,p}&=\frac{-\widehat M_{pn}}{(\beta_p+\beta_n)s}
                  (1+O(s^{-1})),\\
 \partial_s\log r_{s,p}&=-s^{-1}+O(s^{-2}),&
 \frac{\partial_{ss}r_{s,p}}{r_{s,p}}&=2s^{-2}+O(s^{-3}).
 \end{aligned}
\end{equation}
In particular, \(\partial_s\log r_{s,p}\) is bounded for \(s\geq0\).

Fix \(\varepsilon>0\). We choose
\(0<\gamma<\min\{1/2,\beta_n/(2\beta_p)\}\) to be  small,
and subsequently $\Theta>1$  large. Let
\(s\in C^2((-\infty,T))\) be a nondecreasing function, which is zero for \(t\leq T_0\), and satisfies
\[
 \begin{aligned}
 0\leq s'\leq\gamma,\quad |s''|\leq C\gamma^2
       &&(T_0\leq t\leq T_1),\\
 s(t)=\Theta-1+\frac1{\gamma(T-t)},\quad
 T=T_1+\gamma^{-1}
       &&(T_1\leq t<T).
 \end{aligned}
\]
The first part is chosen to reach \(s(T_1)=\Theta\) and to match
the first two derivatives of the {final} part. {It can be translated} in \(t\) after its shape has been fixed.
{Set} \(H(t)=\int_{T_0}^t s(u)\,\mathrm{d}u\) and, for
\(0<|y|<\exp(T)\), {define}
\(\boldsymbol Z(y)=\exp(-H(\log|y|))\boldsymbol r_{s(\log|y|)}\).
It is constant near \(y=0\). {A direct calculation yields}
\begin{equation}\label{comp:eq:OU-compact-ratio}
 \frac{(\widehat{\mathscr A}\boldsymbol Z)_i}{Z_i}
 =\widehat\rho(s)-\widehat B_{ii}s'\partial_s\log r_{s,i}
       -\mathrm{e}^{-2t}\mathcal R_i,
\end{equation}
where the function \(\mathcal R_i\) is defined by
\[
 \mathcal R_i=s^2+s-s'
 +[s''-(2s+1)s']\partial_s\log r_{s,i}
 +(s')^2\frac{\partial_{ss}r_{s,i}}{r_{s,i}}.
\]
On \([T_0,T_1]\), the second term in
\eqref{comp:eq:OU-compact-ratio} is \(O(\gamma)\), so that  \(\mathcal R_i\)
is bounded once \(\Theta\) and the transition have been fixed.
A sufficiently large translation of \(T_0\), together with
\eqref{comp:eq:adjoint-rho-sup}, {bounds} the whole expression by
\(\omega_z+\varepsilon\).
On \([T_1,T)\), {we set} \(w=s-\Theta+1\), so \(s'=\gamma w^2\),
\(s''=2\gamma^2w^3\), and \(1\leq w\leq s\).
For the \(n\)-component,
\(\mathcal R_n=s^2+s-\gamma w^2>0\).
For the \(p\)-component, by \eqref{comp:eq:adjoint-Perron-asymptotics}
we have
\[
 \mathcal R_p
 =s^2\left[1+\gamma(w/s)^2
       -2\gamma^2(w/s)^3(1-w/s)\right]+O(s)
 \geq s^2-Cs>0
\]
uniformly in \(w/s\in(0,1]\), after \(\Theta\) is enlarged. The remaining terms satisfy
\[
 \widehat\rho(s)-\beta_ps'\partial_s\log r_{s,p}
 \leq\widehat M_{nn}-(\beta_n-\gamma\beta_p)s+C
 \leq\omega_z.
\]
The corresponding correction in the \(n\)-component is zero.
{Thus}, \(\widehat{\mathscr A}\boldsymbol Z
\leq(\omega_z+\varepsilon)\boldsymbol Z\).
As \(t\nearrow T\), \(\exp(-H(t))\) vanishes to order \(1/\gamma>2\),
while \(r_{s(t),p}=O(T-t)\). The explicit final part of \(s(t)\)
gives the same vanishing for the first two \(y\)-derivatives.
Hence, extending \(Z\) by zero yields
\(\boldsymbol Z\in C_c^2(\mathbb R;\mathbb R^2)\).
{Pairing this inequality with} the positive adjoint Dirichlet
eigenfunction on \((-R,R)\) {for} \(R>\exp(T)\) {yields}
\(\omega_{z,R}\leq\omega_z+\varepsilon\).
This proves \(\omega_{z,R}\searrow\omega_z\) as \(R\to \infty\).

\medskip
\noindent\emph{Step 2.}
We now establish \eqref{comp:eq:OU-power-super}.
For \(|y|\geq1\), we set
\(\boldsymbol P(y)=|y|^{-\vartheta}\boldsymbol v_z(\vartheta)\).
Then
\begin{equation}\label{comp:eq:OU-exterior-power}
 (\mathscr A_z-\nu)\boldsymbol P
 =\left(\rho_z(\vartheta)-\nu
        -\frac{\vartheta(\vartheta+1)}{y^2}\right)\boldsymbol P.
\end{equation}
By extending \(\boldsymbol P\) to a positive even \(C^2\) function
\(\widetilde{\boldsymbol P}\), without altering it for large \(|y|\), we deduce from \eqref{comp:eq:OU-exterior-power} that we may select an even
\(\boldsymbol h\in C_c^\infty(\mathbb R;[0,\infty)^2)\) with
\(\boldsymbol h\geq-(\mathscr A_z-\nu)\widetilde{\boldsymbol P}\).
Since \(\nu<\omega_z\leq\omega_{z,R}\), the Dirichlet problem
\[
 (\mathscr A_z-\nu)\boldsymbol G_R=\boldsymbol h
       \quad\hbox{in }(-R,R),\qquad
 \boldsymbol G_R(\pm R)=0
\]
admits a nonnegative solution, which is increasing with \(R\).
By the Harnack inequality and local elliptic estimates, we obtain a positive entire solution
\(\mathscr A_z\boldsymbol W=\omega_z\boldsymbol W\).
Since \(\boldsymbol h\) is compactly supported, there exists some \(C\) such that
\((\mathscr A_z-\nu)C\boldsymbol W\geq\boldsymbol h\).
The comparison principle yields \(\boldsymbol G_R\leq C\boldsymbol W\).
Thus, \(\boldsymbol G_R\to\boldsymbol G\) locally in \(C^2\), where
\(\boldsymbol G\) is even, nonnegative, and
\((\mathscr A_z-\nu)\boldsymbol G=\boldsymbol h\).

We choose \(\vartheta_+\in(\vartheta,1)\) with
\(\rho_z(\vartheta_+)>\nu\).
By \eqref{comp:eq:OU-exterior-power}, we observe that
\(|y|^{-\vartheta_+}\boldsymbol v_z(\vartheta_+)\) is a strict
supersolution outside a fixed interval. A fixed multiple of this function dominates
the traces of all \(\boldsymbol G_R\) at the inner boundaries and
their zero traces at \(\pm R\). By the comparison principle, we deduce
\(G_i(y)=O(|y|^{-\vartheta_+})\).

To control the derivatives, we write \(G_i''+B_{ii,z}yG_i'=F_i\)  on a sufficiently large positive half-line, where
\(F_i=\sum_jM_{ij}(z)G_j-\nu G_i=O(y^{-\vartheta_+})\).
Multiplying by \(\exp(B_{ii,z}y^2/2)\) and integrating
from a fixed point if \(B_{ii,z}>0\), or from infinity if
\(B_{ii,z}<0\), yields \(G_i'=O(y^{-\vartheta_+-1})\).
In the latter case, the polynomial bound on \(G_i\) excludes the
growing homogeneous solution. Hence, \(F_i'=O(y^{-\vartheta_+-1})\), and integration by parts in the same formulas yields
\[
 G_i'(y)=\frac{F_i(y)}{B_{ii,z}y}
           +O(y^{-\vartheta_+-3}),\qquad
 G_i''(y)=O(y^{-\vartheta_+-2}).
\]
By evenness, the same estimates hold on the negative half-line.
Therefore, \(\boldsymbol\Psi=\widetilde{\boldsymbol P}+\boldsymbol G\)
satisfies \eqref{comp:eq:OU-power-super}. The proof is complete.
\end{proof}

\subsection{Transport cells}

Let \(G=(a,b)\) be a cell on which \(b_p>0>b_n\), and choose an
admissible finite section \(D=G_{\boldsymbol\delta}\) as in
Definition~\ref{def:main-class-values}. We prescribe zero incoming
trace at each artificial endpoint. If an endpoint of \(G\) is left
unchanged, we retain \eqref{reg:eq:endpoint-conditions}, together with
the corresponding admissibility condition. Let
\(\lambda_{\partial,G}\) denote the minimum of the ceilings associated
with the unchanged endpoints, and set \(\lambda_{\partial,G}=+\infty\)
when neither endpoint is retained. Since every cell in a composite
class has a common zero at one or both ends, every finite section has
at least one artificial endpoint. In particular, its two endpoints
cannot both be repelling.

Applying Lemma~\ref{reg:lem:characteristic-operators} on \(D\), with
zero incoming data at the artificial endpoints, yields a compact
positive feedback operator \(\mathcal K_{D,\lambda}\) for
\(\lambda<\lambda_{\partial,G}\). Its spectral radius is continuous
and strictly increasing in \(\lambda\), and tends to zero as
\(\lambda\to-\infty\). If one endpoint is unchanged, the rank-one
lower bound near its ceiling remains valid because the incoming datum
at the opposite artificial endpoint is zero. When both endpoints are
artificial, we instead restrict the two characteristic integrations
to separated interior subintervals. For a fixed \(0\ne f\geq0\), {we obtain} \(\mathcal K_{D,\lambda}f\geq C\exp(c\lambda)f\) once
\(\lambda\) is large. The spectral radius therefore crosses \(1\)
exactly once. We denote the crossing by
\(\Lambda_D=\Lambda_{G,\boldsymbol\delta}<\lambda_{\partial,G}\)
and the corresponding positive transport mode by \(\boldsymbol U_D\).
For \(\lambda<\Lambda_D\), the same characteristic formulas, together
with the positive Neumann series for
\((I-\mathcal K_{D,\lambda})^{-1}\), yield a solution to
\((\mathscr T-\lambda)\boldsymbol V=\boldsymbol h\) for every
nonnegative compactly supported \(\boldsymbol h\), subject to the
homogeneous endpoint data.

We shall use the following elementary contact observation several
times. Suppose that \(\boldsymbol Z\geq0\), with derivatives at the
endpoints understood in the one-sided sense. At an interior zero of
\(Z_i\), one has \(Z_i'=0\). If \(Z_p(a)=0\) at a left endpoint, then
\(Z_p'(a)\geq0\) and \(b_p(a)\geq0\). Likewise, \(Z_n(b)=0\) at a
right endpoint yields \(Z_n'(b)\leq0\) and \(b_n(b)\leq0\). Since
\(M_{ij}<0\) for \(i\ne j\), in each case the vanishing component
satisfies
\begin{equation}\label{comp:eq:cell-contact-test}
 [ (\mathscr T-\lambda)\boldsymbol Z ]_i\leq0.
\end{equation}
The other possible endpoint contacts will be excluded by the endpoint
conditions. More precisely, at an unchanged left attracting or domain
endpoint, the strict boundary residual
\((M_{nn}-\lambda)Z_n+M_{np}Z_p>0\) is incompatible with \(Z_n=0\), since the left-hand side would then reduce to \(M_{np}Z_p\leq0\). The reflected
argument applies at the right endpoint. At a repelling endpoint, the
power asymptotics will prevent the contact point from reaching the endpoint.

\begin{lemma}\label{comp:lem:cell-exhaustion}
As the truncated interval expands, the associated value decreases and converges to a finite limit \(\Lambda_G^{\mathrm{tr}}\) that is independent of the chosen exhaustion. Upon normalizing the positive eigenfunctions at a fixed interior point, a subsequence converges in \(C^1_{\mathrm{loc}}(G)\) to a function \(\boldsymbol W_G\gg0\) satisfying
\begin{equation}\label{comp:eq:limiting-cell-mode}
 \mathscr T\boldsymbol W_G=\Lambda_G^{\mathrm{tr}}\boldsymbol W_G
 \quad\hbox{in }G,
\end{equation}
which also obeys the prescribed admissibility conditions at every unchanged endpoint.
\end{lemma}

\begin{proof}
Take two finite sections \(D_1\subset D_2\), with eigenpairs
\((\lambda_k,\boldsymbol U^k)\), and suppose for contradiction that
\(\lambda_2>\lambda_1\). At an unchanged repelling endpoint, the exponent
\(q(\lambda)=(M_{ii}-\lambda)/b_i'\), where \(b_i'>0\), decreases
strictly with \(\lambda\). By
Lemma~\ref{reg:lem:repelling-endpoint-asymptotics}, we deduce
\(U_i^2/U_i^1\to+\infty\) in both components. At a shared
artificial left endpoint \(c\),
\((U_n^k)'(c)=M_{np}(c)U_p^k(c)/b_n(c)>0\), whence
\[
 \lim_{x\searrow c}\frac{U_n^2(x)}{U_n^1(x)}
   =\frac{U_p^2(c)}{U_p^1(c)}.
\]
Thus, if the minimum ratio is attained in the incoming component at a
shared artificial endpoint, it is also attained in the outgoing
component there.
The right-end relation follows analogously by reflection. If an artificial
endpoint of \(D_1\) lies in the interior of \(D_2\), the ratio of the
incoming components tends to \(+\infty\). At every remaining endpoint
the ratios have finite positive limits. Hence, the componentwise ratios
on \(\overline D_1\), extended by these endpoint limits, attain a
positive minimum \(t\). The function
\(\boldsymbol Z=\boldsymbol U^2-t\boldsymbol U^1\) is nonnegative,
has a zero in at least one component, and satisfies
\[
 (\mathscr T-\lambda_2)\boldsymbol Z
   =t(\lambda_2-\lambda_1)\boldsymbol U^1\gg0.
\]
At an unchanged left attracting or domain endpoint, its boundary
residual is
\[
 (M_{nn}-\lambda_2)Z_n+M_{np}Z_p
   =t(\lambda_2-\lambda_1)U_n^1>0,
\]
with the analogous identity on the right. The contact observation
\eqref{comp:eq:cell-contact-test}, the endpoint residuals, and the
repelling asymptotics now exclude every possible location of the
minimum. This contradiction proves \(\lambda_2\leq\lambda_1\).

To establish a uniform lower bound, we choose \(\boldsymbol e\gg0\) and a
sufficiently negative \(\lambda_0\) such that
\((\boldsymbol M(x)-\lambda_0I)\boldsymbol e\gg0\) on \(\overline G\).
Decrease \(\lambda_0\), if necessary, so that all exponents at the
unchanged repelling endpoints are positive whenever
\(\Lambda_D<\lambda_0\).Under this condition, both components of
\(\boldsymbol U_D\) vanish at every such endpoint. Assigning the value \(+\infty\) at a zero trace, the ratios
\(e_i/U_{D,i}\) therefore attain a positive minimum. Subtracting the corresponding
multiple of \(\boldsymbol U_D\) from \(\boldsymbol e\) produces a
nonnegative function with a strictly positive differential residual.
Its incoming traces at artificial endpoints are positive. At an
attracting or domain endpoint, the incoming boundary residual is the
corresponding row of
\((\boldsymbol M-\Lambda_DI)\boldsymbol e\), which is also positive.
The contact observation therefore yields a contradiction. Consequently
\(\Lambda_D\geq\lambda_0\) for every finite section.

Let \(L=\inf_D\Lambda_D\). Given \(\varepsilon>0\), choose a finite
section \(D_0\) such that \(\Lambda_{D_0}<L+\varepsilon\). Every
exhaustion eventually contains \(D_0\), and monotonicity places all
subsequent values in \([L,L+\varepsilon)\). Hence every exhaustion
converges to the same finite number
\(\Lambda_G^{\mathrm{tr}}=L\).

Fix \(x_*\in D_0\) and normalize the eigenfunctions by
\(U_{D,p}(x_*)+U_{D,n}(x_*)=1\). For \(D\supset D_0\),
\begin{equation}\label{comp:eq:cell-uniform-spectral-window}
 \lambda_0\leq\Lambda_D\leq\Lambda_{D_0}
       <\lambda_{\partial,G}.
\end{equation}
On each compact subinterval of \(G\), the matrices
\(\operatorname{diag}(b_1^{-1},b_2^{-1})
(\boldsymbol M-\Lambda_DI)\) are uniformly bounded. The transport
system and the normalization at \(x_*\) therefore provide uniform
\(C^1\)-bounds. {By the Arzel\`a--Ascoli theorem}, a subsequence converges in
\(C^1_{\mathrm{loc}}(G)\) to a nonzero function
\(\boldsymbol W_G\geq0\), which satisfies
\eqref{comp:eq:limiting-cell-mode}. If one component vanished at an
interior point, nonnegativity and the corresponding transport equation
would force the other component to vanish there as well. Uniqueness
for the first-order system would then contradict the normalization.
Thus \(\boldsymbol W_G\gg0\) in \(G\).

It remains to pass to the endpoints of \(G\) that are unchanged by the
exhaustion. At a domain endpoint, the standard continuous dependence for
the transport system extends the convergence to a fixed one-sided
neighborhood. At an attracting endpoint, the same conclusion follows
from \eqref{reg:eq:left-attracting-representation} and its reflected
form. {The inequality}
\eqref{comp:eq:cell-uniform-spectral-window} {keeps} the corresponding
exponents uniformly positive, so dominated convergence applies in
these formulas. At a repelling endpoint, we use instead
\eqref{reg:eq:left-repelling-volterra}--%
\eqref{reg:eq:left-repelling-weighted-p}. Since the repelling exponents remain
uniformly above \(-1\), we obtain convergence in the usual
norm at domain and attracting endpoints, and in the appropriate
weighted norm at repelling endpoints. The leading coefficient of the
limit cannot vanish identically. Indeed, the homogeneous local representation
would then make the solution identically zero near that endpoint, and
interior uniqueness would contradict its normalization at \(x_*\).
All unchanged endpoint conditions consequently pass to the limit.
\end{proof}

We now restore diffusion on a fixed truncated interval. The artificial
endpoints are treated with Dirichlet layers, whereas the local models from
Section~\ref{sec:regular-intervals} remain applicable at every
unchanged endpoint.

\begin{lemma}\label{comp:lem:finite-section-viscous-subsolution}
Fix \(D=G_{\boldsymbol\delta}\) and \(\tau>0\).
For all sufficiently small \(d\), there exists a nonzero, nonnegative,
compactly supported subsolution
\(\boldsymbol Q_d\in H^1((\ell,r);\mathbb R^2)\), which is positive in \(D\)
and supported in any prescribed relative neighborhood of
\(\overline D\), such that
\begin{equation}\label{comp:eq:finite-section-viscous-subsolution}
 \mathcal L_d\boldsymbol Q_d
 \leq(\Lambda_D+\tau)\boldsymbol Q_d
\end{equation}
understood in the standard sense for compactly supported subsolutions of the
 Neumann problem.
\end{lemma}

\begin{proof}
Write \(D=(x_L,x_R)\) and let
\(\boldsymbol U\) be its positive transport solution.

\noindent\emph{Step 1.}
We first treat an artificial left endpoint, i.e. \(U_n(x_L)=0\). Introduce
\(s=(x-x_L)/d\) and set
\(\beta_p=b_p(x_L)>0\) and
\(\beta_n=-b_n(x_L)>0\). The transport equations yield
\[a_p=U_p'(x_L)=(M_{pp}(x_L)-\Lambda_D)U_p(x_L)/\beta_p, \quad a_n=U_n'(x_L)=-M_{np}(x_L)U_p(x_L)/\beta_n>0.\]
In the formulas below, coefficients without an explicit spatial
argument are evaluated at \(x_L\). Let
\(\mathcal B_i=-\partial_{ss}-b_i(x_L)\partial_s\), and define
\begin{equation}\label{comp:eq:left-artificial-boundary-functions}
 \begin{aligned}
 \Phi_p&=(1-{\rm e}^{-\beta_ps})U_p(x_L),&
 \Gamma_p&=a_ps(1+{\rm e}^{-\beta_ps})
       +b_p'{\rm e}^{-\beta_ps}(s^2/2+s/\beta_p)U_p(x_L),\\
 \Phi_n&=0,&
 \Gamma_n&=a_ns+C_L(1-{\rm e}^{-\beta_ps}),
 \end{aligned}
\end{equation}
where \(C_L=M_{np}A/[\beta_p(\beta_p+\beta_n)]<0\).
A direct calculation verifies that \(\mathcal B\boldsymbol\Phi=0\) and
\[
 \mathcal B\boldsymbol\Gamma
 =s\operatorname{diag}(b_1',b_2')\boldsymbol\Phi'
  -(\boldsymbol M-\Lambda_D I)\boldsymbol\Phi.
\]
In addition, it follows that
\begin{equation}\label{comp:eq:left-artificial-n-lower-bound}
 \Gamma_n(s)
 =a_n\left[s-\frac{\beta_n(1-\exp(-\beta_ps))}
    {\beta_p(\beta_p+\beta_n)}\right]
 \geq\frac{a_n\beta_p}{\beta_p+\beta_n}s.
\end{equation}
To further estimate 
 the \(n\)-component, 
we select \(\boldsymbol\Theta\) as the solution of
\begin{equation}\label{comp:eq:left-artificial-second-corrector-equation}
 \begin{aligned}
 \mathcal B\boldsymbol\Theta
 ={}&s\operatorname{diag}\bigl(b_p'(x_L),b_n'(x_L)\bigr)
          \boldsymbol\Gamma'+\frac{s^2}{2}
   \operatorname{diag}\bigl(b_p''(x_L),b_n''(x_L)\bigr)
          \boldsymbol\Phi'\\
 &-\bigl(\boldsymbol M(x_L)-
          \Lambda_D\boldsymbol I_2\bigr)
          \boldsymbol\Gamma
   -s\boldsymbol M'(x_L)\boldsymbol\Phi,
 \end{aligned}
\end{equation}
with \(\Theta_p(0)=\Theta_p'(0)=\Theta_n(0)=0\), while omitting the growing
\({\rm e}^{-\beta_ns}\)-mode from \(\Theta_n\).
The right-hand side is a polynomial of degree at most one plus
exponentially decaying polynomial terms. Integrating factors thus yield, for \(s\geq0\),
\[
 |\boldsymbol\Theta(s)|\leq Cs(1+s),\qquad
 |\boldsymbol\Theta'(s)|\leq C(1+s),\qquad
 |\boldsymbol\Theta''(s)|\leq C.
\]
Set \(S_d=c|\log d|\), where \(c\kappa>2\) and
\(\kappa=\min\{\beta_p,\beta_n\}\), and set
\[
 \boldsymbol F_d^L(x_L+ds)
 =\boldsymbol\Phi(s)+d\boldsymbol\Gamma(s)
  +d^2\boldsymbol\Theta(s),
 \qquad 0\leq s\leq S_d^2.
\]
For small \(d\), we deduce from \eqref{comp:eq:left-artificial-n-lower-bound} that \( F_{d,p}^L\geq c_1\min\{s,1\}\), \(F_{d,n}^L\geq c_1ds\), and
\[
 \begin{gathered}
 (F_{d,p}^L)'(x_L)=A\beta_p/d+O(1)>0,\qquad
 (F_{d,n}^L)'(x_L)=\frac{a_n\beta_p}{\beta_p+\beta_n}+O(d)>0.
 \end{gathered}
\]
Both components vanish at \(x_L\). By the Taylor expansions of \(b_i\) 
and \(\boldsymbol M\) near \(x_L\), we derive  
\[
 |[(\mathcal L_d-\Lambda_D)\boldsymbol F_d^L]_p(x_L+ds)|
       \leq\varepsilon_d\min\{s,1\},\quad
 |[(\mathcal L_d-\Lambda_D)\boldsymbol F_d^L]_n(x_L+ds)|
       \leq\varepsilon_d ds,\quad \varepsilon_d\to0,
\]
uniformly for
\(0<s\leq S_d^2\).
The potentially leading drift remainder in the \(n\)-equation is
absent because \(\Phi_n'=0\), and the matrix remainder is
\(o(ds)\). After division by the lower bounds above, every other
term is controlled either by \(d\) times a fixed polynomial in
\(S_d\), or by a coefficient modulus tending to zero. Hence, we obtain
\begin{equation}\label{comp:eq:left-artificial-relative-residual}
 |[(\mathcal L_d-\Lambda_D)\boldsymbol F_d^L]_i|
       \leq o(1)F_{d,i}^L.
\end{equation}

For the matching, the constant term \(dC_L\) in the small component
must be retained. Let \(\boldsymbol e_n\) be the \(n\)-th coordinate
vector and set
\(\boldsymbol U_d^{\mathrm{out}}=\boldsymbol U+dC_L\boldsymbol e_n\).
On \(S_d\leq s\leq S_d^2\),
\begin{equation}\label{comp:eq:left-artificial-overlap-mismatch}
 \begin{aligned}
 |\boldsymbol F_d^L-\boldsymbol U_d^{\mathrm{out}}|
   &\leq C[d^2(1+s)^2+(1+s)^2{\rm e}^{-\kappa s}],\\
 |\partial_x(\boldsymbol F_d^L-\boldsymbol U_d^{\mathrm{out}})|
   &\leq C[d(1+s)+d^{-1}(1+s)^2{\rm e}^{-\kappa s}].
 \end{aligned}
\end{equation}
Throughout this overlap, the \(p\)-components are bounded below by
\(c_1\), while the \(n\)-components are bounded below by \(c_1ds\).
We choose fixed numbers \(0<\rho_1<\rho_2<x_R-x_L\) and a smooth cutoff
\(\zeta_L\) satisfying \(\zeta_L=1\) on \([x_L,x_L+\rho_1]\) and \(\zeta_L=0\) on \([x_L+\rho_2,x_R]\).  
Both components of \(\boldsymbol U\) have a positive lower bound on
\([x_L+\rho_1,x_L+\rho_2]\). We reduce \(d\), if necessary, so that
\(dS_d^2<\rho_1\).
We choose the cutoff \(\chi_d(s)\) to be zero for \(s\leq S_d\),
equals one for \(s\geq S_d^2\), and satisfying
\begin{equation}\label{comp:eq:left-artificial-log-cutoff}
 |\chi_d'(s)|\leq\frac{C}{s\log S_d},
 \qquad
 |\chi_d''(s)|\leq\frac{C}{s^2\log S_d}.
\end{equation}
We define the complete left-end profile by
\begin{equation}\label{comp:eq:left-artificial-joined-profile}
 \boldsymbol F_d^{L,\mathrm{join}}(x)=
 \begin{cases}
 \boldsymbol F_d^L(x),
   &0\leq s\leq S_d,\\[1mm]
 (1-\chi_d(s))\boldsymbol F_d^L(x)
 +\chi_d(s)\boldsymbol U_d^{L,\mathrm{out}}(x),
   &S_d\leq s\leq S_d^2,\\[1mm]
 \boldsymbol U(x)+dC_L\zeta_L(x)\boldsymbol e_n,
   &x\geq x_L+dS_d^2,
 \end{cases}
 \qquad s=\frac{x-x_L}{d}.
\end{equation}
The definitions agree at \(s=S_d\) and \(s=S_d^2\). Indeed,
\(\zeta_L=1\) at \(x=x_L+dS_d^2\), so the third expression equals
\(\boldsymbol U_d^{L,\mathrm{out}}\) there. By choosing the cutoffs flat at the endpoints of their transition intervals, {we ensure} that \(\boldsymbol F_d^{L,\mathrm{join}}\) is piecewise \(C^2\) with
matching first derivatives.

We next estimate the residual on \(S_d\leq s\leq S_d^2\). {Set}
\(\boldsymbol D_d\coloneqq
\boldsymbol U_d^{L,\mathrm{out}}-\boldsymbol F_d^L\).
Dots below denote derivatives with respect to \(s\). 
{A direct calculation yields}
\[
\begin{aligned}
 [(\mathcal L_d-\Lambda_D)
   \boldsymbol F_d^{L,\mathrm{join}}]_i
={}&(1-\chi_d)
 [(\mathcal L_d-\Lambda_D)\boldsymbol F_d^L]_i
 +\chi_d
 [(\mathcal L_d-\Lambda_D)
   \boldsymbol U_d^{L,\mathrm{out}}]_i\\
&-2\dot\chi_d\,\partial_xD_{d,i}
 -\frac{\ddot\chi_d+b_i(x)\dot\chi_d}{d}D_{d,i}.
\end{aligned}
\]
The absolute residual of
\(\boldsymbol U_d^{L,\mathrm{out}}\) is \(O(d)\). Therefore,
on the overlap,
\[
 \frac{\left|
 [(\mathcal L_d-\Lambda_D)
   \boldsymbol U_d^{L,\mathrm{out}}]_p\right|}
 {U_{d,p}^{L,\mathrm{out}}}
 \leq Cd,
 \qquad
 \frac{\left|
 [(\mathcal L_d-\Lambda_D)
   \boldsymbol U_d^{L,\mathrm{out}}]_n\right|}
 {U_{d,n}^{L,\mathrm{out}}}
 \leq\frac{C}{s}\leq\frac{C}{S_d}.
\]
Using \eqref{comp:eq:left-artificial-overlap-mismatch},
\eqref{comp:eq:left-artificial-log-cutoff}, and the componentwise
lower bounds, we obtain
\[
\begin{aligned}
 &\max_{i=p,n}\sup_{S_d\leq s\leq S_d^2}
 \frac{\left|
 2\dot\chi_d\,\partial_xD_{d,i}
 +d^{-1}(\ddot\chi_d+b_i\dot\chi_d)D_{d,i}
 \right|}
 {F_{d,i}^{L,\mathrm{join}}}\leq
 \frac{C}{\log S_d}
 +\frac{C\mathrm e^{-\kappa S_d}}
        {d^2\log S_d}
 +o(1).
\end{aligned}
\]
Since \(S_d=c|\log d|\), the second term {vanishes when}
\(c\kappa>2\). {Combining} this estimate with
\eqref{comp:eq:left-artificial-relative-residual}, we deduce that
\[
 \max_{i=p,n}\sup_{S_d\leq s\leq S_d^2}
 \frac{\left|
 [(\mathcal L_d-\Lambda_D)
   \boldsymbol F_d^{L,\mathrm{join}}]_i\right|}
 {F_{d,i}^{L,\mathrm{join}}}
 \leq
 o(1)+\frac{C}{S_d}+\frac{C}{\log S_d}
 +\frac{C\mathrm{e}^{-\kappa S_d}}
        {d^2\log S_d}
 \to0.
\]

On \(x_L+dS_d^2\leq x\leq x_L+\rho_1\), the joined profile coincides with
\(\boldsymbol U+dC_L\boldsymbol e_n\). The local expansion of
\(U_n\) at \(x_L\) shows that its relative residual is
\(O(S_d^{-2})+O(d)\). On
\([x_L+\rho_1,x_L+\rho_2]\), both components of
\(\boldsymbol U\) are bounded below, and the cutoff involving
\(\zeta_L\) contributes only \(O(d)\) to the relative residual.
Consequently, we obtain
\begin{equation}\label{comp:eq:left-artificial-joined-residual}
 \max_{i=p,n}\sup_{x_L<x<x_L+\rho_2}
 \frac{\left|
 [(\mathcal L_d-\Lambda_D)
   \boldsymbol F_d^{L,\mathrm{join}}]_i\right|}
 {F_{d,i}^{L,\mathrm{join}}}
 \to 0.
\end{equation}
Moreover, we observe that
\begin{equation}\label{comp:eq:left-artificial-joined-matching}
 \begin{gathered}
 \boldsymbol F_d^{L,\mathrm{join}}(x_L)=\boldsymbol0,
 \qquad
 \partial_xF_{d,i}^{L,\mathrm{join}}(x_L)>0,
 \quad i=p,n,\\
 \boldsymbol F_d^{L,\mathrm{join}}\gg\boldsymbol0
 \quad\text{on }(x_L,x_L+\rho_2],
 \qquad
 \boldsymbol F_d^{L,\mathrm{join}}=\boldsymbol U
 \quad\text{for }x\geq x_L+\rho_2.
 \end{gathered}
\end{equation}
Reflecting about \(x_R\) and interchanging \(p\) and \(n\) produces a right-end profile
\(\boldsymbol F_d^{R,\mathrm{join}}\). It coincides with
\(\boldsymbol U\) outside a fixed right-end neighborhood, has
relative residual \(o(1)\), and satisfies
\[
 \boldsymbol F_d^{R,\mathrm{join}}(x_R)=\boldsymbol0,
 \qquad
 \partial_xF_{d,i}^{R,\mathrm{join}}(x_R)<0,
 \quad i=p,n.
\]

\medskip
\noindent\emph{Step 2.}
At each unchanged endpoint,
\(\Lambda_D<\lambda_{\partial,G}\). Since the constructions in
Lemmas~\ref{reg:lem:repelling-module} and
\ref{reg:lem:attracting-domain-modules} rely solely on the local
transport equation, the endpoint condition, and this strict
inequality, they can apply directly to \(\boldsymbol U\).
 Choosing the endpoint neighborhoods to be mutually disjoint, we use
\(\boldsymbol F_d^{L,\mathrm{join}}\) or
\(\boldsymbol F_d^{R,\mathrm{join}}\) {at the respective artificial endpoints, and} the corresponding module from the two lemmas above {at every unchanged endpoint}. {Since} each of
these functions {coincides with} \(\boldsymbol U\) outside its endpoint
neighborhood, we set \(\boldsymbol F_d=\boldsymbol U\) on the
remaining part of \(D\). This construction yields a positive function
\(\boldsymbol F_d\) satisfying
\begin{equation}\label{comp:eq:finite-section-neutral-core}
 |[(\mathcal L_d-\Lambda_D)\boldsymbol F_d]_i|
 \leq\varepsilon_dF_{d,i},
 \qquad i=p,n,
 \qquad \varepsilon_d\to0.
\end{equation}
Both components of \(\boldsymbol F_d\) vanish at every artificial
endpoint and have strictly positive inward derivatives there. At a
domain endpoint, \(\boldsymbol F_d\) satisfies the exact Neumann
condition.

Choosing \(d\) sufficiently small that \(\varepsilon_d<\tau/4\), we obtain
\begin{equation}\label{comp:eq:finite-section-strict-core}
 [\mathcal L_d-(\Lambda_D+\tau)]\boldsymbol F_d
 \leq-\frac{3\tau}{4}\boldsymbol F_d.
\end{equation}
At each unchanged interior endpoint \(\xi\), we take the slow exterior
tail \(\boldsymbol G_d^{\xi,-}\), the separated cutoffs
\(\chi_{F,d}\), \(\chi_{\xi,d}\), and the matching constant
\(A_{\xi,d}>0\) constructed in the proof of
Lemma~\ref{reg:lem:local-realization} at the spectral level \(\Lambda_D+\tau\). We define
\begin{equation}\label{comp:eq:finite-section-Q-construction}
 \boldsymbol q_d
 \coloneqq
 \chi_{F,d}\boldsymbol F_d
 -\sum_\xi
 A_{\xi,d}\chi_{\xi,d}\boldsymbol G_d^{\xi,-},
 \qquad
 \boldsymbol Q_d\coloneqq(\boldsymbol q_d)_+,
\end{equation}
where the sum is taken only over the unchanged interior endpoints and
the positive part is componentwise. The cutoffs are chosen so that
\(\boldsymbol q_d=\boldsymbol F_d\gg\boldsymbol0\) on \(D\), and
they do not modify either artificial layer. 
In view of \eqref{comp:eq:finite-section-strict-core},
\eqref{reg:eq:slow-tail-strong}, and
\eqref{reg:eq:exponential-separation}, the cutoff commutators are
absorbed by the negative residual in
\eqref{comp:eq:finite-section-strict-core} and the positive residual
of the subtracted tails. Hence, we have
$
 [\mathcal L_d-(\Lambda_D+\tau)]\boldsymbol q_d
 \leq\boldsymbol0.
$ 
By the standard positive-part property for cooperative systems, we deduce
\begin{equation}\label{comp:eq:finite-section-Q-inequality}
 \mathcal L_d\boldsymbol Q_d
 \leq(\Lambda_D+\tau)\boldsymbol Q_d.
\end{equation}
Thus, \(\boldsymbol Q_d\) is nonzero and positive in \(D\), while its
support lies in the prescribed relative neighborhood of
\(\overline D\). The proof is complete.
\end{proof}

We shall also require cell supersolutions whose asymptotic behavior at a common zero
can be matched with that of the neighboring Ornstein-Uhlenbeck profile.

\begin{lemma}\label{comp:lem:prescribed-cell-powers}
Fix \(\nu<\Lambda_G^{\mathrm{tr}}\). For each common endpoint \(z\),
choose \(c_z>0\) and \(\vartheta_z\in(0,1)\) such that
\(\rho_z(\vartheta_z)>\nu\). Then there exists
\(\boldsymbol U_G\in C^2(G;(0,\infty)^2)\) satisfying
\begin{equation}\label{comp:eq:strict-cell-super}
 \mathscr T\boldsymbol U_G\geq\nu\boldsymbol U_G \quad \text{ in } G.
\end{equation}
Near an unchanged endpoint, \(\boldsymbol U_G\) solves the transport
equation at level \(\nu\) and satisfies the prescribed admissibility condition. At every
common endpoint, there holds
\begin{equation}\label{comp:eq:prescribed-cell-tail}
 \partial_x^j\!\left[
 U_{G,i}(x)-c_z|x-z|^{-\vartheta_z}v_{z,i}(\vartheta_z)\right]
   =o(|x-z|^{-\vartheta_z-j}),
 \qquad j=0,1,2,\quad i=1,2.
\end{equation}
\end{lemma}

\begin{proof}
Write \(G=(a,b)\). For each
\(z\in\partial G\cap\mathscr Z_K\), let
\(\boldsymbol v_z(\vartheta_z)\gg\boldsymbol0\) be the normalized
positive eigenvector satisfying \(\bigl(\boldsymbol M(z)+\vartheta_z\boldsymbol B_z\bigr)
 \boldsymbol v_z(\vartheta_z)
 =
 \rho_z(\vartheta_z)\boldsymbol v_z(\vartheta_z)\) and \(v_{z,1}(\vartheta_z)+v_{z,2}(\vartheta_z)=1\).
A direct calculation yields
\begin{equation}\label{comp:eq:cell-model-power}
 (\mathscr T-\nu)
 \bigl[|x-z|^{-\vartheta_z}
       \boldsymbol v_z(\vartheta_z)\bigr]
 =
 |x-z|^{-\vartheta_z}
 \left[
  \bigl(\rho_z(\vartheta_z)-\nu\bigr)
       \boldsymbol v_z(\vartheta_z)
  +O(|x-z|)
 \right].
\end{equation}
In view of \(\rho_z(\vartheta_z)>\nu\) and
\(\boldsymbol v_z(\vartheta_z)\gg\boldsymbol0\), we may select
\(0<\delta<(b-a)/4\) such that, for every
\(z\in\partial G\cap\mathscr Z_K\) and \(i=1,2\),
\[
 \left[
  (\mathscr T-\nu)
  \bigl(|x-z|^{-\vartheta_z}
        \boldsymbol v_z(\vartheta_z)\bigr)
 \right]_i
 \geq
 \frac{\rho_z(\vartheta_z)-\nu}{2}
 |x-z|^{-\vartheta_z}v_{z,i}(\vartheta_z)>0\ \text{ for }0<|x-z|<2\delta.
\]
Fix \(\chi\in C^\infty([0,\infty);[0,1])\) such that
\(\chi=1\) on \([0,1]\) and \(\chi=0\) on \([2,\infty)\). We define
\[
 \boldsymbol P(x)=
 \sum_{z\in\partial G\cap\mathscr Z_K}
 c_z\chi\left(\frac{|x-z|}{\delta}\right)
 |x-z|^{-\vartheta_z}\boldsymbol v_z(\vartheta_z).
\]
The summands have disjoint supports. For the summand associated with
\(z\), we have
\[
\begin{aligned}
 &(\mathscr T-\nu)
 \left[
  c_z\chi\left(\frac{|x-z|}{\delta}\right)
  |x-z|^{-\vartheta_z}\boldsymbol v_z(\vartheta_z)
 \right]\\
 ={}&
 c_z\chi\left(\frac{|x-z|}{\delta}\right)
 (\mathscr T-\nu)
 \left[
  |x-z|^{-\vartheta_z}\boldsymbol v_z(\vartheta_z)
 \right] \\
 &-
 c_z\operatorname{diag}(b_1,b_2)
 \partial_x\!\left[
  \chi\left(\frac{|x-z|}{\delta}\right)
 \right]
 |x-z|^{-\vartheta_z}\boldsymbol v_z(\vartheta_z).
\end{aligned}
\]
The second term vanishes for \(|x-z|\leq\delta\), where the first
term is componentwise positive, while both terms vanish for
\(|x-z|\geq2\delta\).  Hence, for \(i=1,2\),
\[
 \operatorname{supp}
 \left(-[(\mathscr T-\nu)\boldsymbol P]_i\right)_+
 \subset
 \bigcup_{z\in\partial G\cap\mathscr Z_K}
 \{x\in G:\delta\leq|x-z|\leq2\delta\}
 \Subset G.
\]
We may therefore choose
\(\boldsymbol h\in C_c^\infty(G;\mathbb R^2)\), nonnegative and
positive in both components on some interior interval, such that
\[
 h_i(x)\geq
 \left(-[(\mathscr T-\nu)\boldsymbol P]_i(x)\right)_+,
 \qquad x\in G,\quad i=1,2.
\]
Thus
\((\mathscr T-\nu)\boldsymbol P+\boldsymbol h\geq\boldsymbol0\).
Both \(\boldsymbol P\) and \(\boldsymbol h\) vanish near every
unchanged endpoint.

For every sufficiently large finite section \(D\), the positive
resolvent introduced before
Lemma~\ref{comp:lem:cell-exhaustion} provides a solution of
\begin{equation}\label{comp:eq:cell-resolvent-correction}
 (\mathscr T-\nu)\boldsymbol V_D=\boldsymbol h,
\end{equation}
with zero incoming data at the artificial endpoints and the actual
endpoint conditions at level \(\nu\). By the strict cooperation, we have
\(\boldsymbol V_D\gg0\) in \(D\).
Let \(\boldsymbol W=\boldsymbol W_G\) be the limiting eigenfunction in
Lemma \ref{comp:lem:cell-exhaustion}. In view of
$
 (\mathscr T-\nu)\boldsymbol W
 = (\Lambda_G^{\mathrm{tr}}-\nu)\boldsymbol W,
$ a fixed multiple \(C\boldsymbol W\) has a residual strictly larger than
\(\boldsymbol h\). We claim that
\(\boldsymbol V_D\leq C\boldsymbol W\), with \(C\) independent of
\(D\). To see this, compare
\(\boldsymbol Z=C\boldsymbol W-s\boldsymbol V_D\) as \(s\) increases
from \(0\) to \(1\). Its differential residual is strictly positive,
and its incoming traces at the artificial endpoints are positive. At
an unchanged left attracting or domain endpoint, the incoming boundary
residual is
\[
 (M_{nn}-\nu)Z_n+M_{np}Z_p
   =C(\Lambda_G^{\mathrm{tr}}-\nu)W_n>0.
\]
Near a repelling endpoint the source vanishes, while since \(q(\nu)>q(\Lambda_G^{\mathrm{tr}})\),
\[
 \frac{V_{D,i}}{W_i}
   =O\!\left(t^{q(\nu)-q(\Lambda_G^{\mathrm{tr}})}\right)\to0,
   \qquad i=1,2.
\]
The right endpoint follows analogously. Hence any first contact
must occur away from the repelling endpoints, where it is excluded by
\eqref{comp:eq:cell-contact-test}. The claimed bound follows.

A sharper estimate is needed near a common endpoint. Choose
\(0<\vartheta_z^-<\vartheta_z\), sufficiently close to
\(\vartheta_z\), so that \(\rho_z(\vartheta_z^-)>\nu\). By
\eqref{comp:eq:cell-model-power},
\(\boldsymbol S_z=t^{-\vartheta_z^-}\boldsymbol v_z(\vartheta_z^-)\)
is a strict supersolution in a smaller neighborhood of \(z\), chosen
disjoint from \(\operatorname{supp}\boldsymbol h\). {The estimate}
\(\boldsymbol V_D\leq C\boldsymbol W\) controls the traces at the
fixed endpoint of this neighborhood lying inside \(G\). A fixed
multiple of \(\boldsymbol S_z\) dominates these traces, while its
incoming trace at the artificial endpoint remains positive. Applying
\eqref{comp:eq:cell-contact-test} to the resulting comparison gives
\begin{equation}\label{comp:eq:common-end-correction-bound}
 0< V_{D,i}(x)\leq C|x-z|^{-\vartheta_z^-},
\end{equation}
with \(C\) independent of the artificial endpoint.

In view of the interior estimate and
\eqref{comp:eq:cell-resolvent-correction}, a subsequence converges
in \(C^1_{\mathrm{loc}}(G)\) to a solution \(\boldsymbol V\) of
\((\mathscr T-\nu)\boldsymbol V=\boldsymbol h\). The positivity of the
characteristic formulas implies \(\boldsymbol V\gg0\). Near an
unchanged endpoint the equation is homogeneous at level \(\nu\), so
the continuous or weighted endpoint representations used in the proof
of Lemma~\ref{comp:lem:cell-exhaustion} pass the required
admissibility to the limit. Near a common endpoint,
\eqref{comp:eq:common-end-correction-bound} and
\(|b_i(x)|\asymp t\) yield
\(|V_i'|=O(t^{-\vartheta_z^--1})\) through the transport equation.
Differentiating once more yields
\(|V_i''|=O(t^{-\vartheta_z^--2})\).

Set \(\boldsymbol U_G:=\boldsymbol P+\boldsymbol V\). The selection of
\(\boldsymbol h\) yields \eqref{comp:eq:strict-cell-super}, and the
last two estimates show that the correction is of lower order than
each prescribed power. It therefore changes neither its coefficient
nor its exponent, which establishes
\eqref{comp:eq:prescribed-cell-tail}.
\end{proof}

\subsection{The value of a composite interval class}

We now splice together the cell profiles and the common-zero profiles constructed above at the value \(\Lambda_K^{\mathrm{comp}}\) given by
\eqref{eq:main-composite-value}.

\begin{theorem}\label{comp:thm:local-realization}
Every composite interval class $K$ has local effective value
$\Lambda_K^{\mathrm{comp}}$ in the sense of
Definition {\rm\ref{def:local-effective-value}}.
\end{theorem}

\begin{proof}
Fix \(\eta>0\) and a sufficiently small relative neighborhood \(U\)
of \(K\) containing no other static class, and set
\(\Lambda=\Lambda_K^{\mathrm{comp}}\).

We begin with the compactly supported subsolution required in
Definition~\ref{def:local-effective-value}(ii). Suppose that
\(\Lambda=\Lambda_G^{\mathrm{tr}}\) for some cell \(G\). Choose a
finite interval \(D\) such that \(\Lambda_D<\Lambda+\eta/2\).
Applying Lemma~\ref{comp:lem:finite-section-viscous-subsolution} with
\(\tau=\eta/2\) and localizing its support within a fixed relative
neighborhood compactly contained in \(U\), we obtain the desired
subsolution.

Suppose instead that \(\Lambda=\omega_z\) at a common zero \(z\).
Select \(R\) such that \(\omega_{z,R}<\omega_z+\eta/2\). Under the
rescaling \(x=z+\sqrt d\,y\), the coefficients of the Dirichlet
problem on \((z-R\sqrt d,z+R\sqrt d)\) converge on \([-R,R]\) to
those of the truncated local problem. 
Hence, as in the proof of Lemma~\ref{comp:lem:mixed-OU}, its principal eigenvalue converges to
\(\omega_{z,R}\). For small \(d\), extending the corresponding positive
eigenfunction by zero yields a compactly supported subsolution in the standard sense,
whose support is contained in \(U\). If the minimum is attained at more
than one cell or common zero, any one of these constructions may be
used.

It remains to construct the positive supersolution in
Definition~\ref{def:local-effective-value}(i). Set
\(\nu=\Lambda-\eta/2\). Since \(\Lambda\leq\omega_z\) at every
common zero, for each \(z\) we can select
\(\vartheta_z\in(0,1)\) such that \(\rho_z(\vartheta_z)>\nu\).
This remains possible when the maximum is attained at \(0\) or
\(1\). We apply Lemma~\ref{comp:lem:prescribed-cell-powers} to every
cell with \(c_z=1\), and apply Lemma~\ref{comp:lem:mixed-OU} at every
common zero, both taking the spectral level as \(\nu\).
The resulting fields satisfy
\(\mathscr T\boldsymbol U_G\geq\nu\boldsymbol U_G\) and
\(\mathscr A_z\boldsymbol\Psi_z\geq\nu\boldsymbol\Psi_z\).
Near an unchanged endpoint, \(\mathscr T\boldsymbol U_G=\nu\boldsymbol U_G\), and the actual admissibility condition holds. Moreover, the two cell profiles
adjacent to \(z\) have the same leading asymptotic
\(|x-z|^{-\vartheta_z}\boldsymbol v_z(\vartheta_z)\).

We rescale the node profile by
\begin{equation}\label{comp:eq:scaled-node-super}
 \boldsymbol\Psi_{z,d}(x)
   =d^{-\vartheta_z/2}
       \boldsymbol\Psi_z\!\left(\frac{x-z}{\sqrt d}\right).
\end{equation}
We select \(r_0>0\), independent of \(d\), so that the
\(r_0\)-neighborhoods of the common zeros are disjoint, avoid
\(\alpha,\beta\), and are contained within the neighborhoods where
\eqref{comp:eq:prescribed-cell-tail} holds. By expanding the coefficients in Taylor series and applying \eqref{comp:eq:OU-power-super}, we obtain
\[
 \left|[\mathcal L_d\boldsymbol\Psi_{z,d}]_i
 -d^{-\vartheta_z/2}
       [\mathscr A_z\boldsymbol\Psi_z]_i
            \!\left(\frac{x-z}{\sqrt d}\right)\right|
 \leq C(r_0+\sqrt d)\Psi_{z,d,i}
 \quad (|x-z|\leq r_0).
\]
In the \(y\)-variable, the drift remainder is of order \(O(\sqrt d\,|y|^2)\). For \(|y|\geq1\), the required weighted
derivative bounds follow from \eqref{comp:eq:OU-power-super}, whereas the 
positivity and \(C^2\)-regularity ensure the necessary control over \(|y|\leq1\). Reducing \(r_0\) and then choosing \(d\) sufficiently small ensures that the relative
error is bounded by \(\eta/8\). Hence, we obtain
\(\mathcal L_d\boldsymbol\Psi_{z,d}
\geq(\nu-\eta/8)\boldsymbol\Psi_{z,d}\).

Set \(h_d=d^{1/4}\) and \(L_d=\log(r_0/h_d)\).
On the portions of a cell outside the \(h_d\)-neighborhoods of the
common zeros and outside the unchanged endpoint neighborhoods, it follows from
\eqref{comp:eq:prescribed-cell-tail} that
\(d|U_{G,i}''|\leq C(d/h_d^2)U_{G,i}=o(1)U_{G,i}\).
Hence,
\(\mathcal L_d\boldsymbol U_G\geq(\nu-\eta/8)\boldsymbol U_G\)
for sufficiently small \(d\).
On each overlap \(h_d\leq t=|x-z|\leq r_0\), we select a scalar cutoff
\(\chi_{z,d}\) that is constant near both ends, equals zero at the inner
edge and one at the outer edge, and satisfies
\(t|\chi_{z,d}'|+t^2|\chi_{z,d}''|\leq C/L_d\).
In view of \(h_d/\sqrt d\to\infty\), we have
\[
 U_{G,i}\asymp\Psi_{z,d,i}\asymp t^{-\vartheta_z},
 \qquad
 |U_{G,i}'|+|\Psi_{z,d,i}'|\leq Ct^{-\vartheta_z-1},
\]
where the bounds are uniform for small \(d\).
For
\(\boldsymbol F_{z,G,d}
=\chi_{z,d}\boldsymbol U_G+(1-\chi_{z,d})\boldsymbol\Psi_{z,d}\),
the matrix term commutes with the scalar cutoff. The gluing error is
\begin{equation}\label{comp:eq:composite-gluing-error}
 E_{z,d,i}
 =-d\chi_{z,d}''(U_{G,i}-\Psi_{z,d,i})
  -2d\chi_{z,d}'(U_{G,i}'-\Psi_{z,d,i}')
  -b_i\chi_{z,d}'(U_{G,i}-\Psi_{z,d,i}).
\end{equation}
Since \(|b_i(x)|\leq Ct\), {we deduce that}
\[
 |E_{z,d,i}|
 \leq C\left(\frac1{L_d}+\frac d{h_d^2L_d}\right)F_{z,G,d,i}
 =o(1)F_{z,G,d,i}.
\]
We perform this gluing at every common zero, and denote the resulting
profile by \(\boldsymbol F_d\). For sufficiently small \(d\), it satisfies
\(\mathcal L_d\boldsymbol F_d\geq(\nu-\eta/4)\boldsymbol F_d\)
in the interior of \(K\), apart from the neighborhoods reserved for
unchanged endpoints. Since each cutoff is constant near the ends of
its overlap, all joins are \(C^2\).

{We now consider} an unchanged endpoint of a cell \(G\). By the definition
of \(\Lambda\) and the spectral window in
\eqref{comp:eq:cell-uniform-spectral-window},
\(\nu<\Lambda\leq\Lambda_G^{\mathrm{tr}}<\lambda_{\partial,G}\).
In a neighborhood of this endpoint, the transport profile solves the
exact equation at level \(\nu\). We may therefore insert the {local models}
from Lemmas~\ref{reg:lem:repelling-module} and
\ref{reg:lem:attracting-domain-modules}, as explained in
Lemma~\ref{comp:lem:finite-section-viscous-subsolution}.
{Taking the endpoint neighborhoods to be} disjoint, we obtain
\[
 (\mathcal L_d-(\Lambda-\eta))\boldsymbol F_d
       \geq\frac{\eta}{4}\boldsymbol F_d
\]
on a neighborhood of \(K\), with exact Neumann data at every domain
endpoint.

To extend this profile across the artificial boundaries, we use the
addition construction from Steps 1-2 in the proof of
Lemma~\ref{reg:lem:local-realization}:
\[
 \boldsymbol\Psi_{d,\eta}
   =\chi_{F,d}\boldsymbol F_d
       +\sum_{\xi\in\{\alpha,\beta\}\cap(\ell,r)}
         A_{\xi,d}\chi_{\xi,d}\boldsymbol G_d^{\xi,-}.
\]
Since the verification in that lemma {relies} only on the endpoint {local models}, their positive
differential residual, and the slow-tail formulas, it applies here
without {modification}. In particular,
\eqref{reg:eq:exponential-separation} makes each
cutoff error exponentially smaller than the profile dominating its
support. Formula~\eqref{reg:eq:slow-tail} absorbs the remaining exterior
errors at the fixed level \(\Lambda-\eta\).
We have therefore obtained a relative interval \(V\), satisfying
\(K\subset V\subset\overline V\subset U\), on which
\begin{equation}\label{comp:eq:composite-final-super}
 \mathcal L_d\boldsymbol\Psi_{d,\eta}
      \geq(\Lambda-\eta)\boldsymbol\Psi_{d,\eta}.
\end{equation}
The function is positive, has no derivative jumps, and satisfies
the Neumann condition at every domain endpoint.

It remains to check the phase normalization. Away from their
endpoints, the cell profiles do not depend on \(d\). On the remaining
parts of \(K\), the polynomial
upper and lower bounds follow from \eqref{comp:eq:scaled-node-super}, the
common-zero tails, and the repelling-endpoint estimates
\eqref{reg:eq:repelling-inner-field} and
\eqref{reg:eq:repelling-common-match}. The profiles near attracting and domain
endpoints remain positive and bounded. Upon multiplication by a
common positive scalar, we therefore have
\(d^C\leq\Psi_{d,\eta,i}\leq d^{-C}\) on \(K\), and
\(\max_i\|-d\log\Psi_{d,\eta,i}\|_{C(K)}\to0\).
At each fixed artificial boundary, the slow-tail formulas yield
\[
 -d\log\Psi_{d,\eta,i}(x)=
 \begin{cases}
  \theta d_H(x,K)+o(1),&\text{adjacent to a repelling endpoint},\\
  -\theta P_\xi(x)+o(1),&\text{adjacent to an attracting endpoint}.
 \end{cases}
\]
In the first case \(d_H(x,K)>0\). In the second,
\(P_\xi(x)>0\) and \(d_H(x,K)=0\). As in
Lemma~\ref{reg:lem:local-realization}, the matching factors satisfy
\(-d\log A_{\xi,d}=o(1)\). Since only finitely many boundary points
are involved, these estimates establish the uniform strict gap
\eqref{eq:local-phase-gap}. When both endpoints of \(K\) are domain
endpoints, there is no artificial boundary to consider. This verifies
Definition~\ref{def:local-effective-value}(i) and completes the proof.
\end{proof}

\bigskip

\noindent{\bf Acknowledgments.} 
We would like to thank Prof. Yuan Lou for insightful discussions on this work, and Prof. King-Yeung (Adrian) Lam 
 for
suggesting the interpretation of regular interval classes in terms
of switching between the component drifts.
The author is partially supported by the NSFC grants (12201041, 12631017) and 
Beijing Institute of Technology Research Fund Program for Young Scholars. 

\bigskip
\noindent{\bf Declaration of AI usage.}
The AI tool used in this paper is DeepSeek, which was used to polish the grammar, phrasing, and readability of the manuscript.

 \bigskip
\noindent{\bf Declarations.}
\begin{itemize}

    \item Conflict of interest: The  author  states that there is no conflict
of interest.

\item Availability of data: Data sharing not applicable to this article as no datasets were
generated or analysed during the current study.
\end{itemize}   

\baselineskip 18pt
\renewcommand{\baselinestretch}{1.2}

\end{document}